\documentclass[12pt,a4paper,oneside]{article}
\usepackage{amsfonts}
\usepackage[style=alphabetic, backend=bibtex, doi=false, url=false]{biblatex}

\let\digamma\relax

\usepackage{amsmath,amssymb,amsthm}
\usepackage{mathrsfs}
\usepackage{geometry}
\usepackage{graphicx}
\usepackage{titlesec}
\usepackage{tensor}
\usepackage{rotating}
\usepackage{mathtools}
\usepackage{subcaption}
\usepackage{multicol}
\usepackage{imakeidx}
\usepackage[nottoc]{tocbibind}
\usepackage[all]{xy}
\usepackage{color}
\usepackage[colorlinks=true, citecolor=red, linkcolor=blue]{hyperref}
\usepackage{enumerate}
\usepackage{enumitem}
\usepackage{upgreek}
\usepackage{bm}
\usepackage[boxsize=3em]{ytableau}

\setlist[enumerate]{align=left}

\def\CC{\mathbb{C}}

\def\RR{\mathbb{R}}

\def\ZZ{\mathbb{Z}}

\def\mA{\mathcal{A}}
\def\mB{\mathcal{B}}
\def\mC{\mathcal{C}}

\def\mE{\mathcal{E}}

\def\mI{\mathcal{I}}

\def\mP{\mathcal{P}}
\def\mR{\mathcal{R}}
\def\mS{\mathcal{S}}
\def\mT{\mathcal{T}}

\def\mZ{\mathcal{Z}}

\def\GL{\mathrm{GL}}
\def\Mp{\mathrm{Mp}}
\def\Or{\mathrm{O}}

\def\SL{\mathrm{SL}}
\def\Sp{\mathrm{Sp}}

\def\SO{\mathrm{SO}}
\def\SL{\mathrm{SL}}

\def\Ato{\mathrm{Ato}}

\def\bdd{\mathrm{bdd}}

\def\cusp{\mathrm{cusp}}

\def\Frob{\mathrm{Frob}}

\def\gp{\mathrm{gp}}

\def\Id{\mathrm{Id}}

\def\Jac{\mathrm{Jac}}
\def\JH{\mathrm{JH}}
\def\Jord{\mathrm{Jord}}

\def\rL{\mathrm{L}}

\def\M{\mathrm{M}}

\def\MW{\mathrm{MW}}

\def\np{\mathrm{np}}

\def\op{\mathrm{op}}

\def\Supp{\mathrm{Supp}}

\def\spec{\mathrm{spec}}

\def\soc{\mathrm{soc}}

\def\Sgn{\mathrm{Sgn}}

\def\triv{\mathrm{triv}}
\def\temp{\mathrm{temp}}

\def\unit{\mathrm{unit}}

\def\W{\mathrm{W}}

\def\Z{\mathrm{Z}}

\def\para{/\kern -0.3em /}

\numberwithin{equation}{section}

\theoremstyle{plain}
\newtheorem{theorem}{Theorem}[section]
\newtheorem{proposition}[theorem]{Proposition}
\newtheorem{corollary}[theorem]{Corollary}
\newtheorem{lemma}[theorem]{Lemma}

\newtheorem{conjecture}[theorem]{Conjecture}

\theoremstyle{definition}
\newtheorem{definition}[theorem]{Definition}
\newtheorem{remark}[theorem]{Remark}

\makeindex[title={Index}]

\title{Construction of Local Arthur Packets for Metaplectic Groups and the Adams Conjecture}
\author{Jiahe Chen}
\date{}

\begin{document}
	
	\maketitle
	
	\begin{abstract}
		In this article, we explicitly construct local Arthur packets for metaplectic groups over non-Archimedean local fields of characteristic zero. Our construction is a generalization of Atobe's construction of local Arthur packets for classical groups. As a result, we prove that the local Arthur packets are multiplicity free. Moreover, we generalize M\oe glin's earlier work about the Adams conjecture to metaplectic groups. 
	\end{abstract}
	
	{\scriptsize
		\begin{tabular}{ll}
			\textbf{MSC (2020)} & Primary 22E50; Secondary 11F27, 11F70\\
			\textbf{Keywords} & metaplectic group, local Arthur packet, endoscopy, theta correspondence, Adams conjecture
	\end{tabular}}
	
	\setcounter{tocdepth}{2}
	\tableofcontents
	
	\section{Introduction}
	
	\subsection*{Overview}
	
	\par Let $F$ be a non-Archimedean local field of characteristic zero, and let $W$ be a symplectic space over $F$. The metaplectic group $\Mp(W)$ is a central extension of $\Sp(W)$. 
	
	\par In \cite{Li2024_arthurpacketsmetaplecticgroups}, W-W. Li defined Arthur packets for metaplectic groups. In his definition, the Arthur packet $\Pi_\psi$ attached to an Arthur parameter $\psi$ is a multi-set of unitary genuine irreducible representations, which can be characterized by endoscopic character relations. However, his proof does not provide concrete information about local Arthur packets. In particular, it is still unknown whether the packets are multiplicity free in general. The goal of this article is to provide a more explicit description of Arthur packets for metaplectic groups.
	
	\par For classical groups, M\oe glin constructed Arthur packets explicitly in a series of works, e.g., \cite{Mœglin2006_arthur_packets, Mœglin2006_Aubert, Mœeglin2009_paquets} and deduced that Arthur packets are multiplicity free in \cite{Mœglin2011_mult_free}. Later, Atobe gave a reformulation of M\oe glin's construction in \cite{Atobe2022_A-packet}, which simplified M\oe glin's construction.
	
	\par M\oe glin's construction can be described by the following diagram (see \cite{Xu2017_Moeglin} for details):
	$$\left\{\text{discrete series}\right\}\Rightarrow\left\{\text{elementary}\right\}\Rightarrow\left\{\text{DDR}\right\}\Rightarrow\left\{\text{of good parity}\right\}\Rightarrow\left\{\text{general}\right\}.$$
	Atobe observed that, by restricting to non-negative DDR packets, it is possible to avoid the concept of elementary packets. That is, Arthur packets can be constructed by the following method:
	$$\left\{\text{discrete series}\right\}\Rightarrow\left\{\text{non-negative DDR}\right\}\Rightarrow\left\{\text{of good parity}\right\}\Rightarrow\left\{\text{general}\right\}.$$
	
	\par Unfortunately, M\oe glin's construction of elementary packets can not be applied directly to the metaplectic groups. For example, let $r(a)$ be the unique $a$-dimensional irreducible algebraic representation of $\SL(2,\CC)$. We consider $\psi_1=\triv\otimes r(2)\otimes r(1)$ and $\psi_2=\triv\otimes r(1)\otimes r(2)$, with $\varepsilon_{i,\pm}\in\mS_{\psi_i}^\vee$ such that $\varepsilon_{1,\pm}(\rho,2,1)=\varepsilon_{2,\pm}(\rho,1,2)=\pm1$. It can be proved that $\pi(\psi_1,\varepsilon_{1,+})$ is cuspidal (see Theorem \ref{theorem: parameter of cuspidal representation} below). If M\oe glin's construction still holds in the metaplectic case (i.e., if \cite[Theorem 6.18]{Xu2017_Moeglin} is true for metaplectic groups), we would have $\pi(\psi_1,\varepsilon_{1,+})=\pi(\psi_2,\varepsilon_{2,+})$. However, it can be deduced from Lemma \ref{lemma: partial Jacquet module of elementary A parameter} that $\pi(\psi_2,\varepsilon_{2,+})$ is not cuspidal. This leads to a contradiction (according to our construction of Arthur packets in Theorem \ref{theorem: A-packets can be constructed via extended multi-segments}, it can be deduced that $\pi(\psi_1,\varepsilon_{1,+})=\pi(\psi_2,\varepsilon_{2,-})$. Thus, to resolve the problem, one possible way is to modify the character $\varepsilon_\psi^{\M/\MW}$ in \cite[Theorem 6.18]{Xu2017_Moeglin}).
	
	\par To avoid the difficulties in elementary packets, we will follow Atobe's method to construct Arthur packets for metaplectic groups. Our main results are the following theorems:
	\begin{theorem}(Theorem \ref{theorem: A-packets can be constructed via extended multi-segments})
		Let $\psi$ be an Arthur parameter of good parity. For $\varepsilon\in\mS_\psi^\vee$, we have $\pi_\Ato(\psi,\varepsilon)=\bigoplus\pi(\mE)$, where $\mE$ runs over all equivalence classes of extended multi-segments with $(\psi,\varepsilon)=(\psi_\mE,\varepsilon_\mE)$ (see \S \ref{subsection: extended multi-segments} for the definition of extended multi-segments for metaplectic groups). Here, the subscript ``$\Ato$" means Atobe's normalization (see \S \ref{section:Atobe's normalization}).
	\end{theorem}
	
	\begin{theorem}(Proposition \ref{prop: reduction to good parity})
		Let $\psi$ be a general Arthur parameter with decomposition $\psi=\psi_{\np}^\vee\oplus\psi_{\gp}\oplus\psi_{\np}^\vee$, where $\psi_{\gp}$ is the good parity part of $\psi$. Then, for $\pi\in\Pi_\psi$, the parabolic induction $\tau_{\psi_\np}\rtimes\pi$ is irreducible, and we have $\Pi_\psi=\left\{\tau_{\psi_\np}\rtimes\pi:\pi\in\Pi_{\psi_\gp}\right\}$.
	\end{theorem}
	
	\begin{theorem}(Corollary \ref{coro: A-packet is multiplicity free})
		Let $\psi$ be an Arthur parameter, then $\Pi_\psi$ is multiplicity free.
	\end{theorem}
	
	\par Let $V$ be an odd-dimensional split vector space of discriminant $1$ over $F$. There exists a theta correspondence map $\theta_{V,W}$ from the set of equivalence classes of irreducible genuine representations of $\Mp(W)$ to the set of irreducible representations of $\Or(V)$. Adams conjectured the following:
	
	\begin{conjecture}\label{conjecture: the adams conjecture}
		Let $\pi\in\Pi_\psi$ with $\psi$ an Arthur parameter. If $\theta_{V,W}(\pi)\neq0$, then $\theta_{V,W}(\pi)\in\Pi_{\psi_\alpha}$, where $\alpha=\dim(V)-\dim(W)-1$ and $$\psi_\alpha=\psi\oplus\triv\otimes r(1)\otimes r(\alpha).$$
	\end{conjecture}
	
	\par As a consequence of M\oe glin's construction of Arthur packets, in \cite{Moeglin2011_adams_conjecture}, she proved this conjecture for classical groups in the case where $\dim(V)-\dim(W)\gg0$. In our work, we generalize her result to the metaplectic groups and obtain the following theorem:
	
	\begin{theorem}(Theorem \ref{theorem: adams conjecture})
		If $\alpha\gg0$, then Conjecture \ref{conjecture: the adams conjecture} is true.
	\end{theorem}
	
	\subsection*{Idea of the proofs}
	
	\par The strategy in the article can be summarized as follows:
	
	\begin{enumerate}
		\item[$(1)$] Let $\psi$ be a non-negative DDR. To construct representations in $\Pi_\psi$, we first generalize \cite[Theorem 7.5]{Xu2017_Moeglin} to the metaplectic groups via endoscopy theory (Theorem \ref{theorem: Mp version of Theorem 7.5 in Xu_moeglin}). This result provides an inductive description for the representation $\pi(\psi,\varepsilon)$ with $\varepsilon\in\mS_\psi$.
		\item[$(2)$] Using the inductive description in Theorem \ref{theorem: Mp version of Theorem 7.5 in Xu_moeglin}, we follow M\oe glin's method in \cite{Mœeglin2009_paquets} to construct representations in $\Pi_\psi$. Specifically, these representations are the socles (i.e., maximal semisimple subrepresentations) of certain parabolic inductions (Theorem \ref{theorem: A-packets can be constructed via extended multi-segments - NDDR case}). This step involves only representation-theoretic techniques and does not involve endoscopy theory. Hence, it is parallel to \cite{Mœeglin2009_paquets}.
		\item[$(3)$] To construct representations in $\Pi_\psi$ for $\psi$ of good parity, we follow Atobe's method in \cite{Atobe2022_A-packet}. By the compatibility of spectral transfer and partial Jacquet module, which was established by Fei Chen in \cite{Chen2024_commutationtransferaubertzelevinskiinvolution}, we prove that every representation in $\Pi_\psi$ is a derivative of a representation in $\Pi_{\psi_\textbf{t}}$, where $\psi_{\textbf{t}}$ is a non-negative DDR (Proposition \ref{prop: reduction to non-negative DDR}). This enables us to describe the representations in $\Pi_\psi$ by Atobe's extended multi-segments (Theorem \ref{theorem: A-packets can be constructed via extended multi-segments}).
		\item[$(4)$] For general $\psi$, the representations in $\Pi_\psi$ are irreducible parabolic inductions of the form $\tau\rtimes\pi$. Here $\tau$ is a generalized segment, $\pi\in\Pi_{\psi_\gp}$, and $\psi_\gp$ is the good parity part of $\psi$ (Proposition \ref{prop: reduction to good parity}). We will prove this fact by following \cite[\S 6]{Mœglin2006_arthur_packets}. Since the proof does not involve endoscopy theory, our proof is parallel to \cite[\S 6]{Mœglin2006_arthur_packets}.
		\item[$(5)$] The arguments used in M\oe glin's works are highly technical. Therefore, to avoid repetitions, we use theta correspondence to transfer the classical results to metaplectic groups wherever possible. For example, we obtain the theorem of multiplicity one by theta correspondence. To make the theory of theta correspondence applicable, we study the compatibility of theta correspondence with Arthur parameters and prove the Adams conjecture for $\alpha\gg0$ in \S \ref{section: Adams conjecture} based on Atobe and Gan's earlier work in \cite{Atobe2017_local}.
		\item[$(6)$] While the proofs of many propositions can be simplified by theta correspondence, there are still some that can not, such as the proofs of Proposition \ref{proposition: key proposition} and Theorem \ref{theorem: construction of A-packet non-negative DDR case}. For these, we must repeat M\oe glin's work. These proofs are lengthy and technical, but they differ only slightly from M\oe glin's work.
	\end{enumerate}

	\subsection*{Organization}
	
	\par We now outline the structure of this article. In \S \ref{section: notation and preliminaries}, we recall the necessary background and results needed in our work. In \S \ref{section: derivative and socles}, we generalize Atobe's theory of derivatives and socles to metaplectic groups. In \S \ref{section: discrete series}, we study the properties of discrete series, which are crucial for the construction of Arthur packets. In \S \ref{section: non-negative DDR}, we construct Arthur packets for non-negative DDR parameters. The proof of Theorem \ref{theorem: construction of A-packet non-negative DDR case} is deferred to \S \ref{section: proof of the theorem}, since it is lengthy and technical. In \S \ref{section: construction of Arthur packets}, we construct Arthur packets for general parameters. In \S \ref{section: Adams conjecture}, we study the compatibility of theta lifts and Arthur parameters and prove the Adams conjecture in the $\alpha\gg0$ case. Finally, in \S \ref{section: Consequences of the Adams conjecture}, we derive some consequences of the Adams conjecture, including the multiplicity free property of Arthur packets.
	
	\subsection*{Acknowledgement}
	
	\par The author would like to thank his advisor Prof. Wen-Wei Li for his detailed
	instruction and helpful advice. The author is grateful to Marcela Hanzer for
	useful comments.
	
	\section{Notations and preliminaries}\label{section: notation and preliminaries}
	
	\par In this section, we introduce the notations used throughout the article.
	
	\subsection{Generalized segments}\label{subsection: generalized segments}
	
	\par Let $\rho$ be an irreducible cuspidal representation of $\GL(d_\rho)$ and let $x, y$ be real numbers such that $x-y\in\ZZ$. A tuple of representations of the form $$(\rho|\cdot|^x,\rho|\cdot|^{x+\zeta},\dots,\rho|\cdot|^{y-\zeta},\rho|\cdot|^y)$$ is called a segment, where $\zeta=\Sgn(y-x)$. We denote such a segment by $[x,y]_\rho$.
	\index{segment}
	
	\par It is well known that the parabolic induction representation $$\rho|\cdot|^x\times\rho|\cdot|^{x+\zeta}\times\cdots\times\rho|\cdot|^y$$ has a unique irreducible subrepresentation and a unique irreducible quotient. We denote the subrepresentation by $\Z([x,y]_\rho)$ and the quotient by $\rL([x,y]_\rho)$.
	
	\par For simplicity, we abuse notation by using $[x,y]_\rho$ to also denote the representation $\Z([x,y]_\rho)$.
	\index{x-y segment @$[x,y]_\rho$}
	
	\par Fix a $\zeta\in\left\{\pm 1\right\}$. A generalized segment is a matrix $$\begin{bmatrix}
	    x_{11}&\dots&x_{1n}\\
	    \vdots&&\vdots\\
	    x_{m1}&\dots&x_{mn}\\
	\end{bmatrix}_\rho$$
	with $x_{i+1,j}=x_{ij}-\zeta$, $x_{i,j+1}=x_{ij}+\zeta$ for all $i,j$. The representation attached to the above generalized segment is defined to be the unique irreducible subrepresentation of $$[x_{11},x_{1n}]_\rho\times[x_{21},x_{2n}]_\rho\times\cdots\times[x_{m1},x_{mn}]_\rho.$$ We also abuse notation and simply use the above matrix to denote the corresponding representation.
	\index{generalized segment}
	
	\par It should be noted that, by M\oe glin-Waldspurger algorithm in \cite{Moeglin1986_involution}, the generalized segment representation above is also the unique irreducible subrepresentation of $$[x_{11},x_{m1}]_\rho\times[x_{12},x_{m2}]_\rho\times\cdots\times[x_{1n},x_{mn}]_\rho.$$
	
	\subsection{Metaplectic groups}
	
	\par Let $(W_{2n},\left<\cdot,\cdot\right>)$ be a symplectic $F$-vector space of dimension $2n$. We fix a symplectic basis $$p_1,\dots,p_n,q_1,\dots,q_n$$ of $W$. Then, there is a standard Borel pair $(B,T)$ of $\Sp(W)$ corresponds to the flag $\left\{0\right\}\subset Fp_1\subset\cdots\subset\bigoplus_{i=1}^nFp_i$. The standard parabolic and Levi subgroups are thus defined.
	
	\par Given an additive character $\uppsi$ of $F$, the metaplectic group associated to $(W_{2n},\left<\cdot,\cdot\right>)$ and $\uppsi$ in this work is the central extension of locally compact groups $$1\rightarrow\bm\mu_8\rightarrow\widetilde{\Sp}(W)\rightarrow\Sp(W)\rightarrow1.$$ The extension is constructed using the Schr\"{o}dinger model for the irreducible representations of Heisenberg group of $(W_{2n},\left<\cdot,\cdot\right>)$ with central character $\uppsi$.
	\index{metaplectic group}
	
	\par We will denote $\widetilde{\Sp}(W_{2n})$ by $\widetilde{\Sp}(2n)$ for simplicity. For any subgroup $H\subset\Sp(2n)$, we denote the inverse image of $H$ in $\widetilde{\Sp}(2n)$ by $\widetilde{H}$. One of the advantages of working with the $\bm\mu_8$ extension instead of working with the usual $\bm\mu_2$ extension is that, for any standard parabolic subgroup $P=MN$, where $M=\prod_{i=1}^k\GL(n_i)\times\Sp(2m)\subset\Sp(2n)$ is the Levi factor of $P$ and $N$ is the unipotent radical, we have canonical isomorphisms (depending on the additive character $\uppsi$)
	$$\begin{aligned}
		\widetilde{M}&\cong\prod_{i=1}^k\GL(n_i)\times\widetilde{\Sp}(2m),\\
		\widetilde{N}&\cong N.
	\end{aligned}$$ We say $\widetilde{P}$ is a standard parabolic subgroup of $\widetilde{\Sp}(2n)$ and say $\widetilde{P}=\widetilde{M}\widetilde{N}=\widetilde{M}N$ is the Levi decomposition of $\widetilde{P}$.
	\index{Sp tilde@$\widetilde{\Sp}(2n)$}
	
	\par Write $G_{2n}=\Sp(2n)$, $\widetilde{G}_{2n}=\widetilde{\Sp}(2n)$. We say a representation $(\pi, V)$ of $\widetilde{G}_{2n}$ is genuine if $\bm\mu_8$ acts by $z\mapsto z\cdot\Id$. We define the sets $$\Pi_-(\widetilde{G}_{2n})\supset\Pi_{\unit,-}(\widetilde{G}_{2n})\supset\Pi_{\temp,-}(\widetilde{G}_{2n})\supset\Pi_{2,-}(\widetilde{G}_{2n})$$ to be the sets of isomorphism classes of genuine representations of $\widetilde{G}_{2n}$, which are respectively all such representations, the unitary ones, the tempered ones, and the essentially square-integrable ones.
	\index{G tilde@$\widetilde{G}_{2n}$}
	\index{genuine! representation}
	\index{Pi genuine@$\Pi_-(\widetilde{G}_{2n})$, $\Pi_{\unit,-}(\widetilde{G}_{2n})$}
	\index{Pi genuine2@$\Pi_{\temp,-}(\widetilde{G}_{2n})$, $\Pi_{2,-}(\widetilde{G}_{2n})$}
	
	\par Let $\mR_-(\widetilde{G}_{2n})$ be the Grothendieck group of the category of genuine smooth representations of $\widetilde{G}_{2n}$ of finite length and write $\mR_-(\widetilde{G})=\bigoplus_{n\ge0}\mR_-(\widetilde{G}_{2n})$. Now, the normalized parabolic induction defines a map $\mR(\GL)\otimes\mR_-(\widetilde{G}), \pi\otimes\sigma\mapsto\pi\rtimes\sigma$ and the normalized Jacquet module gives a map $\mu^*:\mR_-(\widetilde{G})\rightarrow\mR(\GL)\otimes\mR_-(\widetilde{G})$ (for the definition of $\mu^*$, see \cite[\S 4.2]{Hanzer2010}).
	\index{R grothendieck group@$\mR_-(\widetilde{G}_{2n})$}
	
	\par The metaplectic version of the Tadi\'{c} formula is proved in \cite{Hanzer2010}:
	
	\begin{proposition}(\cite[Proposition 4.5]{Hanzer2010})\label{prop: Tadic formula}
		For $\pi\in\mR(\GL)$ and $\sigma\in\mR_-(\widetilde{G}_{2n})$, we have the following Tadi\'{c} formula $$\mu^*(\pi\rtimes\sigma)=M^*(\pi)\rtimes\mu^*(\sigma).$$ The $M^*$ above is defined by $M^*=(m\otimes\Id)\circ(\vee\otimes m^*)\circ\kappa\circ m^*$, where $m$ and $m^*$ are the multiplication and comultiplication of $\mR(\GL)$, $\kappa$ is defined by $\pi_1\otimes\pi_2\mapsto\pi_2\otimes\pi_1$, and $\vee$ is the contragredient.
	\end{proposition}
	
	\par Note that the Tadi\'{c} formula here differs slightly from that in \cite{Hanzer2010}. That is, we do not need the character $\alpha$ that occurs in \cite[Proposition 4.5]{Hanzer2010}. This difference is due to the fact that \cite{Hanzer2010} deals with the $\bm\mu_2$ extension, whereas we consider the $\bm\mu_8$ extension.
	
	\par The MVW-involution is a powerful tool to study the irreducible representations of classical groups. For metaplectic groups, the existence of the MVW-involution was established in \cite{Mœglin2006_correspondances}, from which we can deduce the following proposition:
	
	\begin{proposition}\label{prop: consequence of MVW-involution}
		Let $\pi\in\Pi(\GL)$, $\sigma\in\Pi_-(\GL)$, then we have
		\begin{enumerate}
			\item[(1)] $\pi\rtimes\sigma=\pi^\vee\rtimes\sigma$ in $\mR_-(\GL)$.
			\item[(2)] $\tau$ is an irreducible subrepresentation of $\pi\rtimes\sigma$ if and only if $\tau$ is an irreducible quotient representation of $\pi^\vee\rtimes\sigma$.
		\end{enumerate}
	\end{proposition}
	\begin{proof}
		This is a standard consequence of the MVW-involution. See, e.g., \cite[Theorem 2.1]{Hanzer2008_algebraic}.
	\end{proof}
	
	\subsection{Endoscopy}

	\par Following \cite{Li2024_arthurpacketsmetaplecticgroups}, we define $\mE_{\mathrm{ell}}(\widetilde{G}_{2n})$ to be the set of pairs $\textbf{G}^!=(n',n'')$, where 
	$$\begin{aligned}
		(n',n'')\in\ZZ_{\ge0}^2,\\
		n'+n''=n.
	\end{aligned}$$ We call $\textbf{G}^!\in\mE_{\mathrm{ell}}(\widetilde{G}_{2n})$ an elliptic endoscopic datum of $\widetilde{G}_{2n}$ and $G^!:=\SO(2n'+1)\times\SO(2n''+1)$ the endoscopic group attached to the datum $\textbf{G}^!$. Here $\SO(2n+1)$ means the split form.
	\index{E endoscopic datum@$\mE_{\mathrm{ell}}(\widetilde{G}_{2n})$}
	\index{G endoscopic datum@$\textbf{G}^{"!}$, $G^{"!}$}
	
	\par It should be noted that $(n',n'')$ and $(n'',n')$ will give rise to inequivalent endoscopic data. This is different from the endoscopy for $\SO(2n+1)$.

	\par A function $f:\widetilde{G}_{2n}\rightarrow\CC$ is said to be genuine (resp. anti-genuine) if $f(z\widetilde{x})=zf(\widetilde{x})$ (resp. $f(z\widetilde{x})=z^{-1}f(\widetilde{x})$) for all $\widetilde{x}\in\widetilde{G}_{2n}$ and $z\in\bm\mu_8$. For $\textbf{G}^!\in\mE_{\mathrm{ell}}(\widetilde{G}_{2n})$, we define $\mI_{--}(\widetilde{G}_{2n})$ and $S\mI(G^!)$ to be the quotient of the space of anti-genuine $C_c^\infty$-functions on $\widetilde{G}_{2n}$ (resp. $C_c^\infty$-function on $G^!(F)$) modulo those with zero orbital integrals (resp. stable orbital integrals) along all strongly regular semisimple orbits as in \cite[Definition 2.3.4, 2.3.5]{Li2024_stabilizationtraceformulametaplectic}
	\index{genuine! function}
	\index{anti-genuine function}
	
	\par The geometric	transfer in \cite[Theorem 3.8.1]{Li2024_stabilizationtraceformulametaplectic} is a linear map
	$$\mT_{\textbf{G}^!,\widetilde{G}_{2n}}:\mI_{--}(\widetilde{G}_{2n})\otimes\mathrm{mes}(G_{2n})\rightarrow S\mI(G^!)\otimes\mathrm{mes}(G^!)$$ characterized by matching orbital integrals (see \cite{Li2024_stabilizationtraceformulametaplectic} for more details).
	
	\par Denote by $D_-(\widetilde{G}_{2n})$ (resp. $SD(G^!)$) the dual space of $\mI_{--}(\widetilde{G}_{2n})$ (resp. $S\mI(G^!)$). They are the spaces of genuine invariant (resp. stably invariant) distribution on $\widetilde{G}_{2n}$ (resp. $G^!(F)$). By dualizing $\mT_{\textbf{G}^!,\widetilde{G}_{2n}}$, we obtain the spectral transfer
	$$\mT_{\textbf{G}^!,\widetilde{G}_{2n}}^\vee:SD(G^!)\otimes\mathrm{mes}(G^!)^\vee\rightarrow D_-(\widetilde{G}_{2n})\otimes\mathrm{mes}(G_{2n})^\vee.$$
	\index{T transfer@$\mT_{\textbf{G}^{"!},\widetilde{G}_{2n}}$, $\mT_{\textbf{G}^{"!},\widetilde{G}_{2n}}^\vee$}
	
	\par For every genuine admissible representation $\pi$ of $\widetilde{G}_{2n}$ of finite length, its character $\Theta_\pi$ belongs to $D_-(\widetilde{G}_{2n})$. We denote by $D_{\spec,-}(\widetilde{G}_{2n})$ the linear subspace spanned by these characters, and let $SD_{\spec}(G^!)=D_{\spec}(G^!)\cap SD(G^!)$.
	
	\begin{theorem}(\cite[Theorem 2.4.4]{Li2024_arthurpacketsmetaplecticgroups})
		The spectral transfer $\mT_{\textbf{G}^!,\widetilde{G}_{2n}}^\vee$ restricts to a linear map
		$$SD_{\spec}(G^!)\otimes\mathrm{mes}(G^!)^\vee\rightarrow D_{\spec,-}(\widetilde{G}_{2n})\otimes\mathrm{mes}(G_{2n})^\vee.$$
	\end{theorem}

	\subsection{Arthur parameters}\label{subsection: Arthur parameters}
	
	\par The dual group of $\widetilde{G}_{2n}$ is defined to be $\widetilde{G}_{2n}^\vee:=\Sp(2n,\CC)$. It is the same as the dual group of $\SO(2n+1)$. An A-parameter of $\widetilde{G}_{2n}$ is a $\widetilde{G}_{2n}^\vee$-conjugacy class of admissible homomorphisms $$\psi:W_F\times\SL(2,\CC)\times\SL(2,\CC)\rightarrow\widetilde{G}_{2n}^\vee$$ such that the image of the Weil group $W_F$ is bounded.
	\index{G complex dual@$\widetilde{G}_{2n}^\vee$}
	\index{A-parameter}
	
	\par We can regard $\psi$ as a representation of $W_F\times\SL(2,\CC)\times\SL(2,\CC)$. It decomposes as 
	\begin{equation}\label{equation: decompose of an A-parameter}
	    \psi=\bigoplus_{i\in I}m_i\psi_i
	\end{equation}
	where $I$ is a finite set, $\psi_i$ is irreducible, $m_i\in\ZZ_{\ge 1}$, such that $i\neq j\Rightarrow\psi_i\not\cong\psi_j$. We further write $$\psi_i=\rho_i\otimes r(a_i)\otimes r(b_i),$$ where $\rho_i$ can be identified with unitary irreducible cuspidal representations of $\GL(d_\rho,F)$ by local Langlands correspondence for $\GL(d_\rho,F)$ and $r(a_i)$ is the unique irreducible algebraic representation of $\SL(2,\CC)$ of dimension $a$.
	
	\par Following \cite[\S 3.1]{Li2024_arthurpacketsmetaplecticgroups}, we can decompose $I$ into 
	\begin{equation}\label{equation: decompose of the index set of an A-parameter}
		I=I^+\sqcup I^-\sqcup J\sqcup J'
	\end{equation}
	where $J$ and $J'$ are related by a bijection $j\leftrightarrow j'$, such that
	
	\begin{itemize}
		\item if $i\in I^+$ then $\rho_i\otimes r(a_i)\otimes r(b_i)$ is of symplectic type, i.e.
		\begin{itemize}
			\item[-] either $\rho_i$ is symplectic and $a_i+b_i$ is even,
			\item[-] or $\rho_i$ is orthogonal and $a_i+b_i$ is odd;
		\end{itemize}
		\item if $i\in I^-$ then $\rho_i\otimes r(a_i)\otimes r(b_i)$ is of orthogonal type, i.e.
		\begin{itemize}
			\item[-] either $\rho_i$ is orthogonal and $a_i+b_i$ is even,
			\item[-] or $\rho_i$ is symplectic and $a_i+b_i$ is odd;
		\end{itemize}
		\item if $j\in J$, then $\rho_j\otimes r(b_j)$ is not self dual, and $$\phi_{j'}=\phi_j^\vee,\; m_{j'}=m_j.$$
	\end{itemize}
	
	We say that $\psi$ is of good parity if $I=I^+$. Further, if $I=I^+$ and $m_i=1$ for all $i\in I$, we say $\psi$ is discrete.
	\index{A-parameter!discrete}
	\index{A-parameter!of good parity}
	
	\par Let $\Psi(\widetilde{G}_{2n})\supset\Psi_\gp(\widetilde{G}_{2n})\supset\Psi_2(\widetilde{G}_{2n})$ be the sets of equivalence classes of A-parameters, A-parameters of good parity and discrete A-parameters, respectively. Also, we let $\Phi_\bdd(\widetilde{G}_{2n})$ be the subset of $\Psi(\widetilde{G}_{2n})$ consisting of A-parameters $\phi$ which are trivial on the second $\SL(2,\CC)$ factor. Finally, we set $\Phi_{\bdd,\gp}(\widetilde{G}_{2n})=\Phi_\bdd(\widetilde{G}_{2n})\cap\Psi_\gp(\widetilde{G}_{2n})$ and $\Phi_{\bdd,2}(\widetilde{G}_{2n})=\Phi_\bdd(\widetilde{G}_{2n})\cap\Psi_2(\widetilde{G}_{2n})$. Elements of $\Phi_\bdd(\widetilde{G}_{2n})$ are called bounded L-parameters of $\widetilde{G}_{2n}$.
	\index{Psi A-parameter@$\Psi(\widetilde{G}_{2n})$, $\Psi_{\gp}(\widetilde{G}_{2n})$, $\Psi_2(\widetilde{G}_{2n})$}
	\index{Phi L-parameter bounded@$\Phi_{\bdd}(\widetilde{G}_{2n})$, $\Phi_{\bdd,\gp}(\widetilde{G}_{2n})$, $\Phi_{\bdd,2}(\widetilde{G}_{2n})$}
	\index{bounded L-parameter}
	
	\par Let $\psi\in\Psi(\widetilde{G}_{2n})$. We define its centralizer group and the corresponding component group as
	$$\begin{aligned}
		S_\psi&:=Z_{\widetilde{G}_{2n}^\vee}(\mathrm{Im}(\psi)),\\
		\mS_\psi&:=\pi_0(S_\psi).
	\end{aligned}$$
	\index{S centralizer@$S_\psi$, $\mS_\psi$}
	
	\par By \cite[\S1.4]{Arthur2013_book} we have canonical isomorphisms
	\begin{equation}\label{equation: centralizer group of A-parameter}
		\begin{aligned}
			S_\psi&\cong\prod_{i\in I^+}\Or(m_i,\CC)\times\prod_{i\in I^-}\Sp(m_i,\CC)\times\prod_{j\in J}\GL(m_j,\CC),\\
			\mS_\psi&\cong\bm\mu_2^{I^+};
		\end{aligned}
	\end{equation}
	the quotient map $S_\psi\rightarrow\mS_\psi$ is given by taking determinants in $\Or(m_i,\CC)$.
	
	\par By the isomorphism $\mS_\psi\cong\bm\mu_2^{I^+}$, we view elements of $\mS_\psi$ as $\bm\mu_2$-valued functions on $I^+$. Let $\varepsilon=(\varepsilon_i)$ and $s=(s_i)$ be two elements in $\bm\mu_2^{I^+}$. There is a non-degenerate pairing on $\bm\mu_2^{I^+}$ defined by $(\varepsilon,s)=\prod_i\varepsilon_i*s_i$, where $$\varepsilon_i*s_i=\begin{cases}
	    -1&\text{if }\varepsilon_i=s_i=-1,\\
	    1&\text{otherwise.}
	\end{cases}$$ Thus we also view characters $\varepsilon$ in $\mS_\psi^\vee$ as functions on $I^+$.
	
	\par We define the multi-set of Jordan blocks for $\psi$ as follows, $$\Jord(\psi):=\left\{(\rho_i,a_i,b_i)\text{ with multiplicity }m_i:i\in I\right\}.$$ For any $\rho$, we define $$\Jord_\rho(\psi):=\left\{(a_i,b_i):(\rho,a_i,b_i)\in\Jord(\psi)\right\}.$$ In particular, when $\phi$ is a L-parameter, we view $\Jord_\rho(\phi)$ as a multi-subset of $\ZZ$.
	\index{Jordan block}
	\index{Jordan block of a parameter@$\Jord(\psi)$, $\Jord_\rho(\psi)$}
	
	\par When $\psi$ is discrete, there exists a natural bijection between $\Jord(\psi)$ and $I^+$, thus we can view elements of $\mS_\psi$ and $\mS_\psi^\vee$ as functions on $\Jord(\psi)$. In general, we will also view elements of $\mS_\psi$ and $\mS_\psi^\vee$ as functions on $\Jord(\psi)$ by abuse of notation.
	
	\par Set $s_\psi\in S_\psi$ to be $$s_\psi:=\psi\left(1,\begin{pmatrix}-1&\\&-1\end{pmatrix}\right).$$ Then, for $i\in I^{\pm}$ (resp. $j\in J$), the projection of $s_\psi$ to the corresponding direct factors of $S_\psi$ equals $1$ if $b_i$ (resp. $b_j$) is odd, $-1$ if $b_i$ (resp. $b_j$) is even.
	\index{s psi element@$s_\psi$}
	
	\subsection{Arthur packets}
	
	\par For $\psi\in\Psi(\widetilde{G}_{2n})$, we define 
	$$S_{\psi,2}:=\left\{s\in S_\psi:s^2=1\right\}.$$ \index{S centralizer of order 2@$S_{\psi,2}$}
	By \cite[Proposition 4.2.1]{Li2024_arthurpacketsmetaplecticgroups}, we have the following basic bijection
	\begin{equation}\label{equation: basic bijection}
		\left\{(\textbf{G}^!,\psi^!):\textbf{G}^!\in\mE_{\mathrm{ell}}(\widetilde{G}_{2n}),  \psi^!\in\Psi(G^!)\right\}\leftrightarrow\left\{(\psi,s):\psi\in\Psi(\widetilde{G}), s\in S_{\psi,2}/\text{conj}\right\}.
	\end{equation}
	To be more precise, for $s\in S_{\psi,2}$, the underlying $\CC$-vector space decomposes into $$V_{\psi}=V_{\psi}^+\oplus V_{\psi}^-,$$ where $V_{\psi}^{\pm}$ are the $\pm1$-eigenspaces of $s$. Denote by $\psi^{s=\pm1}$ the action of $W_F\times\SL(2,\CC)\times\SL(2,\CC)$ on $V_\psi^{\pm}$. Let $\textbf{G}^!=(\dim V_\psi^+,\dim V_\psi^-)$, $\psi^!=\psi^{s=1}\times\psi^{s=-1}$. Then the pair $(\textbf{G}^!,\psi^!)$ corresponds to $(\psi,s)$ in (\ref{equation: basic bijection}).
	\index{psi endoscopic@$\psi^{"!}$, $\psi^{s=\pm 1}$}
	
	\par Given a $\psi\in\Psi(\widetilde{G}_{2n})$ and an $s\in\mS_{\psi,2}$. We choose $(\psi,s)\leftrightarrow(\textbf{G}^!,\psi^!)$ as in (\ref{equation: basic bijection}). Then we define a genuine distribution $T_{\psi,s}$ by
	\begin{equation}\label{equation: distribution T_psi s}
		T_{\psi,s}:=\epsilon(\psi^{s=-1})\cdot\mT_{\textbf{G}^!,\widetilde{G}_{2n}}^\vee(S\Theta_{\psi^!}^{G^!}).
	\end{equation} Here $\epsilon(\psi^{s=-1}):=\epsilon(\psi^{s=-1}|_{W_F\times\SL(2,\CC)})$ is the local root number (for definition of local root number, see \cite[\S 4.1]{Li2024_arthurpacketsmetaplecticgroups}), and $S\Theta_{\psi^!}^{G^!}$ is the stable distribution on $G^!(F)$ in \cite[Theorem 2.2.1]{Arthur2013_book}, Whittaker-normalized as in \cite[\S 2.1]{Arthur2013_book}.
	\index{T psi s genuine distribution@$T_{\psi,s}$}
	\index{ep local root number@$\epsilon(\psi^{s=-1})$}
	\index{local root number}
	\index{STheta stable distribution@$S\Theta_{\psi^{"!}}^{G^{"!}}$}
	
	\par By \cite[Lemma 4.3.3]{Li2024_arthurpacketsmetaplecticgroups}, $T_{\psi,s}$ depends only on the image $\underline{s}$ of $s$ under the natural projection $S_\psi\rightarrow\mS_\psi$. Conversely, by (\ref{equation: centralizer group of A-parameter}), it is not hard to see that $S_{\psi,2}\rightarrow\mS_\psi$ is surjective. Thus, the symbol $T_{\psi,\underline{s}}$ for $\underline{s}\in\mS_\psi$ is well-defined.
	
	\par For $\varepsilon\in\mS_\psi^\vee$, we put $\pi(\psi,\varepsilon):=|\mS_\psi|^{-1}\sum_{\underline{s}\in\mS_\psi}\varepsilon(s_\psi\underline{s})T_{\psi,\underline{s}}.$ The main local result of \cite{Li2024_arthurpacketsmetaplecticgroups} is the following theorem:
	\index{pi representation of A-type@$\pi(\psi,\varepsilon)$}
	
	\begin{theorem}(\cite[Theorem 4.5.2]{Li2024_arthurpacketsmetaplecticgroups})\label{theorem: pi(psi,varepsilon) is a finite length representation}
		Let $\psi\in\Psi(\widetilde{G}_{2n})$. Then $\pi(\psi,\varepsilon)$ is a linear combination (possibly zero) of unitary genuine irreducible representations of $\widetilde{G}_{2n}$ with coefficients in $\ZZ_{\ge0}$ for all $\varepsilon\in\mS_\psi^\vee$.
	\end{theorem}

	\par By the theorem above, we may view $\pi(\psi,\varepsilon)$ as a semisimple genuine representation of $\widetilde{G}_{2n}$ of finite length. Now, the A-packet $\Pi_\psi$ associated to an A-parameter $\psi$ is the multi-set consisting of all irreducible constituents of $\bigoplus_{\varepsilon\in\mS_\psi^\vee}\pi(\psi,\varepsilon)$.
	\index{Pi A-packet@$\Pi_\psi$} 
	\index{A-packet}
	
	\par Fix a $\psi\in\Psi(\widetilde{G}_{2n})$. We write $\psi=\bigoplus_{i\in I}m_i\psi_i$ as in (\ref{equation: decompose of an A-parameter}) and decompose $I$ into $I^+\sqcup I^-\sqcup J\sqcup J'$ as in (\ref{equation: decompose of the index set of an A-parameter}). Then $\psi_\gp=\bigoplus_{i\in I^+}m_i\psi_i$ and $\psi_\np=\bigoplus_{i\in I^-}\frac{m_i}{2}\psi_i\bigoplus_{j\in J}m_j\psi_j$ are well-defined A-parameters and we have the following decomposition:
	\begin{equation}\label{equation: reduction to good parity}
        \psi=\psi_\np^\vee\oplus\psi_\gp\oplus\psi_\np.
	\end{equation} By (\ref{equation: centralizer group of A-parameter}), there exists a canonical isomorphism $\mS_\psi\cong\mS_{\psi_\gp}$, hence we may identify $\mS_{\psi_\gp}^\vee$ with $\mS_{\psi}^\vee$. Then, we have the following proposition, which can be viewed as a more explicit version of \cite[Proposition 4.5.3]{Li2024_arthurpacketsmetaplecticgroups}:
	\index{psi good parity@$\psi_\gp$, $\psi_\np$}
	
	\begin{proposition}\label{prop: proposition 4.5.3 in ArthurMp}
		For $\varepsilon\in\mS_{\psi_\gp}^\vee$, we have $\pi(\psi,\varepsilon)=\tau_{\psi_\np}\rtimes\pi(\psi_\gp,\varepsilon)$, where $$\tau_{\psi_\np}=\bigtimes_{(\rho,a,b)\in\Jord(\psi_\np)}\begin{bmatrix}\frac{a_i-b_i}{2}&\dots&\frac{a_i+b_i}{2}-1\\\vdots&&\vdots\\-\frac{a_i+b_i}{2}+1&\dots&-\frac{a_i-b_i}{2}\end{bmatrix}_\rho.$$
	\end{proposition}
	\begin{proof}
		As in the proof of \cite[Proposition 4.4.1]{Li2024_arthurpacketsmetaplecticgroups}, by a suitable choice of endoscopic data \textbf{s}, spectral transfer commutes with parabolic induction (see also \cite[\S 3.8]{Li2024_stabilizationtraceformulametaplectic}). Note that by \cite[\S 6]{Mœglin2006_arthur_packets}, for $s\in S_{\psi_\gp,2}$, we have $\Pi_{\psi^!}=\tau_{\psi_\np}\rtimes\Pi_{\psi_\gp^!}$, where the parabolic induction on the right-hand side takes place in the first $\SO$-factor. Thus, we have $T_{\psi,\underline{s}}=\tau_{\psi_\np}\rtimes T_{\psi,\underline{s}}$ for all $\underline{s}\in\mS_\psi$, from which we deduce the proposition.
	\end{proof}
	\index{tau generalized segment@$\tau_{\psi_\np}$}
	
	\subsection{Langlands correspondence}
	
	\par The local Langlands correspondence for $\widetilde{G}_{2n}$ has been established in \cite{GanSavin2012_LLC} via theta correspondence, and it can be phrased as the following theorem:
	
	\begin{theorem}(\cite{GanSavin2012_LLC})\label{theorem: LLC for Mp}
		There is a canonical decomposition
		\begin{equation}\label{euqation: LLC for Mp}
			\begin{aligned}
				\Pi_-(\widetilde{G}_{2n})&=\bigsqcup_{\phi\in\Phi(\widetilde{G}_{2n})}\Pi_\phi\\
				\mS_\phi^\vee&\stackrel{1:1}{\longleftrightarrow}\Pi_\phi\\
				\varepsilon&\mapsto\pi(\phi,\varepsilon).
			\end{aligned}
		\end{equation}
		Moreover, the decomposition in (\ref{euqation: LLC for Mp}) restricts to
		$$\begin{aligned}
			\Pi_{\temp,-}(\widetilde{G}_{2n})&=\bigsqcup_{\phi\in\Phi_{\bdd}(\widetilde{G}_{2n})}\Pi_\phi,\\
			\Pi_{2,-}(\widetilde{G}_{2n})&=\bigsqcup_{\phi\in\Phi_{\bdd,2}(\widetilde{G}_{2n})}\Pi_\phi.
		\end{aligned}$$
	\end{theorem}
	
	\par One natural question is whether the $\pi(\phi,\varepsilon)$ in Theorem \ref{theorem: LLC for Mp} is the same as the $\pi(\psi,\varepsilon)$ in Theorem \ref{theorem: pi(psi,varepsilon) is a finite length representation} when $\psi\in\Phi_{\bdd}(\widetilde{G}_{2n})$. This issue was addressed in \cite{Luo2020_ECR}. Specifically, it was proved that the $\pi(\phi,\varepsilon)$ in Theorem \ref{theorem: LLC for Mp} satisfies the endoscopic character relation in the following theorem:
	
	\begin{theorem}\cite[Theorem 1.1]{Luo2020_ECR}\label{theorem: ECR for tempered representation}
		For $\phi\in\Phi_{\bdd}(\widetilde{G})$ and $\underline{s}\in\mS_\phi$, we have $$T_{\phi,\underline{s}}=\sum_{\varepsilon\in\mS_\phi^\vee}\varepsilon(\underline{s})\pi(\phi,\varepsilon).$$Or equivalently, for $\varepsilon\in\mS_\phi^\vee$, we have $$\pi(\phi,\varepsilon)=|\mS_\phi|^{-1}\sum_{\underline{s}\in\mS_\phi}\varepsilon(\underline{s})T_{\phi,\underline{s}}.$$
	\end{theorem}
	
	\par By Theorem \ref{theorem: ECR for tempered representation}, the endoscopy theory can be applied to the study of tempered representations of $\widetilde{G}_{2n}$. For example, the following corollary can be proved by the endoscopy theory:
	
	\begin{corollary}\label{coro: strengthening of proposition 4.5.3 in ArthurMp - tempered case}
		Suppose that $\phi=\phi_1^\vee\oplus\phi_0\oplus\phi_1$ with $\phi\in\Phi_{\bdd}(\widetilde{G}_{2n})$ and $\phi_0\in\Phi_{\bdd,\gp}(\widetilde{G}_{2n})$. By (\ref{equation: centralizer group of A-parameter}), there exists an embedding $\mS_{\phi_0}\xhookrightarrow{}\mS_{\phi}$, $s\mapsto s_>$ by setting $s_>(\rho,a)=1$ for all $(\rho,a)\notin\Jord(\phi_0)$. Taking the dual, we have a projection $\mS_\phi^\vee\twoheadrightarrow\mS_{\phi_0}^\vee$. Then, for any $\varepsilon\in\mS_{\phi_0}^\vee$, we have
		$$\tau_{\phi_1}\rtimes\pi(\phi_0,\varepsilon)=\bigoplus_{\varepsilon_>\rightarrow\varepsilon}\pi(\phi,\varepsilon_>),$$ where $\varepsilon_>$ runs through all characters in $\mS_\phi^\vee$ whose image in $\mS_{\phi_0}^\vee$ is $\varepsilon$.
	\end{corollary}
	\begin{proof}
		By making induction on the number of irreducible components of $\phi_1$, it is not hard to see that we only need to consider the case when $\phi_1=\rho\otimes r(a)$ is irreducible. If $\phi_1$ is not of good parity, the corollary is just a special case of Proposition \ref{prop: proposition 4.5.3 in ArthurMp}. Thus, we may assume that $\phi_1$ is of good parity. If $(\rho,a)\in\Jord(\phi_0)$, then $\mS_{\phi_0}\xhookrightarrow{}\mS_\phi$ will be a bijection. By \cite[(7.6)]{Xu2017_cuspidal} and the compatibility of spectral transfer with parabolic induction, we have $T_{\phi,s}=\tau_{\phi_1}\rtimes T_{\phi_0,s}$ for all $s\in\mS_{\phi_0}\cong\mS_{\phi}$, which completes the proof in the case that $(\rho,a)\in\Jord(\phi_0)$.
		
		\par When $(\rho,a)\notin\Jord(\phi_0)$, we define $s_0\in\mS_\phi$ by $$\underline{s}_0(\rho',a')=\begin{cases}
			-1&\text{if } (\rho',a')=(\rho,a)\\
			1&\text{otherwise}.
		\end{cases}$$ For $\varepsilon\in\mS_{\phi_0}^\vee$, we define $\varepsilon_{\pm}$ by
		$$\varepsilon_{\pm}(\rho',a')=\begin{cases}
			\pm1&\text{if } (\rho',a')=(\rho,a)\\
			\varepsilon(\rho',a')&\text{otherwise}.
		\end{cases}$$
		Then we only need to prove that $$\tau_{\phi_1}\rtimes\pi(\phi_0,\varepsilon)=\pi(\phi,\varepsilon_+)\oplus\pi(\phi,\varepsilon_-).$$
				
		\par By definition, we have 
		$$\begin{aligned}
			\pi(\phi,\varepsilon_+)\oplus\pi(\phi,\varepsilon_-)=&|\mS_\phi|^{-1}\sum_{s\in\mS_\phi}(\varepsilon_+(s)+\varepsilon_-(s))T_{\phi,s}\\
			=&|\mS_\phi|^{-1}\sum_{s\in\mS_{\phi_0}}((1+1)\varepsilon(s)T_{\phi,s_>}+(1-1)\varepsilon(s)T_{\phi,s_>s_0})\\
			=&|\mS_{\phi_0}|^{-1}\sum_{s\in\mS_{\phi_0}}\varepsilon(s)T_{\phi,s_>}\\
		\end{aligned}$$ Now, since $s_>(\rho,a)=1$, the compatibility of spectral transfer with parabolic induction implies that $T_{\phi,s_>}=\tau_{\phi_1}\rtimes T_{\phi_0,s}$. This completes the proof.
	\end{proof}
	
	\subsection{Theta correspondence}
	
	\par Let $V_{2m+1}$ be a $2m+1$ dimensional split orthogonal space of discriminant $1$ over $F$. Write $H_{2m+1}=\Or(V_{2m+1})$. The pair $(\widetilde{G}_{2n},H_{2m+1})$ is a reductive dual pair in a certain metaplectic group.
	\index{orthogonal group@$\Or(V_{2m+1})$}
	\index{H orthogonal group@$H_{2m+1}$}
	
	\par Let $\omega_{V_{2m+1},W_{2n}}$ be the Weil representation of $\widetilde{G}_{2n}\times H_{2m+1}$. For $\pi\in\Pi_-(\widetilde{G}_{2n})$, the maximal $\pi$-isotypic quotient of $\omega_{V_{2m+1},W_{2n}}$ is of the form $$\pi\otimes\Theta_{V_{2m+1},W_{2n}}(\pi),$$ where $\Theta_{V_{2m+1},W_{2n}}(\pi)$ is a smooth representation of $H_{2m+1}$.
	
	\par We denote by $\theta_{V_{2m+1},W_{2n}}(\pi)$ the maximal semisimple quotient of $\Theta_{V_{2m+1},W_{2n}}(\pi)$. The following theorem is well-known:
	\index{theta lifts@$\Theta_{V_{2m+1},W_{2n}}$, $\theta_{V_{2m+1},W_{2n}}$}
	
	\begin{theorem}(Howe's duality)\label{theorem: Howe's duality}
		Let $\pi_1,\pi_2\in\Pi_-(\widetilde{G}_{2n})$. Then 
		\begin{enumerate}
			\item[(1)] if $\theta_{V_{2m+1},W_{2n}}(\pi_1)\neq0$, then $\theta_{V_{2m+1},W_{2n}}(\pi_1)$ is irreducible;
			\item[(2)] if $\theta_{V_{2m+1},W_{2n}}(\pi_1)\cong\theta_{V_{2m+1},W_{2n}}(\pi_2)\neq0$, then $\pi_1\cong\pi_2$.
		\end{enumerate}
	\end{theorem}
	
	\par We set $\alpha=2m-2n$, $\Theta_{-\alpha}=\Theta_{V_{2m+1},W_{2n}}$ and $\theta_{-\alpha}=\theta_{V_{2m+1},W_{2n}}$. When $\alpha\gg0$, Kudla's filtration provides the following lemma:
	\index{theta lifts alpha@$\Theta_{-\alpha}$, $\theta_{-\alpha}$}
	
	\begin{lemma}\label{lemma: compatibility of theta lifts and socle}
		Let $\pi\in\Pi_-(\widetilde{G}_{2n})$, $\pi_0\in\Pi_-(\widetilde{G}_{2n-2k})$ and $\sigma\in\Pi(\GL(k))$. Then for sufficiently large $\alpha$ (depending only on $\sigma$), we have:
		\begin{enumerate}
			\item[(1)] $\pi\xhookrightarrow{}\sigma\rtimes\pi_0$ implies that $\Theta_{-\alpha}(\pi)\xhookrightarrow{}\sigma\rtimes\Theta_{-\alpha}(\pi_0)$.
			\item[(2)] If the Jacquet module of $\pi$ has only one irreducible subquotient on which $\GL(k)$ acts by $\sigma$ (i.e., $\Jac_{\sigma}(\pi)$ is irreducible; for the definition of $\Jac_{\sigma}$, see \S \ref{subsection: partial Jacquet modules and socles} below), then $\pi\xhookrightarrow{}\sigma\rtimes\pi_0$ implies that $\theta_{-\alpha}(\pi)\xhookrightarrow{}\sigma\rtimes\theta_{-\alpha}(\pi_0)$.
		\end{enumerate}
	\end{lemma}
	\begin{proof}
		This follows by the same method as in the proof of \cite[Lemma 4.9, 4.11]{Bakic2024_thetacorrespondencearthurpackets}.
	\end{proof}
	
	\section{Derivatives and socles}\label{section: derivative and socles}
	
	\par The theory of $\rho$-derivatives and socles plays an important role in Atobe's construction of A-packets. Therefore, we will develop the theory of $\rho$-derivatives and socles for metaplectic groups in this section. The main results of this section are Propositions \ref{prop: highest rho-derivative is irreducible in non-self-dual case}, \ref{prop: parabolic induction is SI in non-self-dual case}, \ref{prop: highest [0,zeta]_rho-derivative is irreducible}, and \ref{prop: [0,zeta]_rho rtimes sigma is SI}, which are basically the metaplectic version of \cite[Theorem 2.2, Theorem 2.3]{Atobe2022_A-packet}.
	
	\par Since the definitions and proofs in this section are parallel to \cite[\S 3]{AtobeMínguez_2023}, readers familiar with the theory may skim this section.
	
	\subsection{Partial Jacquet modules and socles}\label{subsection: partial Jacquet modules and socles}
	
	\begin{definition}
		For $\pi\in\Pi(\GL(d_\pi))$ and $\sigma\in\Pi(\GL(d_\sigma))$, we can write $$m^*(\sigma)=\pi\otimes\Jac_\pi(\sigma)+\sum_i\pi_i\otimes\sigma_i,$$ where $\pi_i$ are irreducible representations of $\GL(d_i)$ which are not isomorphic to $\pi$. We call $\Jac_\pi(\sigma)$ the (left) partial Jacquet module of $\sigma$. We say that $\sigma$ is left $\pi$-reduced if $\Jac_\pi(\sigma)=0$. We can also define the right partial Jacquet module similarly, and we will denote the right partial Jacquet module by $\Jac_\pi^\op$. In particular, when $\rho\in\Pi_{\unit,\cusp}(\GL(d_\rho))$ and $x\in\RR$, we denote $\Jac_{\rho|\cdot|^x}$ by $\Jac_{\rho,x}$.
	\end{definition}

	\begin{definition}\
		For $\pi\in\Pi(\GL(d_\pi))$ and $\sigma\in\Pi_-(\widetilde{G}_{2n})$, we can write $$\mu^*(\sigma)=\pi\otimes\Jac_\pi(\sigma)+\sum_i\pi_i\otimes\sigma_i,$$ where $\pi_i$ are irreducible representations of $\GL(d_i)$ which are not isomorphic to $\pi$. We call $\Jac_\pi(\sigma)$ the partial Jacquet module of $\sigma$. We say that $\sigma$ is $\pi$-reduced if $\Jac_\pi(\sigma)=0$. In particular, when $\rho\in\Pi_{\unit,\cusp}(\GL(d_\rho))$ and $x\in\RR$, we denote $\Jac_{\rho|\cdot|^x}$ by $\Jac_{\rho,x}$.
	\end{definition}
	\index{partial Jacquet module}
	\index{Jac partial Jacquet module@$\Jac_\pi$, $\Jac_\pi^\op$, $\Jac_{\rho,x}$}
	
	\par The following are two elementary but useful lemmas about partial Jacquet modules:
	
	\begin{lemma}\label{lemma: partial Jacquet module and embedding}
		If $\pi\in\Pi_-(\widetilde{G}_{2n})$, $\rho\in\Pi_{\unit,\cusp}(\GL(d_\rho))$, $x_1,\dots,x_n\in\RR$ and $\Jac_{\rho,x_n}\circ\Jac_{\rho,x_{n-1}}\circ\cdots\circ\Jac_{\rho,x_1}(\pi)=\sigma$. Then there exists an irreducible constituent $\sigma'$ of $\sigma$ such that there exists an inclusion $$\pi\xhookrightarrow{}\rho|\cdot|^{x_1}\times\rho|\cdot|^{x_2}\times\cdots\times\rho|\cdot|^{x_n}\rtimes\sigma'.$$
	\end{lemma}
	\begin{proof}
		This is a standard fact about Jacquet modules, which can be proved by the same argument as in \cite[Lemma 5.3]{Xu2017_cuspidal}.
	\end{proof}
	
	\begin{lemma}\label{lemma: computation of partial Jacquet module}
		If $\pi\in\mR(\GL)$, $\sigma\in\mR_-(\widetilde{G})$, $\rho\in\Pi_{\unit,\cusp}(\GL(d_\rho))$, $x\in\RR$, then $$\Jac_{\rho,x}(\pi\rtimes\sigma)=\Jac_{\rho,x}(\pi)\rtimes\sigma+\Jac_{\rho^\vee,-x}^\op(\pi)\rtimes\sigma+\pi\rtimes\Jac_{\rho,x}(\sigma).$$
	\end{lemma}
	\begin{proof}
		This is a direct consequence of the Tadi\'{c} formula (Proposition \ref{prop: Tadic formula}). For a reference, see \cite[\S 5]{Xu2017_cuspidal}
	\end{proof}
	
	\begin{definition}\label{def: socle}
		Let $\pi$ be a finite length genuine representation of $\widetilde{G}_{2n}$. We denote by $\soc(\pi)$ the largest semisimple subrepresentation of $\pi$ and call it the socle of $\pi$. We also define the cosocle $\cos(\pi)$ of $\pi$ to be the largest semisimple quotient of $\pi$. We say that $\pi$ is SI (socle irreducible) if $\soc(\pi)$ is irreducible and occurs with multiplicity one in $\JH(\pi)$.
	\end{definition}
	\index{socle, $\soc(\pi)$}
	\index{cosocle, $\cos(\pi)$}
	\index{SI}
	
	\begin{lemma}\label{lemma: criterion for SI}
		Let $\pi$ be a finite length representation of $\GL(d_\pi)$ and $\sigma$ be a finite length genuine representation of $\widetilde{G}_{2n}$. Suppose that $\pi$ and $\sigma$ are SI, and that $\soc(\pi)\otimes\soc(\sigma)$ occurs with multiplicity one in $\mu^*(\pi\rtimes\sigma)$. Then $\pi\rtimes\sigma$ is SI and $\soc(\pi\rtimes\sigma)=\soc(\soc(\pi)\rtimes\soc(\sigma))$.
	\end{lemma}
	\begin{proof}
		Suppose that $\tau$ is an irreducible subrepresentation of $\pi\rtimes\sigma$. The Frobenius reciprocity implies that $\soc(\pi)\otimes\soc(\sigma)\le\pi\otimes\sigma\le\mu^*(\tau)$ (note that $\mR_-(\widetilde{G})$ is the free abelian group generated by $\Pi_-(\widetilde{G})$, the symbol $\pi\le\tau$ means that $\tau-\pi\in\ZZ_{\ge0}\Pi_-(\widetilde{G})$). Thus $\pi\rtimes\sigma$ contains at most one irreducible subrepresentation, and the multiplicity of $\tau$ in $\JH(\pi\rtimes\sigma)$ is at most one.
	\end{proof}
	
	\begin{lemma}\label{lemma: reduced decomposition of representation}
		For $\sigma\in\Pi_-(\widetilde{G}_{2n})$, $\rho\in\Pi_{\unit,\cusp}(\GL(d_\rho))$ and $X\subset\left\{\rho|\cdot|^x:x\in\RR\right\}$, there exist $\pi\in\Pi(\GL(d))$ and $\tau\in\Pi_-(\widetilde{G}_{2n-2d})$ such that
		\begin{enumerate}
			\item[(1)] $\sigma\xhookrightarrow{}\pi\rtimes\tau$.
			\item[(2)] $\Supp(\pi)\subset X$.
			\item[(3)] $\tau$ is $\rho|\cdot|^x$-reduced for all $\rho|\cdot|^x\in X$.
		\end{enumerate}
		Moreover, write $X^\vee=\left\{\rho^\vee|\cdot|^{-x}:\rho|\cdot|^x\in X\right\}$. If $X\cap X^\vee=\emptyset$, then the pair $(\pi,\tau)$ is unique, and the parabolic induction $\pi\rtimes\tau$ is SI.
	\end{lemma}
	\begin{proof}
		The existence part can be proved using the same argument as in \cite[Lemma 2.1.2]{Jantzen2007_jacquet}. We only prove the uniqueness part of the proposition. Suppose that $(\pi',\tau')$ is another pair satisfying the given conditions, then $\Jac_\pi(\pi'\rtimes\tau')\neq0$. By Proposition \ref{prop: Tadic formula}, there exist $\lambda_1\otimes\tau_1\le\mu^*(\tau')$, $\pi_1\otimes\pi_2\le m^*(\pi)$, $\pi_3\otimes\pi_4\le m^*(\pi_1)$ such that $\pi\le\pi_2^\vee\times\pi_3\times\lambda_1$. Since $\tau'$ is $X$-reduced, we have $\lambda_1=1$ (the ``$1$" here means the multiplicative identity of $\mR(\GL)$). Also, since $X\cap X^\vee=\emptyset$, we have $\pi_2=1$. In conclusion, we have $\pi=\pi_3$, and thus $\Jac_\pi(\pi')\neq0$. But we also have $\Jac_{\pi'}(\pi)\neq0$ by symmetry, hence $\pi=\pi'$. Note that the above computation also implies that $\Jac_{\pi}(\pi\rtimes\tau)=\tau$. Hence $\sigma=\Jac_{\pi}(\sigma)=\Jac_{\pi'}(\sigma)=\sigma'$. Finally, it follows from Lemma \ref{lemma: criterion for SI} that $\pi\rtimes\tau$ is SI.
	\end{proof}
	
	\subsection{The non-self-dual case}
	
	\begin{definition}\label{def: rho-derivative}
		For $\sigma\in\Pi_-(\widetilde{G}_{2n})$, $\rho\in\Pi_{\unit,\cusp}(\GL(d_\rho))$ and $x\in\RR$, the $k$-th $\rho|\cdot|^x$-derivative of $\sigma$ is defined to be $$D_{\rho,x}^{(k)}(\sigma):=\frac{1}{k!}\underbrace{\Jac_{\rho,x}\circ\cdots\circ\Jac_{\rho,x}}_{k\text{ times}}(\sigma).$$ When $D_{\rho,x}^{(k)}(\sigma)\neq 0$ but $D_{\rho,x}^{(k+1)}(\sigma)=0$, we call $D_{\rho,x}^{(k)}(\sigma)$ the highest $\rho|\cdot|^x$-derivative of $\sigma$. Notice that $\Jac_{\min}((\rho|\cdot|^x)^k)=k!\underbrace{\rho|\cdot|^x\otimes\cdots\otimes\rho|\cdot|^x}_{k\text{ times}}$. Thus, it is not hard to see that $D_{\rho,x}^{(k)}(\sigma)=\Jac_{(\rho|\cdot|^x)^k}(\sigma)$. In particular, $D_{\rho,x}^{(k)}(\sigma)$ is always a representation.
	\end{definition}
	\index{$\rho$-derivative}
	\index{derivative rho@$D_{\rho,x}^{(k)}$}
	
	\begin{proposition}\label{prop: highest rho-derivative is irreducible in non-self-dual case}
		Let $\sigma\in\Pi_-(\widetilde{G}_{2n})$, $\rho\in\Pi_{\unit,\cusp}(\GL(d_\rho))$ and $x\in\RR$. Suppose that $D_{\rho,x}^{(k)}(\sigma)$ is the highest derivative of $\sigma$. Then the following hold:
		\begin{enumerate}
			\item[(1)] If $\rho|\cdot|^x$ is not self-dual, then $D_{\rho,x}^{(k)}(\sigma)$ is irreducible, $\rho^k\rtimes D_{\rho,x}^{(k)}(\sigma)$ is SI and $\sigma=\soc(\rho^k\rtimes D_{\rho,x}^{(k)}(\sigma))$.
			\item[(2)] If $\rho|\cdot|^x$ is self-dual, then $D_{\rho,x}^{(k)}(\sigma)=m\sigma_0$, where $\sigma_0$ is irreducible and $m\in\ZZ_{\ge1}$. In this case, $\sigma$ is still a subrepresentation of $\rho^k\rtimes\sigma_0$.
		\end{enumerate}
	\end{proposition}
	\begin{proof}
		By Lemma \ref{lemma: reduced decomposition of representation}, there exists an irreducible genuine representation $\sigma_0$ such that $\sigma\xhookrightarrow{}(\rho|\cdot|^x)^r\rtimes\sigma_0$ and $\Jac_{\rho,x}(\sigma_0)=0$. By Lemma \ref{lemma: computation of partial Jacquet module}, we have 
		$$\Jac_{\rho,x}((\rho|\cdot|^x)^r\rtimes\sigma_0)=\begin{cases}
			r(\rho|\cdot|^x)^{r-1}\rtimes\sigma_0&\text{if }\rho|\cdot|^x\text{ is not self-dual,}\\
			2r(\rho|\cdot|^x)^{r-1}\rtimes\sigma_0&\text{if }\rho|\cdot|^x\text{ is self-dual.}\\
		\end{cases}.$$ Now the proposition follows from the above calculation and Lemma \ref{lemma: reduced decomposition of representation}.
	\end{proof}
	
	\begin{proposition}\label{prop: parabolic induction is SI in non-self-dual case}
		Let $\sigma\in\Pi_-(\widetilde{G}_{2n})$, $\rho\in\Pi_{\unit,\cusp}(\GL(d_\rho))$ and $x\in\RR$. Suppose that $\rho|\cdot|^x$ is not self-dual. Then $(\rho|\cdot|^x)^r\rtimes\sigma$ is SI for all $r\in\ZZ_{\ge0}$.
	\end{proposition}
	\begin{proof}
		Let $\sigma_0=D_{\rho,x}^{(k)}(\sigma)$ be the highest derivative. Then Proposition \ref{prop: highest rho-derivative is irreducible in non-self-dual case} implies that $\sigma_0$ is irreducible and $(\rho|\cdot|^x)^r\rtimes\sigma\xhookrightarrow{}(\rho|\cdot|^x)^{k+r}\rtimes\sigma_0$. By Lemma \ref{lemma: reduced decomposition of representation}, we know that $(\rho|\cdot|^x)^{k+r}\rtimes\sigma_0$ is SI. Thus $(\rho|\cdot|^x)^r\rtimes\sigma$ is also SI.
	\end{proof}
	
	\begin{definition}\label{def: socle non-self-dual case}
		Let $\sigma\in\Pi_-(\widetilde{G}_{2n})$, $\rho\in\Pi_{\unit,\cusp}(\GL(d_\rho))$ and $x\in\RR$. Suppose that $\rho|\cdot|^x$ is not self-dual. We define $$S_{\rho,x}^{(r)}(\sigma):=\soc((\rho|\cdot|^x)^r\rtimes\sigma).$$ By Proposition \ref{prop: parabolic induction is SI in non-self-dual case}, we know that $S_{\rho,x}^{(r)}(\sigma)$ is irreducible and $S_{\rho,x}^{(r)}(\sigma)=\underbrace{S_{\rho,x}^{(1)}\circ\cdots S_{\rho,x}^{(1)}}_{k\text{ times}}(\sigma)$. It follows from Proposition \ref{prop: highest rho-derivative is irreducible in non-self-dual case} that, if $D_{\rho,x}^{(k)}(\sigma)$ is the highest derivative, then we have $\sigma=S_{\rho,x}^{(k)}\circ D_{\rho,x}^{(k)}(\sigma)$.
	\end{definition}
	\index{socle functor rho@$S_{\rho,x}^{(r)}$}
	
	\subsection{Self-dual case}
	
	\par When $\rho$ is self-dual, the highest $\rho$-derivative of a genuine irreducible representation may not be irreducible and $\rho^r\rtimes\sigma$ may not be SI. For example, suppose that $\rho$ is of symplectic type, $\phi\in\Phi_\bdd(\widetilde{G})$ and $\sigma\in\Pi_\phi$. Then Corollary \ref{coro: strengthening of proposition 4.5.3 in ArthurMp - tempered case} implies that $\rho\rtimes\sigma$ is irreducible if $(\rho,1)\in\Jord(\phi)$, in which case the highest $\rho$-derivative of $\rho\rtimes\sigma$ is not irreducible; or semisimple of length two if $(\rho,1)\notin\Jord(\phi)$, in which case $\rho\rtimes\sigma$ is not SI.
	
	\par To avoid the problems mentioned above, we introduce the concept of the $[0,\zeta]_\rho$-derivative.
	
	\begin{definition}\label{def: [0,zeta]_rho-derivative}
		For $\sigma\in\Pi_-(\widetilde{G}_{2n})$, $\rho\in\Pi_{\unit,\cusp}(\GL(d_\rho))$, $\zeta\in\left\{\pm1\right\}$. The $k$-th $[0,\zeta]_\rho$-derivative of $\sigma$ is defined to be $$D_{[0,\zeta]_\rho}^{(k)}(\sigma):=\Jac_{[0,\zeta]_\rho^k}(\sigma).$$ When $D_{[0,\zeta]_\rho}^{(k)}(\sigma)\neq 0$ but $D_{[0,\zeta]_\rho}^{(k+1)}(\sigma)=0$, we call $D_{[0,\zeta]_\rho}^{(k)}(\sigma)$ the highest $[0,\zeta]_\rho$-derivative of $\sigma$. 
	\end{definition}
	\index{zero zeta derivative@$[0,\zeta]$-derivative}
	\index{derivative zero-zeta@$D_{[0,\zeta]_\rho}^{(k)}$}
	
	\begin{lemma}\label{lemma: computation of partial Jacquet module of [0,zeta]_rho rtimes sigma}
		For $\sigma\in\Pi_-(\widetilde{G}_{2n})$, $\rho\in\Pi_{\unit,\cusp}(\GL(d_\rho))$, $\zeta\in\left\{\pm1\right\}$. Suppose that $\Jac_{\rho,\zeta}(\sigma)=0$ and $\Jac_{[0,\zeta]_\rho}(\sigma)=0$. Then we have $$\Jac_{[0,\zeta]_\rho^r}([0,\zeta]_\rho^k\rtimes\sigma)=\binom{k}{r}[0,\zeta]_\rho^{k-r}\rtimes\sigma.$$
	\end{lemma}
	\begin{proof}
		Without loss of generality, we consider only the case where $\zeta=-1$. Since $[0,-1]_\rho$ commutes with $\rho$ and $\rho|\cdot|^{-1}$, we have 
		$$\begin{aligned}
			m^*([0,-1]_\rho^k)&=([0,-1]_\rho\otimes 1+\rho\otimes\rho|\cdot|^{-1}+1\otimes[0,-1]_\rho)^k\\
			&=\sum_{k_1+k_2+k_3=k}\frac{k!}{k_1!k_2!k_3!}[0,-1]_\rho^{k_1}\times\rho^{k_2}\otimes(\rho|\cdot|^{-1})^{k_2}\times[0,-1]_\rho^{k_3}.
		\end{aligned}$$
		Note that $(\rho|\cdot|^{-1})^\vee=\rho^\vee|\cdot|^1$, $([0,-1]_\rho)^\vee=[1,0]_{\rho^\vee}$, which will never be a part of $[0,-1]_\rho^r\otimes\Jac_{[0,-1]_\rho^r}(\sigma)$. Thus, we can write $$M^*([0,-1]_\rho^k)=\sum_{k_1+k_2+k_3=k}\frac{k!}{k_1!k_2!k_3!}[0,-1]_\rho^{k_1}\times\rho^{k_2}\otimes(\rho|\cdot|^{-1})^{k_2}\times[0,-1]_\rho^{k_3}+(\cdots),$$ where none of the terms in $(\cdots)$ will appear in $[0,-1]_\rho^r\otimes\Jac_{[0,-1]_\rho^r}(\sigma)$. If a term $[0,-1]_\rho^{k_1}\times\rho^{k_2}\otimes(\rho|\cdot|^{-1})^{k_2}\times[0,-1]_\rho^{k_3}$ appears in $[0,-1]_\rho^r\otimes\Jac_{[0,-1]_\rho^r}(\sigma)$, we must have $k_1=r$, $k_2=0$, $k_3=k-r$, since $\Jac_{\rho,\zeta}(\sigma)=0$ and $\Jac_{[0,\zeta]_\rho}(\sigma)=0$. This completes the proof.
	\end{proof}
	
	\begin{lemma}\label{lemma: description of highest [0,zeta]_rho-derivative}
		For $\sigma\in\Pi_-(\widetilde{G}_{2n})$, $\rho\in\Pi_{\unit,\cusp}(\GL(d_\rho))$, $\zeta\in\left\{\pm1\right\}$. Suppose that $\Jac_{\rho,\zeta}(\sigma)=0$. Let $D_\rho^{(k_0)}(\sigma)=m\sigma_0$, $D_{\rho,\zeta}^{(k_1)}(\sigma_0)=\sigma_1$ be the highest derivatives. Then we have
		\begin{enumerate}
			\item[(1)] $k_0\ge k_1$.
			\item[(2)] $D_{[0,\zeta]_\rho}^{(k_1)}(\sigma)$ is the highest derivative.
			\item[(3)] $D_{[0,\zeta]_\rho}^{(k_1)}(\sigma)\le\rho^{k_0-k_1}\rtimes\sigma_1$.
		\end{enumerate}
		In particular, when $k_0=k_1$, we have $D_{[0,\zeta]_\rho}^{(k_1)}(\sigma)=\sigma_1$. 
	\end{lemma}
	\begin{proof}
		Without loss of generality, we consider only the case where $\zeta=-1$. By Lemma \ref{lemma: partial Jacquet module and embedding}, we know that $\sigma$ is a subrepresentation of $\rho^{k_0}\times(\rho|\cdot|^{-1})^{k_1}\rtimes\sigma_1$. Thus, by \cite[Lemma 3.2]{Mœglin2002_construction}, there exists an irreducible subquotient $\pi$ of $\rho^{k_0}\times(\rho|\cdot|^{-1})^{k_1}$ such that $\sigma$ is a subrepresentation of $\pi\rtimes\sigma_1$. Since $\rho\times\rho|\cdot|^{-1}=[0,-1]_\rho+[-1,0]_\rho$ in $\mR(\GL)$ and $\Jac_{\rho,-1}(\sigma)=0$, we must have $k_0\ge k_1$ and $\pi=[0,-1]_\rho^{k_1}\times\rho^{k_0-k_1}$. Note that $\Jac_\rho(\sigma_1)=0$, because $D_\rho^{(k_0)}(\sigma)$ is the highest derivative and $\sigma$ can be embedded into $[0,-1]_\rho^{k_1}\times\rho^{k_0-k_1}\rtimes\sigma_1$. Thus we have $\Jac_{[0,-1]_\rho}(\sigma_1)=0$. Therefore it follows from Proposition \ref{prop: Tadic formula} that $\Jac_{\rho,-1}(\rho^{k_0-k_1}\rtimes\sigma_1)=0$ and $\Jac_{[0,-1]_\rho}(\rho^{k_0-k_1}\rtimes\sigma_1)=0$. Now, Lemma \ref{lemma: computation of partial Jacquet module of [0,zeta]_rho rtimes sigma} implies that $D_{[0,\zeta]_\rho}^{(k_1)}(\sigma)$ is the highest derivative and $D_{[0,\zeta]_\rho}^{(k_1)}(\sigma)\le\rho^{k_0-k_1}\rtimes\sigma_1$.
	\end{proof}
	
	\begin{proposition}\label{prop: highest [0,zeta]_rho-derivative is irreducible}
		For $\sigma\in\Pi_-(\widetilde{G}_{2n})$, $\rho\in\Pi_{\unit,\cusp}(\GL(d_\rho))$, $\zeta\in\left\{\pm1\right\}$. Suppose that $\Jac_{\rho,\zeta}(\sigma)=0$. Then the highest $[0,\zeta]_\rho$-derivative $D_{[0,\zeta]_\rho}^{(k)}(\sigma)$ is irreducible. Furthermore, we have $\Jac_{\rho,\zeta}(D_{[0,\zeta]_\rho}^{(k)}(\sigma))=0$ and $\sigma\xhookrightarrow{}[0,\zeta]_\rho^k\rtimes D_{[0,\zeta]_\rho}^{(k)}(\sigma)$.
	\end{proposition}
	\begin{proof}
		Without loss of generality, we only consider the case where $\zeta=-1$. From the proof of Lemma \ref{lemma: description of highest [0,zeta]_rho-derivative}, we see that $\sigma$ can be embedded into $[0,-1]_\rho^{k}\times\rho^{k_0-k}\rtimes\sigma_1$. Thus, by \cite[Lemma 3.2]{Mœglin2002_construction}, there exists an irreducible subquotient $\sigma'$ of $\rho^{k_0-k}\rtimes\sigma_1$ such that $\sigma\xhookrightarrow{}[0,-1]_\rho^{k}\rtimes\sigma'$. Since $\Jac_{\rho,-1}(\sigma')=0$ and $\Jac_{[0,-1]_\rho}(\sigma')=0$, Lemma \ref{lemma: computation of partial Jacquet module of [0,zeta]_rho rtimes sigma} implies that $D_{[0,\zeta]_\rho}^{(k)}([0,-1]_\rho^{k}\rtimes\sigma')=\sigma'$, hence $D_{[0,\zeta]_\rho}^{(k)}(\sigma)=\sigma'$ is irreducible and $\sigma\xhookrightarrow{}[0,\zeta]_\rho^k\rtimes D_{[0,\zeta]_\rho}^{(k)}(\sigma)$.
	\end{proof}
	
	\begin{proposition}\label{prop: [0,zeta]_rho rtimes sigma is SI}
		For $\sigma\in\Pi_-(\widetilde{G}_{2n})$, $\rho\in\Pi_{\unit,\cusp}(\GL(d_\rho))$, $\zeta\in\left\{\pm1\right\}$. Suppose that $\Jac_{\rho,\zeta}(\sigma)=0$. Then $[0,\zeta]_\rho^r\rtimes\sigma$ is SI for all $r\ge0$.
	\end{proposition}
	\begin{proof}
		By Proposition \ref{prop: highest [0,zeta]_rho-derivative is irreducible}, we may replace $\sigma$ by its highest derivative $D_{[0,\zeta]_\rho}^{(k)}(\sigma)$ and assume that $\Jac_{[0,\zeta]_\rho}(\sigma)=0$. Now, the proposition follows from Lemma \ref{lemma: computation of partial Jacquet module of [0,zeta]_rho rtimes sigma} by computing the highest derivative.
	\end{proof}
	
	\begin{definition}\label{def: socle self-dual case}
		Let $\sigma\in\Pi_-(\widetilde{G}_{2n})$, $\rho\in\Pi_{\unit,\cusp}(\GL(d_\rho))$ and $\zeta\in\left\{\pm1\right\}$. Suppose that $\Jac_{\rho,\zeta}(\sigma)=0$. We define $$S_{[0,\zeta]_\rho}^{(r)}(\sigma):=\soc([0,\zeta]_\rho^r\rtimes\sigma).$$ By Proposition \ref{prop: [0,zeta]_rho rtimes sigma is SI}, we know that $S_{[0,\zeta]_\rho}^{(r)}(\sigma)$ is irreducible and $$S_{[0,\zeta]_\rho}^{(r)}(\sigma)=\underbrace{S_{[0,\zeta]_\rho}^{(1)}\circ\cdots S_{[0,\zeta]_\rho}^{(1)}}_{k\text{ times}}(\sigma).$$ It follows from Proposition \ref{prop: highest [0,zeta]_rho-derivative is irreducible} that, if $D_{[0,\zeta]_\rho}^{(k)}(\sigma)$ is the highest derivative, then $\sigma=S_{[0,\zeta]_\rho}^{(k)}\circ D_{[0,\zeta]_\rho}^{(k)}(\sigma)$.
	\end{definition}
	\index{socle functor zero-zeta@$S_{[0,\zeta]_\rho}^{(r)}$}
	
	\section{Discrete series}\label{section: discrete series}
	
	\par Discrete series are fundamental building blocks of Atobe's construction of A-packets in \cite{Atobe2022_A-packet}. Therefore, before we start our construction, we need to study the properties of discrete series.
	
	\par In the first two subsections of this section, we will generalize the results of \cite{Xu2017_cuspidal} on discrete series to metaplectic groups. In the last subsection, we will prove the key proposition about discrete series in \cite{Mœglin2006_Aubert} for metaplectic groups. The main results of this section are Proposition \ref{prop: partial Jacquet module of discrete series}, \ref{prop: cuspidal support of discrete series}, and \ref{proposition: key proposition}.
	
	\subsection{Computation of partial Jacquet module}
	
	\par In this subsection, we compute the partial module for discrete series and parameterize the cuspidal representations of metaplectic groups.
	
	\par The main tool used in this subsection is the following lemma from \cite{Chen2024_commutationtransferaubertzelevinskiinvolution}, which describes the commutation of spectral transfer and partial Jacquet module.
	
	\begin{lemma}(\cite[Corollary 3.3]{Chen2024_commutationtransferaubertzelevinskiinvolution})\label{lemma: commutation of spectral transfer and partial Jacquet module}
		Let $\rho\in\Pi_{\unit,\cusp}(\GL(d_\rho))$, $x\in\RR$ and $\textbf{G}^!\in\mE_{\mathrm{ell}}(\widetilde{G}_{2n})$ with $G^!=\SO(2n'+1)\times\SO(2n''+1)$. Denote by $M^!_{(d_1,d_2)}$ the Levi subgroup $(\GL(d_1)\times\SO(2(n'-d_\rho)+1))\times(\GL(d_2)\times\SO(2(n'-d_2)+1))$ of $G^!$, and let $G^!_{(d_1,d_2)}$ be the $\SO$ part of $M^!_{(d_1,d_2)}$. Then $\textbf{G}^!_{(d_1,d_2)}$ is an elliptic endoscopic datum of $\widetilde{G}_{2(n-d_1-d_2)}$, and we have $$\Jac_{\rho,x}\circ\mT_{\textbf{G}^!,\widetilde{G}_{2n}}^\vee=\mT_{\textbf{G}_{(d_\rho,0)}^!,\widetilde{G}_{2n-2d_\rho}}\Jac_{\rho,x}^1+\omega_\rho(-1)\mT_{\textbf{G}_{(0,d_\rho)}^!,\widetilde{G}_{2n-2d_\rho}}\Jac_{\rho,x}^{-1}.$$ Here $\Jac_{\rho,x}^1$ and $\Jac_{\rho,x}^{-1}$ mean taking the partial Jacquet module of $G^!$ in the first and second $\SO$ factor respectively, and $\omega_\rho$ is the central character of $\rho$.
	\end{lemma}
	\index{Jac pm1@$\Jac_{\rho,x}^{\pm1}$}
	\index{omega rho central character@$\omega_\rho$}
	
	\par Using Lemma \ref{lemma: commutation of spectral transfer and partial Jacquet module}, we can compute the partial Jacquet module of the distribution $T_{\phi,s}$ (see (\ref{equation: distribution T_psi s}) for the definition of $T_{\phi,s}$).
	
	\begin{lemma}\label{lemma: partial Jacquet module of T phi s}
		Let $\phi\in\Phi_{\bdd,2}(\widetilde{G}_{2n})$, $s\in S_{\phi,2}$ (in the discrete case, we have $S_{\phi,2}=\mS_\phi$), $\rho\in\Pi_{\unit,\cusp}(\GL(d_\rho))$, and $x\in\RR$. Then the following holds:
		\begin{enumerate}
			\item[(1)] If $x=0$ or $(\rho,2x+1)\notin\Jord(\phi)$, we have $\Jac_{\rho,x}(T_{\phi,s})=0$.
			\item[(2)] If $x>0$ and $(\rho,2x+1)\in\Jord(\phi)$, let $\phi_-=\phi\oplus r(2x-1)\ominus\rho\otimes r(2x+1)$ and let $s_-\in S_{\phi_-,2}$ is defined by $s_-(\rho,2x-1)=s(\rho,2x+1)$. Then we have 
			$$\Jac_{\rho,x}(T_{\phi,s})=\begin{cases}
				T_{\phi_-,s_-}&\text{if } (\rho,2x+1)\neq (\triv,2)\text{ or }s(\rho,2x+1)=1\\
				-T_{\phi_-,s_-}&\text{if } (\rho,2x+1)= (\triv,2)\text{ and }s(\rho,2x+1)=-1
			\end{cases}$$
		\end{enumerate}
	\end{lemma}
	\begin{proof}
		By \cite[Lemma 7.3]{Xu2017_cuspidal}, when $x=0$ or $(\rho,2x+1)\notin\Jord(\phi)$, we have $\Jac_{\rho,x}^{\pm1}(S\Theta_{\phi^!}^{G^!})=0$; and when $(\rho,2x+1)\in\Jord(\phi)$, choose $(\textbf{G}^!_-,\phi_-^!)\leftrightarrow(\phi_-,s_-)$ as in (\ref{equation: basic bijection}), we have $\Jac_{\rho,x}^{s(\rho,2x+1)}(S\Theta_{\phi^!}^{G^!})=S\Theta_{\phi_-^!}^{G_-^!}$ and $\Jac_{\rho,x}^{-s(\rho,2x+1)}(S\Theta_{\phi^!}^{G^!})=0$. In conclusion, we have $$\Jac_{\rho,x}(T_{\phi,s})=\begin{cases}
		    0&\text{if }x=0\text{ or }(\rho,2x+1)\notin\Jord(\phi)\\
		    \frac{\epsilon(\phi_-^{s_-=-1})}{\epsilon(\phi^{s=-1})}T_{\phi_-,s_-}&\text{if }x>0, (\rho,2x+1)\in\Jord(\phi) \text{ and } s(\rho,2x+1)=1\\
		    \omega_\rho(-1)\frac{\epsilon(\phi_-^{s_-=-1})}{\epsilon(\phi^{s=-1})}T_{\phi_-,s_-}&\text{if }x>0, (\rho,2x+1)\in\Jord(\phi) \text{ and } s(\rho,2x+1)=-1.
		\end{cases}$$
		\par When $s(\rho,2x+1)=1$, we have $\phi_-^{s_-=-1}=\phi^{s=-1}$. Hence $\Jac_{\rho,x}(T_{\phi,s})=T_{\phi_-,s_-}$ in this case; when $s(\rho,2x+1)=-1$, then $\frac{\epsilon(\phi_-^{s_-=-1})}{\epsilon(\phi^{s=-1})}=\frac{\epsilon(\rho\otimes r(2x+1))}{\epsilon(\rho\otimes r(2x-1))}$. Now \cite[\S 4.1]{Li2024_arthurpacketsmetaplecticgroups} implies that $$\frac{\epsilon(\rho\otimes r(2x+1))}{\epsilon(\rho\otimes r(2x-1))}=\begin{cases}
			\omega_\rho(-1)(-\rho(\Frob))^2&\text{if }\rho\text{ is an unramified character and }2x+1>1\\
			\omega_\rho(-1)(-\rho(\Frob))^1&\text{if }\rho\text{ is an unramified character and }2x+1=1\\
			\omega_\rho(-1)&\text{otherwise}.
		\end{cases}$$ Note that $(\rho,2x+1)\in\Jord(\phi)$ implies that $\rho$ is self-dual. Thus if $\rho$ is an unramified character, we must have $\rho(\Frob)=\pm1$, and $\rho(\Frob)=1$ if and only if $\rho=\triv$. This completes the proof.
	\end{proof}
	
	\par Now, we can prove the metaplectic version of \cite[Lemma 7.3]{Xu2017_cuspidal}.
	
	\begin{proposition}\label{prop: partial Jacquet module of discrete series}
		Let $\phi\in\Phi_{\bdd,2}(\widetilde{G}_{2n})$, $\varepsilon\in\mS_\phi^\vee$, $\rho\in\Pi_{\unit,\cusp}(\GL(d_\rho))$, and $x\in\RR$. Then the following holds:
		\begin{enumerate}
			\item[(1)] If $x=0$ or $(\rho,2x+1)\notin\Jord(\phi)$, we have $\Jac_{\rho,x}\pi(\phi,\varepsilon)=0$.
			\item[(2)] If $x>\frac{1}{2}$, $(\rho,2x+1)\in\Jord(\phi)$, and $(\rho,2x-1)\notin\Jord(\phi)$, we have $\Jac_{\rho,x}\pi(\phi,\varepsilon)=\pi(\phi_-,\varepsilon_-)$, where $\phi_-=\phi\oplus \rho\otimes r(2x-1)\ominus \rho\otimes r(2x+1)$ (the symbol $\ominus$ means delete the $\rho\otimes r(2x+1)$ from the parameter) and $\varepsilon_-$ is defined by $\varepsilon_-(\rho,2x-1)=\varepsilon(\rho,2x+1)$.
			\item[(3)] If $x>\frac{1}{2}$, $(\rho,2x+1)\in\Jord(\phi)$, and $(\rho,2x-1)\in\Jord(\phi)$, we have 
			$$\Jac_{\rho,x}\pi(\phi,\varepsilon)=\begin{cases}
				\pi(\phi_-,\varepsilon_-)&\text{if }\varepsilon(\rho,2x-1)\varepsilon(\rho,2x+1)=1\\
				0&\text{otherwise}.
			\end{cases}$$ Here $\phi_-=\phi\oplus \rho\otimes r(2x-1)\ominus \rho\otimes r(2x+1)$ and $\varepsilon_-$ is the restriction of $\varepsilon$ on $\Jord(\phi_-)$.
			\item[(4)] If $x=\frac{1}{2}$, $(\rho,2x+1)\in\Jord(\phi)$ and $\rho\neq\triv$, then we have
			$$\Jac_{\rho,x}(\pi(\phi,\varepsilon))=\begin{cases}
				\pi(\phi_-,\varepsilon_-)&\text{if }\varepsilon(\rho,2)=1\\
				0&\text{otherwise}.
			\end{cases}$$
			\item[(5)] If $x=\frac{1}{2}$, $(\rho,2x+1)\in\Jord(\phi)$ and $\rho=\triv$, then we have
			$$\Jac_{\rho,x}(\pi(\phi,\varepsilon))=\begin{cases}
			\pi(\phi_-,\varepsilon_-)&\text{if }\varepsilon(\triv,2)=-1\\
			0&\text{otherwise}.
			\end{cases}$$
		\end{enumerate}
	\end{proposition}
	\begin{proof}
		$(1)$ is a direct consequence of Lemma \ref{lemma: partial Jacquet module of T phi s}. For $(2)$, the map $s\mapsto s_-$ in Lemma \ref{lemma: partial Jacquet module of T phi s} gives a bijection $\mS_\phi\rightarrow\mS_{\phi_-}$ with $\varepsilon(s)=\varepsilon_-(s_-)$. Thus we have $$\begin{aligned}
			\Jac_{\rho,x}(\pi(\phi,\varepsilon))&=|\mS_\phi|^{-1}\sum_{s\in\mS_\phi}\varepsilon(s)T_{\phi_-,s_-}\\
			&=|\mS_{\phi_-}|^{-1}\sum_{s_-\in\mS_{\phi_-}}\varepsilon(s_-)T_{\phi_-,s_-}\\
			&=\pi(\phi_-,\varepsilon_-).
		\end{aligned}$$
		\par For $(3)$, the map $s\mapsto s_-$ Lemma \ref{lemma: partial Jacquet module of T phi s} induces a map $\mS_\phi\rightarrow\mS_{\phi_-},\underline{s}\mapsto\underline{s}_-$, where $\underline{s}_-(\rho,2x-1)=\underline{s}(\rho,2x-1)\underline{s}(\rho,2x+1)$. Let $\underline{s}_0\in\mS_\phi$ be defined by $\underline{s}_0(\rho',a)=-1$ if and only if $(\rho',a)=(\rho,2x\pm1)$. Then we have an exact sequence
		$$\xymatrix{
			0\ar[r]&\left<\underline{s}_0\right>\ar[r]&\mS_\phi\ar[r]^-{\underline{s}\mapsto\underline{s}_-}&\mS_{\phi_-}\ar[r]&0.
		}$$
		
		Now, for $\varepsilon\in\mS_\phi^\vee$, we have
		
		$$\begin{aligned}
			\Jac_{\rho,x}(\pi(\phi,\varepsilon))&=|\mS_\phi|^{-1}\sum_{s\in\mS_\phi}\varepsilon(s)T_{\phi_-,\underline{s}_-}\\
			&=|\mS_\phi|^{-1}\sum_{\underline{s}_-\in\mS_{\phi_-}}\varepsilon_-(\underline{s}_-)(1+\varepsilon(s_0))T_{\phi_-,\underline{s}_-}\\
			&=\begin{cases}
				\pi(\phi_-,\varepsilon_-)&\text{if }\varepsilon(s_0)=1\\
				0&\text{if }\varepsilon(s_0)=-1.
			\end{cases}
		\end{aligned}$$ It follows directly from the definition that $\varepsilon(s_0)=\varepsilon(\rho,2x-1)\varepsilon(\rho,2x+1)$.
		\par For $(4)$ and $(5)$, let $s_0\in\mS_\phi$ be defined by $s_0(\rho',a)=-1$ if and only if $(\rho',a)=(\rho,2)$. Then, we have an exact sequence 
		$$\xymatrix{
			0\ar[r]&\left<s_0\right>\ar[r]&\mS_\phi\ar[r]^{s\mapsto s_-}&\mS_{\phi_-}\ar[r]&0.
		}$$
		
		When $\rho\neq\triv$, we can deduce $(4)$ similarly to $(3)$; when $\rho=\triv$, then Lemma \ref{lemma: partial Jacquet module of T phi s} implies that $\Jac_{\rho,x}(T_{\phi,s_-s_0})=-T_{\phi_-,s_-}$ for $s_-\in\mS_{\phi_-}$. Thus we have 
		$$\begin{aligned}
			\Jac_{\rho,x}(\pi(\phi,\varepsilon))&=|\mS_\phi|^{-1}\sum_{s_-\in\mS_{\phi_-}}\varepsilon_-(s_-)(\Jac_{\rho,x}(T_{\phi,s_-})+\varepsilon(\rho,2)\Jac_{\rho,x}(T_{\rho,s_-s_0}))\\
			&=|\mS_\phi|^{-1}\sum_{s_-\in\mS_{\phi_-}}\varepsilon_-(s_-)(1-\varepsilon(s_0))T_{\phi_-,s_-}.
		\end{aligned}$$ This completes the proof.
	\end{proof}
	
	\par As a consequence of Proposition \ref{prop: partial Jacquet module of discrete series}, we can parameterize the cuspidal representations of metaplectic groups.
	
	\begin{theorem}\label{theorem: parameter of cuspidal representation}
		Let $\phi\in\Phi_{\bdd,2}(\widetilde{G}_{2n})$ and $\varepsilon\in\mS_\phi^\vee$. Then $\pi(\phi,\varepsilon)$ is a cuspidal representation if and only if it satisfies the following conditions:
		\begin{enumerate}
			\item[(1)] If $(\rho,a)\in\Jord(\phi)$ and $a-2>0$, then $(\rho,a-2)\in\Jord(\phi)$.
			\item[(2)] If $(\rho,a),(\rho,a-2)\in\Jord(\phi)$, then $\varepsilon(\rho,a-2)\varepsilon(\rho,a)=-1$.
			\item[(3)] If $(\rho,2)\in\Jord(\phi)$, then 
			$$\varepsilon(\rho,2)=\begin{cases}
				-1&\text{if }\rho\neq\triv\\
				1&\text{if }\rho=\triv.\\
			\end{cases}$$
		\end{enumerate}
	\end{theorem}
	\begin{proof}
		Follows directly from Proposition \ref{prop: partial Jacquet module of discrete series}.
	\end{proof}
	
	\par Note that in the metaplectic case, we must treat the trivial block separately. This is a metaplectic feature that does not appear in the classical group case. By Theorem \ref{theorem: parameter of cuspidal representation}, the condition $(\mT)$ in \cite[Theorem 1.2]{Atobe2017_local} can be interpreted as a cuspidal condition.
	
	\par When $\psi$ is an elementary A-parameter, the computations involving local root numbers will be different. We do not discuss this situation in detail, but rather use the following lemma to illustrate a consequence of the differences.
	
	\begin{lemma}\label{lemma: partial Jacquet module of elementary A parameter}
		Let $\phi\in\Phi_{\bdd,2}(\widetilde{G}_{2n})$, $\varepsilon\in\mS_\phi^\vee$, $\rho\in\Pi_{\unit,\cusp}(\widetilde{G})$. Suppose that $\rho$ is of orthogonal type and $(\rho,2)\notin\Jord(\phi)$. Define $\psi=\phi\oplus\rho\otimes r(1)\otimes r(2)$ and $$\varepsilon_{\pm}(\rho',a',b')=\begin{cases}
			\varepsilon(\rho',a',b')&\text{if }(\rho',a',b')\in\Jord(\phi)\\
			\pm1&\text{otherwise.}
		\end{cases}$$
		Then $\Jac_{\rho,-\frac{1}{2}}\pi(\psi,\varepsilon_+)=\pi(\phi,\varepsilon)$ and $\Jac_{\rho,-\frac{1}{2}}\pi(\psi,\varepsilon_-)=0$.
	\end{lemma}
	\begin{proof}
		Using the same method as in Lemma \ref{lemma: partial Jacquet module of T phi s}, for $s\in S_{\psi,2}$ we have
		$$\Jac_{\rho,-\frac{1}{2}}(T_{\psi,s})=\begin{cases}
			\frac{\epsilon(\psi^{s=-1})}{\epsilon(\phi^{s=-1})}T_{\phi,s_-}&\text{if }s(\rho,1,2)=1\\
			\omega_\rho(-1)\frac{\epsilon(\psi^{s=-1})}{\epsilon(\phi^{s=-1})}T_{\phi,s_-}&\text{if }s(\rho,1,2)=-1,\\
		\end{cases}$$ where $\frac{\epsilon(\psi^{s=-1})}{\epsilon(\phi^{s=-1})}=\varepsilon(\rho\otimes r(1))^2=\omega_\rho(-1)$ when $s(\rho,1,2)=-1$. Thus, we conclude that $\Jac_{\rho,-\frac{1}{2}}(T_{\psi,s})=T_{\phi,s_-}$, which directly yields $\Jac_{\rho,-\frac{1}{2}}\pi(\psi,\varepsilon_+)=\pi(\phi,\varepsilon)$ as in Proposition \ref{prop: partial Jacquet module of discrete series}.
	\end{proof}
	
	\subsection{Cuspidal support of discrete series}
	
	\par In this subsection, we generalize \cite[Theorem 8.1]{Xu2017_cuspidal} to metaplectic groups, which is a direct consequence of Corollary \ref{coro: strengthening of proposition 4.5.3 in ArthurMp - tempered case} and Proposition \ref{prop: partial Jacquet module of discrete series}.
	
	\begin{proposition}\label{prop: cuspidal support of discrete series}
		Let $\phi\in\Phi_{\bdd,2}(\widetilde{G}_{2n})$ and $\varepsilon\in\mS_\phi^\vee$. Fix an $(\rho,a)\in\Jord(\phi)$ and denote by $a_-$ the biggest positive integer smaller than $a$ in $\Jord_\rho(\phi)$. If such an integer does not exist, we will always assume $\varepsilon(\rho,a)\varepsilon(\rho,a_-)=-1$ and we write $a_-=0$ if $a$ is even, and $a_-=-1$ if $a$ is odd. Then, the following holds:
		\begin{enumerate}
			\item[(1)] If $\varepsilon(\rho,a)\varepsilon(\rho,a_-)=-1$ and $a>a_-+2$, then $$\pi(\phi,\varepsilon)=\soc([\frac{a-1}{2},\frac{a_-+3}{2}]_\rho\rtimes\pi(\phi_-,\varepsilon_-)),$$ where $\phi_-=\phi\oplus\rho\otimes r(a_-+2)\ominus\rho\otimes r(a)$, and $\varepsilon_-$ is defined by $\varepsilon_-(\rho,a_-+2)=\varepsilon(\rho,a)$.
			\item[(2)] If $\varepsilon(\rho,a)\varepsilon(\rho,a_-)=1$, then $\pi(\phi,\varepsilon)$ is a subrepresentation of $[\frac{a-1}{2},-\frac{a_--1}{2}]_\rho\rtimes\pi(\phi_-,\varepsilon_-)$, where $\phi_-=\phi\ominus\rho\otimes r(a_-)\ominus\rho\otimes r(a)$ and $\varepsilon_-$ is the restriction of $\varepsilon$ on $\Jord(\phi_-)$. Furthermore, let $\widetilde{\varepsilon}\in\mS_\phi^\vee$ be defined as
			$$\widetilde{\varepsilon}(\rho',a')=\begin{cases}
				\varepsilon(\rho',a')&\text{if }(\rho',a')\neq(\rho,a),(\rho,a_-)\\
				-\varepsilon(\rho',a')&\text{if }(\rho',a')=(\rho,a)\text{ or }(\rho,a_-).\\
			\end{cases}$$ Then we have $$\pi(\phi,\varepsilon)\oplus\pi(\phi,\widetilde{\varepsilon})=\soc([\frac{a-1}{2},-\frac{a_--1}{2}]_\rho\rtimes\pi(\phi_-,\varepsilon_-)).$$ In particular, when $a$ is even or $a_-\neq\min\Jord_\rho(\phi)$, then $\pi(\phi,\varepsilon)$ and $\pi(\phi,\widetilde{\varepsilon})$ can be distinguished by partial Jacquet module.
			\item[(3)] If $a=\min\Jord_\rho(\phi)$ is even and $\varepsilon(\rho,a)=\begin{cases}
				1&\text{if }\rho\neq\triv\\
				-1&\text{if }\rho=\triv\\
			\end{cases}$, then $$\pi(\phi,\varepsilon)=\soc([\frac{a-1}{2},\frac{1}{2}]_\rho\rtimes\pi(\phi_-,\varepsilon_-)),$$ where $\phi_-=\phi\ominus\rho\otimes r(a)$ and $\varepsilon_-$ is the restriction of $\varepsilon$ on $\Jord(\phi_-)$.
		\end{enumerate}
	\end{proposition}
	\begin{proof}
		For $(1)$, Proposition \ref{prop: partial Jacquet module of discrete series} implies that $\Jac_{\rho,\frac{a_-+3}{2}}\circ\cdots\circ\Jac_{\rho,\frac{a-1}{2}}(\pi(\phi,\varepsilon))=\pi(\phi_-,\varepsilon_-)$. Then, by Lemma \ref{lemma: partial Jacquet module and embedding}, $\pi(\phi,\varepsilon)$ can be embedded into $\rho|\cdot|^{\frac{a_-+3}{2}}\times\cdots\times\rho|\cdot|^{\frac{a-1}{2}}\rtimes\pi(\phi_-,\varepsilon_-)$. It follows from Lemma \ref{lemma: reduced decomposition of representation} that $\rho|\cdot|^{\frac{a_-+3}{2}}\times\cdots\times\rho|\cdot|^{\frac{a-1}{2}}\rtimes\pi(\phi_-,\varepsilon_-)$ is SI. Hence we have $\pi(\phi,\varepsilon)=\soc([\frac{a-1}{2},\frac{a_-+3}{2}]_\rho\rtimes\pi(\phi_-,\varepsilon_-))$. It is not hard to see that $(3)$ can be proved by the same method.
		\par For $(2)$, let $\phi'=\phi\oplus\rho\otimes r(a_-)\ominus\rho\otimes r(a)$ and let $\varepsilon_{\pm}$ be defined as
		$$\varepsilon_{\pm}(\rho',a')=\begin{cases}
		    \varepsilon(\rho',a')&\text{if }(\rho',a')\neq(\rho,a_-)\\
		    \pm1&\text{if }(\rho',a')=(\rho,a_-).\\
		\end{cases}$$ 
		Then, by Proposition \ref{prop: partial Jacquet module of discrete series}, there exists an $\zeta\in\left\{\pm 1\right\}$ such that $\pi(\phi,\varepsilon)=\soc([\frac{a-1}{2},\frac{a_-+1}{2}]_\rho\rtimes\pi(\phi',\varepsilon_\zeta))$. On the other hand, Corollary \ref{coro: strengthening of proposition 4.5.3 in ArthurMp - tempered case} implies that $[\frac{a_--1}{2},-\frac{a_--1}{2}]_\rho\rtimes\pi(\phi_-,\varepsilon_-)=\pi(\phi',\varepsilon_+)\oplus\pi(\phi',\varepsilon_-)$. By combining these two facts, we conclude $(2)$.
	\end{proof}
	
	\subsection{The key proposition}\label{section: the key proposition}
	
	\par We start this subsection with some definitions following \cite{Mœglin2006_Aubert}.
	
	\begin{definition}\label{def: a_rho phi epsilon}
		Let $\phi\in\Phi_{\bdd,2}(\widetilde{G}_{2n})$, $\varepsilon\in\mS_\phi^\vee$ and $\rho\in\Pi_{\unit,\cusp}(\GL(d_\rho))$. We define $b_{\rho,\phi,\varepsilon}$ to be the biggest integer $b$ in $\Jord_\rho(\phi)$ satisfying the following conditions:
		\begin{enumerate}
			\item[(1)] If $a\in\Jord_\rho(\phi)$ with $a\le b$, then $a-2\in\Jord_\rho(\phi)$ whenever $a-2>0$.
			\item[(2)] If $a, a-2\in\Jord_\rho(\phi)$ and $a\le b$, then $\varepsilon(\rho,a-2)\varepsilon(\rho,a)=-1$.
			\item[(3)] If $2\in\Jord_\rho(\phi)$ and $2\le b$, then $$\varepsilon(\rho,2)=\begin{cases}
				-1&\text{if }\rho\neq\triv\\
				1&\text{if }\rho=\triv.
			\end{cases}$$
		\end{enumerate}
		When such an integer $b$ does not exist, we set $b_{\rho,\phi,\varepsilon}=0$ (resp. $b_{\rho,\phi,\varepsilon}=-1$) if $\rho$ is of orthogonal type (resp. of symplectic type). Further, we denote by $a_{\rho,\phi,\varepsilon}$ the smallest integer $a\in\Jord_\rho(\phi)$ such that $a>b_{\rho,\phi,\varepsilon}$. If such $a$ does not exist, we set $a_{\rho,\phi,\varepsilon}=\infty$. By Proposition \ref{prop: partial Jacquet module of discrete series}, we know that $\pi(\phi,\varepsilon)$ is $\rho$-cuspidal, i.e., $\Jac_{\rho,x}(\pi(\phi,\varepsilon))=0$ for all $x\in\RR$, if and only if $a_{\rho,\phi,\varepsilon}=\infty$.
	\end{definition}
	\index{a@$a_{\rho,\phi,\varepsilon}$}
	\index{b@$b_{\rho,\phi,\varepsilon}$}
	
	\par In \cite{Mœglin2006_Aubert}, M\oe glin constructs elementary A-packets. Though we will not use the concept of elementary A-packets in this article, it will be beneficial to know M\oe glin's construction in the discrete case. To be precise, M\oe glin's construction in the discrete case can be formulated by the following proposition:
	
	\begin{proposition}\label{prop: Moeglin's construction in discrete case}
		Let $\phi\in\Phi_{\bdd,2}(\widetilde{G}_{2n})$, $\varepsilon\in\mS_\phi^\vee$, $\rho\in\Pi_{\unit,\cusp}(\widetilde{G})$. Suppose that $a_{\rho,\phi,\varepsilon}<\infty$. Write $x=\frac{a_{\rho,\phi,\varepsilon}-1}{2}$ and $y=\frac{b_{\rho,\phi,\varepsilon}-1}{2}$. Then we have:
		\begin{enumerate}
			\item[(1)] If $a_{\rho,\phi,\varepsilon}>b_{\rho,\phi,\varepsilon}+2$, then $\pi(\phi,\varepsilon)=\soc(\rho|\cdot|^x\rtimes\pi(\phi',\varepsilon'))$, where $\phi'=\phi\oplus\rho\otimes r(a_{\rho,\phi,\varepsilon}-2)\ominus r(a_{\rho,\phi,\varepsilon})$ and $\varepsilon'$ is defined by $\varepsilon'(\rho,a_{\rho,\phi,\varepsilon}-2)=\varepsilon(\rho,a_{\rho,\phi,\varepsilon})$.
			\item[(2)] If $a_{\rho,\phi,\varepsilon}=b_{\rho,\phi,\varepsilon}+2$ and $b_{\rho,\phi,\varepsilon}\neq 1$, then $\pi(\phi,\varepsilon)$ is the unique subrepresentation of $[x,-y]_\rho\rtimes\pi(\phi',\varepsilon')$ satisfying $\Jac_{\rho,y}(\pi(\phi,\varepsilon))=0$. Here $\phi'=\phi\ominus\rho\otimes r(a_{\rho,\phi,\varepsilon})\ominus\rho\otimes r(b_{\rho,\phi,\varepsilon})$, and $\varepsilon'$ is the restriction of $\varepsilon$ on $\Jord(\phi')$.
			\item[(3)] If $a_{\rho,\phi,\varepsilon}=b_{\rho,\phi,\varepsilon}+2$ and $b_{\rho,\phi,\varepsilon}=1$, let $\phi_-=\phi\oplus\rho\otimes r(1)\ominus\rho\otimes r(3)$, $\phi'=\phi\ominus\rho\otimes r(1)\ominus\rho\otimes r(3)$, $\varepsilon'$ is the restriction of $\varepsilon$ on $\Jord(\phi')$, $\varepsilon_{\pm}$ are the characters in $\mS_{\phi_-}^\vee$ defined in the proof of Proposition \ref{prop: cuspidal support of discrete series}. By Corollary \ref{coro: strengthening of proposition 4.5.3 in ArthurMp - tempered case}, we know that $\rho\rtimes\pi(\phi',\varepsilon')=\pi(\phi_-,\varepsilon_+)\oplus\pi(\phi_-,\varepsilon_-)$. Then, the following holds:
			\begin{enumerate}
				\item[(3.1)] $\pi(\phi,\varepsilon)=\soc(\rho|\cdot|\rtimes\pi(\phi_-,\varepsilon_{\varepsilon(\rho,1)}))$.
				\item[(3.2)] when $|\Jord_\rho(\phi)|>2$, let $\phi''=\phi'\oplus\rho\otimes r(1)\ominus\rho\otimes r(a_{\rho,\phi',\varepsilon'})$, $x'=\frac{a_{\rho,\phi',\varepsilon'}-1}{2}$, $\zeta'=\varepsilon'(\rho,a_{\rho,\phi',\varepsilon'})$, and let $\varepsilon''$ be defined by $\varepsilon''(\rho,1)=\varepsilon'(\rho,a_{\rho,\phi',\varepsilon'})$. Define $\sigma_q=[x',0]_\rho\rtimes\pi(\phi'',\varepsilon'')$ and $\sigma_s=\soc(\rho\times[x',1]_\rho)\rtimes\pi(\phi'',\varepsilon'')$. Then $\pi(\phi_-,\varepsilon_{\zeta'})=[\sigma_q]\cap\rho\rtimes\pi(\phi',\varepsilon')$ and $\pi(\phi_-,\varepsilon_{-\zeta'})=[\sigma_s]\cap\rho\rtimes\pi(\phi',\varepsilon')$.
			\end{enumerate}
		\end{enumerate}
	\end{proposition}
	\begin{proof}
		The $(1)$, $(2)$, and $(3.1)$ are just special cases of Proposition \ref{prop: cuspidal support of discrete series}. Thus we only need to prove $(3.2)$. Let $\sigma=\rho\times[x',1]_\rho\rtimes\pi(\phi'',\varepsilon'')$, then Zelevinsky's classification for $\GL$ gives the following exact sequence
		$$\xymatrix{
			0\ar[r]&\sigma_s\ar[r]&\sigma\ar[r]&\sigma_q\ar[r]&0.
		}$$
		
		By $(3)$ of Proposition \ref{prop: cuspidal support of discrete series}, we know that $\pi(\phi',\varepsilon')=\soc([x',1]_\rho\rtimes\pi(\phi'',\varepsilon''))$. Thus $\pi(\phi_-,\varepsilon_{\pm})$ are subrepresentations of $\sigma$, because we have $\pi(\phi_-,\varepsilon_+)\oplus\pi(\phi_-,\varepsilon_-)=\rho\rtimes\pi(\phi',\varepsilon')$. By Lemma \ref{lemma: computation of partial Jacquet module}, it can be shown that $\Jac_{\rho,x}(\sigma_s)=\Jac_{\rho,x}(\sigma_q)=[x',1]_\rho\rtimes\pi(\phi'',\varepsilon'')$ and $\Jac_{\rho,x}(\pi(\phi_-,\varepsilon_{\pm}))=\pi(\phi',\varepsilon')$. Since $\pi(\phi',\varepsilon')=\soc([x',1]_\rho\rtimes\pi(\phi'',\varepsilon''))$ and $[x',1]_\rho\rtimes\pi(\phi'',\varepsilon'')$ is SI, there exist a unique $\zeta\in\left\{\pm1\right\}$, such that $\pi(\phi_-,\varepsilon_{\zeta})=[\sigma_q]\cap\rho\rtimes\pi(\phi',\varepsilon')$ and $\pi(\phi_-,\varepsilon_{-\zeta})=[\sigma_s]\cap\rho\rtimes\pi(\phi',\varepsilon')$. Hence, we only need to prove that $\zeta=\varepsilon(\rho,a_{\rho,\phi',\varepsilon'})$.
		\par When $\varepsilon(\rho,a_{\rho,\phi',\varepsilon'})=1$, it follows from Lemma \ref{lemma: computation of partial Jacquet module} that $\Jac_{\rho,1}\circ\dots\circ\Jac_{\rho,x'}(\pi(\phi_-,\varepsilon_-))=0$, hence $\pi(\phi_-,\varepsilon_-)$ is not a subrepresentation of $[x',0]_\rho\rtimes\pi(\phi'',\varepsilon'')$, which implies that $\pi(\phi_-,\varepsilon_-)\le\sigma_s$ and hence $\zeta=1$. The case of $\varepsilon(\rho,a_{\rho,\phi',\varepsilon'})=-1$ can be proved similarly.
	\end{proof}
	
	\par Before we start the proof of the key proposition, we first state and prove the three basic properties for discrete series:
	
	\begin{enumerate}
		\item[$\bullet$] (Jacquet module): If $\Jac_{\rho,x}\pi(\phi,\varepsilon)\neq0$, then $b_{\rho,\phi,\varepsilon}<2x+1\in\Jord_\rho(\phi)$.
		\item[$\bullet$] (Non-unitary irreducibility): When $x\ge\frac{1}{2}$, if $2x-1\notin\Jord_\rho(\phi)\cup\left\{0\right\}$ or $x\le\frac{b_{\rho,\phi,\varepsilon}-1}{2}$, then $\rho|\cdot|^x\rtimes\pi(\phi,\varepsilon)$ is irreducible.
		\item[$\bullet$] (Unitary reducibility): When $\rho$ is of symplectic type, if $1\in\Jord_\rho(\phi)$, then $\rho\rtimes\pi(\phi,\varepsilon)$ is irreducible; otherwise $\rho\rtimes\pi(\phi,\varepsilon)$ is a semisimple, multiplicity free, length two representation. Furthermore, let $\tau$ be an irreducible subrepresentation of $\rho\rtimes\pi(\phi,\varepsilon)$, then $\rho^r\rtimes\pi(\phi,\varepsilon)$ is irreducible for all $r\in\ZZ_{\ge1}$. 
	\end{enumerate}
	
	\par The Jacquet module property and unitary reducibility property for discrete series are consequences of Corollary \ref{coro: strengthening of proposition 4.5.3 in ArthurMp - tempered case} and Proposition \ref{prop: partial Jacquet module of discrete series}. However, the proof of the non-unitary irreducibility property is not easy even in the discrete case. Fortunately, since M\oe glin has already established the property for representations of classical groups in elementary A-packets, we can avoid the difficulties by using theta correspondence.
	
	\begin{proposition}\label{prop: non-unitary irreducibility for discrete series}
		Let $\phi\in\Phi_{\bdd,2}(\widetilde{G}_{2n})$, $\varepsilon\in\mS_\phi^\vee$, $\rho\in\Pi_{\unit,\cusp}(\widetilde{G})$, $x\in\RR$. When $x\ge\frac{1}{2}$, if $2x-1\notin\Jord_\rho(\phi)\cup\left\{0\right\}$ or $x\le\frac{b_{\rho,\phi,\varepsilon}-1}{2}$, then $\rho|\cdot|^x\rtimes\pi(\phi,\varepsilon)$ is irreducible.
	\end{proposition}
	\begin{proof}
		Write $\sigma_{\pm}=\soc(\rho|\cdot|^{\pm x}\rtimes\pi(\phi,\varepsilon))$. By Proposition \ref{prop: consequence of MVW-involution} and \ref{prop: parabolic induction is SI in non-self-dual case}, we only need to prove that $\sigma_+=\sigma_-$. For $\alpha\gg0$, Lemma \ref{lemma: compatibility of theta lifts and socle} implies that $\theta_{-\alpha}(\sigma_{\pm})=\soc(\rho|\cdot|^{\pm x}\rtimes\theta_{-\alpha}(\pi(\phi,\varepsilon)))$. By Proposition \ref{prop: adams conjecture - discrete case} below, we have $\theta_{-\alpha}(\pi(\phi,\varepsilon))=\pi_\M(\phi_\alpha,\varepsilon^\M_\alpha)$, where $\phi_\alpha$ is an elementary A-parameter. Now, the non-unitary irreducibility of $\pi(\phi_\alpha,\varepsilon_\alpha)$ implies that $\theta_{-\alpha}(\sigma_+)=\theta_{-\alpha}(\sigma_-)$. Thus, the proposition follows from Howe's duality (Theorem \ref{theorem: Howe's duality}).
	\end{proof}
	
	\par Now, we can state and prove the key proposition for discrete series.
	
	\begin{proposition}(key proposition)\label{proposition: key proposition}
		Let $\phi\in\Phi_{\bdd,2}(\widetilde{G}_{2n})$, $\varepsilon\in\mS_\phi^\vee$, and a self-dual $\rho\in\Pi_{\unit,\cusp}(\GL(d_\rho))$. Suppose that $\mE$ is a multi-set of real numbers such that $|x|<\frac{a_{\rho,\phi,\varepsilon}-1}{2}$ for all $x\in\mE$. Then, for all irreducible subquotient $\tau$ of $\times_{x\in\mE}\rho|\cdot|^x\rtimes\pi(\phi,\varepsilon)$, there exists a totally ordered multi-set $\mE'$ of real numbers, such that $$\mE\cup-\mE=\mE'\cup-\mE'$$ and there exists an embedding $\tau\xhookrightarrow{}\times_{x\in\mE'}\rho|\cdot|^x\rtimes\pi(\phi,\varepsilon)$.
	\end{proposition}
	\begin{proof}
		We will prove this theorem by induction on $|\mE|$ and $a_{\rho,\phi,\varepsilon}$. When $|\mE|=0$, the proposition is obvious. When $a_{\rho,\phi,\varepsilon}=\infty$, then $\pi(\phi,\varepsilon)$ is $\rho$-cuspidal. Consider the cuspidal support of $\tau$, there exists some $x_0\in\RR$ such that $\Jac_{\rho,x_0}(\tau)\neq0$. However, the $\rho$-cuspidal condition implies that $x_0\in\mE\cup-\mE$. Thus, there exists an irreducible subquotient $\tau'$ of $\times_{x\in\mE-\left\{x_0\text{ or }-x_0\right\}}\rho|\cdot|\rtimes\pi(\phi,\varepsilon)$ such that $\tau$ is a subrepresentation of $\rho|\cdot|^{x_0}\rtimes\tau'$, which completes the proof in the $\rho$-cuspidal case.
		
		\par Now, we assume the proposition holds for all triples $(\mE_0,\phi_0,\varepsilon_0)$ such that either $a_{\rho,\phi_0,\varepsilon_0}>a_{\rho,\phi,\varepsilon}$ or $a_{\rho,\phi_0,\varepsilon_0}=a_{\rho,\phi,\varepsilon}$ with $|\mE_0|<|\mE|$.
		
		\par For $x_0\in\mE$, we have $|x_0|<\frac{a_{\rho,\phi,\varepsilon}-1}{2}$. Therefore, by Lemma \ref{lemma: computation of partial Jacquet module}, $\Jac_{\rho,x_0}(\times_{x\in\mE}\rho|\cdot|^x\rtimes\pi(\phi,\varepsilon))$ is a multiple of $\times_{x\in\mE-\left\{x_0\right\}}\rho|\cdot|^x\rtimes\pi(\phi,\varepsilon)$. Thus, if there exist some $x_0\in\mE\cup-\mE$ such that $\Jac_{\rho,x_0}(\tau)\neq 0$, then there exists an irreducible subquotient $\tau'$ of $\times_{x\in\mE-\left\{x_0\text{ or }-x_0\right\}}\rho|\cdot|^x\rtimes\pi(\phi,\varepsilon)$ such that $\tau\xhookrightarrow{}\rho|\cdot|^{x_0}\rtimes\tau'$. But we know that the proposition is true for $\tau'$ by the induction hypothesis, which completes the proof. Hence we may assume that $\Jac_{\rho,x}(\tau)=0$ for all $x\in\mE\cup-\mE$.
		
		\par Write $a=a_{\rho,\phi,\varepsilon}$, $b=b_{\rho,\phi,\varepsilon}$. We first consider the case when $\varepsilon(\rho,a)\neq\varepsilon(\rho,b)$. From the definitions of $a$, $b$, we must have $a>b+2$ in this case. Let $\phi_1=\phi\oplus\rho\otimes r(b+2)\ominus\rho\otimes r(a)$ and let $\varepsilon_1$ be defined by $\varepsilon_1(\rho,b+2)=\varepsilon(\rho,a)$. Then, we have $b_{\rho,\phi_1,\varepsilon_1}=b+2$ and $a_{\rho,\phi_1,\varepsilon_1}>a$. Let $\mE_1=\mE\cup\left\{\frac{a-1}{2},\frac{a-3}{2},\dots,\frac{b+3}{2}\right\}$, then $\tau$ is an irreducible subquotient of $\times_{x\in\mE_1}\rho|\cdot|^x\rtimes\pi(\phi_1,\varepsilon_1)$. Hence there exists a totally ordered multi-set of real numbers $\mE_1'$ such that $\mE_1\cup-\mE_1=\mE_1'\cup-\mE_1'$ and $\tau\xhookrightarrow{}\times_{x\in\mE_1'}\rho|\cdot|^x\rtimes\pi(\phi_1,\varepsilon_1)$. Furthermore, there exists an irreducible subquotient $\sigma$ of $\times_{x\in\mE_1'}\rho|\cdot|^x$ such that $\tau$ is a subrepresentation of $\sigma\rtimes\pi(\phi_1,\varepsilon_1)$.
		
		\par Since $\Jac_{\rho,x}(\tau)=0$ for all $x\in\mE\cup-\mE$, we have $\Jac_{\rho,x}(\sigma)=0$ for all $x\in\mE\cup-\mE$. By Zelevinsky's classification for $\GL$, it can be verified that the only possibility is $\sigma=[\delta\frac{a-1}{2},\delta t]_\rho$ for some $\delta\in\left\{\pm 1\right\}$ and $t<\frac{a-1}{2}$. Since $\Jac_{\rho,-\frac{a-1}{2}}(\tau)=0$, we must have $\delta=1$, hence $\sigma=[\frac{a-1}{2},t]_\rho$. If $t\neq 0$ and $|t|\neq\frac{b+3}{2}$, then $t\in\mE\cup-\mE$, and the non-unitary irreducibility implies that $\rho|\cdot|^t\rtimes\pi(\phi_1,\varepsilon_1)$ is reducible. Thus $[\frac{a-1}{2},t]_\rho\rtimes\pi(\phi_1,\varepsilon_1)$ can be embedded into $[\frac{a-1}{2},t+1]_\rho\times(\rho|\cdot|^t\rtimes\pi(\phi_1,\varepsilon_1))=[\frac{a-1}{2},t+1]_\rho\times(\rho|\cdot|^{-f}\rtimes\pi(\phi_1,\varepsilon_1))=\rho|\cdot|^{-t}\times[\frac{a-1}{2},t+1]_\rho\rtimes\pi(\phi_1,\varepsilon_1)$, which contradicts the fact that $\Jac_{\rho,-t}(\tau)=0$. In conclusion, $|t|=0$ or $|t|=\frac{b+3}{2}$.
		
		\par When $t=\frac{b+3}{2}$, the proposition trivially holds because $\mE=\emptyset$ in this case. When $t=-\frac{b+3}{2}$, then $\left\{|x|:x\in\mE\right\}=\left\{|x|:x\in[\frac{b+3}{2},-\frac{b+1}{2}]\right\}$. Thus $\tau$ is a subquotient of $\rho|\cdot|^{\frac{b+3}{2}}\times\cdots\times\rho|\cdot|^{-\frac{b+1}{2}}\rtimes\pi(\phi,\varepsilon)$. Let $T$ be a subquotient of $\rho|\cdot|^{\frac{b+3}{2}}\times\cdots\times\rho|\cdot|^{-\frac{b+1}{2}}$, such that $\tau$ is a subquotient of $T\rtimes\pi(\phi,\varepsilon)$. Note that $[\frac{a-1}{2},-\frac{b+3}{2}]\otimes\pi(\phi_1,\varepsilon_1)\le\mu^*(\tau)$. By Proposition \ref{prop: Tadic formula}, there exist $T_1\otimes T_2\le m^*(T)$, $T_3\otimes T_4\le m^*(T)$, $\lambda\otimes\sigma\le\mu^*(\pi(\phi,\varepsilon))$, such that $[\frac{a-1}{2},-\frac{b+3}{2}]\le T_2^\vee\times T_3\times\lambda$ and $\pi(\phi_1,\varepsilon_1)\le T_4\rtimes\sigma$. Note that $\Jac_{\rho,x}(\pi(\phi_1,\varepsilon_1))=0$ for every $\rho|\cdot|^x\in\Supp(T)$, we must have $T_4=1$. Hence we have $\sigma=\pi(\phi_1,\varepsilon_1)$ and $T_3=T_1$. On the other hand, since $\pi(\phi,\varepsilon)=\soc([\frac{a-1}{2},\frac{b+3}{2}]_\rho\rtimes\pi(\phi_1,\varepsilon_1))$ and $m^*([\frac{a-1}{2},\frac{b+3}{2}]_\rho)=\sum_i[\frac{a-1}{2},i]_\rho\otimes[i-1,\frac{b+3}{2}]_\rho$, we must have $\lambda=[\frac{a-1}{2},\frac{b+3}{2}]_\rho$ by considering the cuspidal support. Now, by Zelevinsky's classification, there exist only two possibilities:
		\begin{enumerate}
			\item[(1)] $T_1=[\frac{b+1}{2},-\frac{b+1}{2}]_\rho$, $T_2=\rho|\cdot|^{\frac{b+3}{2}}$.
			\item[(2)] $T_1=1$, $T_2=[\frac{b+3}{2},-\frac{b+1}{2}]_\rho$.
		\end{enumerate}
		
		\par In the first case, we have $T=\soc(T_1\times T_2)$. By non-unitary irreducibility, $T_2\rtimes\pi(\phi,\varepsilon)$ is irreducible. Note that $\cos(T_1\times T_2^\vee)=\soc(T_2\times T_1^\vee)$, we have the following exact sequence
		$$\xymatrix{
			0\ar[r]&[\frac{b+1}{2},-\frac{b+3}{2}]\rtimes\pi(\phi,\varepsilon)\ar[r]&T_1\times T_2\rtimes\pi(\phi,\varepsilon)\ar[r]&\soc(T_2^\vee\times T_1)\rtimes\pi(\phi,\varepsilon)\ar[r]&0.
		}$$
		
		By computation of Jacquet module, it can be proved that $T\rtimes\pi(\phi,\varepsilon)$ is cosocle irreducible and $\soc(T_2^\vee\times T_1)\rtimes\pi(\phi,\varepsilon)$ is SI with $\cos(T\rtimes\pi(\phi,\varepsilon))=\soc(\soc(T_2^\vee\times T_1)\rtimes\pi(\phi,\varepsilon))$. Thus, unless $\tau=\soc(\soc(T_2^\vee\times T_1)\rtimes\pi(\phi,\varepsilon))$, in which case the proposition is true for $\tau$, $\tau$ will be a subquotient of $[\frac{b+1}{2},-\frac{b+3}{2}]\rtimes\pi(\phi,\varepsilon)$, and we reduced to the case $(2)$ above.
		
		\par In case $(2)$, $\tau$ is a subquotient of $[\frac{b+1}{2},-\frac{b+3}{2}]_\rho\rtimes\pi(\phi,\varepsilon)$ and $\tau$ is a subrepresentation of $[\frac{a-1}{2},-\frac{b+3}{2}]\rtimes\pi(\phi_1,\varepsilon_1)$. Note that $\frac{b+3}{2}\in\mE$ in this case. Hence we have $\frac{b+3}{2}<\frac{a-1}{2}\Rightarrow a>b+4$.
		
		\par Consider $D=\Jac_{\rho,-\frac{b+3}{2}}\circ\cdots\circ\Jac_{\rho,\frac{a-1}{2}}$, then $D([\frac{b+1}{2},-\frac{b+3}{2}]_\rho\rtimes\pi(\phi,\varepsilon))=\pi(\phi_1,\varepsilon_1)$. Hence $\tau$ is the unique subquotient of $[\frac{b+1}{2},-\frac{b+3}{2}]_\rho\rtimes\pi(\phi,\varepsilon)$ such that $D(\tau)=\pi(\phi_1,\varepsilon_1)$.
		
		\par Let $\phi_+=\phi\oplus\rho\otimes r(b+2)\oplus\rho\otimes r(b+4)$, and let $\varepsilon_+$ be defined by $\varepsilon_+(\rho,b+2)=\varepsilon_+(\rho,b+4)=\varepsilon(\rho,a)$. Then Proposition \ref{prop: cuspidal support of discrete series} implies that $\pi(\phi_+,\varepsilon_+)$ is both subrepresentation of $[\frac{b+3}{2},-\frac{b+1}{2}]_\rho\rtimes\pi(\phi,\varepsilon)$ and $[\frac{a-1}{2},-\frac{b+3}{2}]\rtimes\pi(\phi_1,\varepsilon_1)$. Thus, we have $\tau=\pi(\phi_+,\varepsilon_+)$ and $\tau$ is a subrepresentation of $[\frac{b+3}{2},-\frac{b+1}{2}]_\rho\rtimes\pi(\phi,\varepsilon)$.
		
		\par Now, we consider the case when $\varepsilon(\rho,a)=\varepsilon(\rho,b)$. Let $\phi_1=\phi\ominus\rho\otimes r(b)\ominus\rho\otimes r(a)$, and let $\varepsilon_1$ be the restriction of $\varepsilon$ on $\Jord(\phi_1)$. Then Proposition \ref{prop: cuspidal support of discrete series} implies that $\pi(\phi,\varepsilon)$ is a subrepresentation of $[\frac{a-1}{2},-\frac{b-1}{2}]_\rho\rtimes\pi(\phi_1,\varepsilon_1)$. Let $\mE_1=\mE\cup\left\{\frac{a-1}{2},\dots,-\frac{b-1}{2}\right\}$. Then $\tau$ is a subquotient of $\times_{x\in\mE_1}\rho|\cdot|^x\rtimes\pi(\phi_1,\varepsilon_1)$. By induction hypothesis, there exists a totally ordered multi-set $\mE_1'$ of real numbers, such that $\mE_1\cup-\mE_1=\mE_1'\cup-\mE_1'$ and $\tau\xhookrightarrow{}\times_{x\in\mE_1'}\rho|\cdot|^x\rtimes\pi(\phi_1,\varepsilon_1)$.
		
		\par Let $\sigma'$ be a subquotient of $\times_{x\in\mE_1'}\rho|\cdot|^x$ such that $\tau$ is a subrepresentation of $\sigma'\rtimes\pi(\phi_1,\varepsilon_1)$. By Zelevinsky's classification, $\sigma'=\soc([d_1,f_1]_\rho\times\cdots\times[d_l,f_l]_\rho)$ with $d_i\ge f_i$ and $d_1\le d_2\le\cdots\le d_l$. Note that $\Jac_{\rho,d_1}(\tau)\neq 0$, then the Jacquet module property implies that $d_1=\frac{a-1}{2}$. Since $\frac{a-1}{2}$ is the unique maximal element in $\left\{|x|:x\in\mE_1\right\}$, we have $l=1$, $\sigma'=[\frac{a-1}{2},f]_\rho$, $f\le-\frac{b-1}{2}$. If $f=-\frac{b-1}{2}$, then $\mE=\emptyset$ and the proposition holds naturally. When $f<-\frac{b-1}{2}$, then $b<2|f|+1<a$. 
		
		\par Since $b_{\rho,\phi_1,\varepsilon_1}=b_{\rho,\phi,\varepsilon}-2$, the non-unitary irreducibility implies that $\rho|\cdot|^f\rtimes\pi(\phi_1,\varepsilon_1)$ is irreducible. Therefore $\tau$ can be embedded into $[\frac{a-1}{2},f+1]_\rho\times\rho|\cdot|^f\rtimes\pi(\phi_1,\varepsilon_1)=\rho|\cdot|^{-f}\times[\frac{a-1}{2},f+1]_\rho\rtimes\pi(\phi_1,\varepsilon_1)$, which contradict the assumption that $\Jac_{\rho,x}(\tau)=0$ for all $x\in\mE\cup-\mE$.
	\end{proof}
	\begin{remark}
		The above proof is nearly an English translation of \cite[\S 3]{Mœglin2006_Aubert} with a few simplifications and modifications. Readers can consult \cite[\S 3]{Mœglin2006_Aubert} for more details.
	\end{remark}
	
	\section{Non-negative DDR}\label{section: non-negative DDR}
	
	\par For $\psi\in\Psi(\widetilde{G}_{2n})$ and $(\rho,a,b)\in\Jord(\psi)$, we write $A=\frac{a+b}{2}-1$, $B=\frac{a-b}{2}$. Then $(\rho,a,b)$ is completely determined by the triple $(\rho,A,B)$. Thus, we can replace $(\rho,a,b)$ by $(\rho,A,B)$.
	
	\par Let $\psi\in\Psi_\gp(\widetilde{G}_{2n})$, $\rho\in\Pi_{\unit,\cusp}(\GL(d_\rho))$. If there exists a total order $\prec_\rho$ on $\Jord_\rho(\psi)$, such that 
	$$\begin{array}{cc}
		\Jord_\rho(\psi)=\left\{(\rho,A_1,B_1)\prec_\rho(\rho,A_2,B_2)\prec_\rho\cdots\prec_\rho(\rho,A_m,B_m)\right\},\\
		0\le B_1\le A_1<B_2\le A_2<\cdots<B_m\le A_m.
	\end{array}$$
	Then we say $\psi$ is a non-negative $\rho$-DDR (DDR means ``discrete diagonal restriction"). Furthermore, if $\psi$ is a non-negative $\rho$-DDR for every $\rho$, we will call $\psi$ a non-negative DDR.
	\index{DDR}
	\index{DDR!$\rho$-DDR}
	\index{DDR!non-negative}
	
	\par In this section, we fix a non-negative DDR $\psi\in\Psi_\gp(\widetilde{G}_{2n})$, and we construct $\Pi_\psi$ via extended multi-segments. This work can be viewed as a metaplectic version of \cite{Mœeglin2009_paquets}, but we will reformulate M\oe glin's result by using the language of extended multi-segments.
	
	\par The following proposition is crucial for constructing $\Pi_\psi$:
	
	\begin{proposition}\label{prop: Mp version of proposition 5.9 in Xu_moeglin}
		We fix an $(\rho,A,B)\in\Jord(\psi)$ with $A>B$. For $s\in\mS_\psi$, we have 
		$$\begin{aligned}
			T_{\psi,s}=&\bigoplus_{C\in\left(B,A\right]}(-1)^{A-C}[B,-C]_\rho\rtimes\Jac_{\rho,C}\circ\Jac_{\rho,C-1}\circ\cdots\circ\Jac_{\rho,B+2}T_{\psi_1,s_1}\\
			&\oplus(-1)^{\left[\frac{A-B+1}{2}\right]}T_{\psi_2,s_2},
		\end{aligned}$$ where 
		$$\begin{aligned}
		    \Jord(\psi_1)&=\Jord(\psi)\cup\left\{(\rho,A,B+2)\right\}-\left\{(\rho,A,B)\right\}\\ \Jord(\psi_2)&=\Jord(\psi)\cup\left\{(\rho,A,B+1),(\rho,B,B)\right\}-\left\{(\rho,A,B)\right\},
		\end{aligned}$$ and $s(\rho,A,B)=s_1(\rho,A,B+2)=s_2(\rho,A,B+1)=s_2(\rho,B,B)$. In particular, if $A=B+1$, $\psi_1$ is defined by $\Jord(\psi_1)=\Jord(\psi)-\left\{(\rho,A,B)\right\}$.
	\end{proposition}
	\begin{proof}
		Choose $(\textbf{G}^!,\psi^!)\leftrightarrow(\psi,s)$ and $(\textbf{G}_i^!,\psi_i^!)\leftrightarrow(\psi_i,s_i)$ as in (\ref{equation: basic bijection}). Note that the character $\varepsilon_\psi^{\MW/\W}$ in \cite[Definition 5.5]{Xu2017_Moeglin} is trivial when $\psi$ is a non-negative DDR. Then \cite[Proposition 5.9]{Xu2017_Moeglin} implies that $$\begin{aligned}
			S\Theta_{\psi^!}^{G^!}=&\bigoplus_{C\in\left(B,A\right]}(-1)^{A-C}[B,-C]_\rho\rtimes\Jac_{\rho,C}\circ\Jac_{\rho,C-1}\circ\cdots\circ\Jac_{\rho,B+2}S\Theta_{\psi_1^!}^{G_1^!}\\
			&\oplus(-1)^{\left[\frac{A-B+1}{2}\right]}S\Theta_{\psi_2^!}^{G_2^!}.
		\end{aligned}$$
		By \cite[\S 3.8]{Li2024_stabilizationtraceformulametaplectic}, we have $$\mT_{\textbf{G}^!,\widetilde{G}}^\vee\circ i_{\widetilde{M}_s^!}^{G^!}\circ[z_s]=i_{\widetilde{M}}^{\widetilde{G}}\circ\mT_{\textbf{M}_s^!,\widetilde{M}}^\vee$$ for a Levi subgroup $M$ of $G$ and $s\in\mE(M,\textbf{G}^!)$ (for definitions of $\textbf{M}_s$, $\mE(M,\textbf{G}^!)$ and $[z_s]$, see \cite[\S 3]{Chen2024_commutationtransferaubertzelevinskiinvolution}). Then we have
		$$\mT_{\textbf{G}^!,\widetilde{G}}([B,-C]_\rho\rtimes(\cdots))=\begin{cases}
			[B,-C]_\rho\rtimes\mT_{\textbf{G}_-^!,\widetilde{G}_-}(\cdots)&\text{if } s(\rho,A,B)=1\\
			\omega_\rho(-1)^{B+C+1}[B,-C]_\rho\rtimes\mT_{\textbf{G}_-^!,\widetilde{G}_-}(\cdots)&\text{if }s(\rho,A,B)=-1.\\
		\end{cases}$$
		
		Combining this with Lemma \ref{lemma: commutation of spectral transfer and partial Jacquet module}, we have
		$$\frac{\mT_{\textbf{G}^!,\widetilde{G}}([B,-C]_\rho\rtimes\Jac_{\rho,C}\cdots\circ\Jac_{\rho,B+2}S\Theta_{\phi_1^!}^{G_1^!})}{[B,-C]_\rho\rtimes\Jac_{\rho,C}\cdots\circ\Jac_{\rho,B+2}\mT_{\textbf{G}_1^!,\widetilde{G}_1}(S\Theta_{\phi_1^!}^{G_1^!})}=\begin{cases}
			1&\text{if }s(\rho,A,B)=1\\
			\omega_\rho(-1)^{2C}&\text{if }s(\rho,A,B)=-1.\\
		\end{cases}$$
		Thus, when $s(\rho,A,B)=1$, we have $$\begin{aligned}
			T_{\psi,s}=&\epsilon(\psi^{s=-1})\mT_{\textbf{G}^!,\widetilde{G}_{2n}}(S\Theta_{\psi^!}^{G^!})\\
			=&\bigoplus_{C\in\left(B,A\right]}\frac{\epsilon(\psi^{s=-1})}{\epsilon(\psi_1^{s_1=-1})}(-1)^{A-C}[B,-C]_\rho\rtimes\Jac_{\rho,C}\circ\Jac_{\rho,C-1}\circ\cdots\circ\Jac_{\rho,B+2}T_{\psi_1,s_1}\\
			&\oplus\frac{\epsilon(\psi^{s=-1})}{\epsilon(\psi_2^{s_2=-1})}(-1)^{\left[\frac{A-B+1}{2}\right]}T_{\psi_2,s_2}.\\
			=&\bigoplus_{C\in\left(B,A\right]}(-1)^{A-C}[B,-C]_\rho\rtimes\Jac_{\rho,C}\circ\Jac_{\rho,C-1}\circ\cdots\circ\Jac_{\rho,B+2}T_{\psi_1,s_1}\\
			&\oplus(-1)^{\left[\frac{A-B+1}{2}\right]}T_{\psi_2,s_2}.
		\end{aligned}$$ This completes the proof in the case where $s(\rho,A,B)=1$. When $s(\rho,A,B)=-1$, we need to consider the effect of local root number $\epsilon(\psi^{s=-1})$. We first compute that $$\frac{\epsilon(\psi_2^{s_2=-1})}{\epsilon(\psi^{s=-1})}=\frac{\epsilon(\rho\otimes r(a+1))^{b-1}\epsilon(\rho\otimes r(a-b+1))}{\epsilon(\rho\otimes r(a))^b}=1.$$ When $A>B+1$, we have
		$$\frac{\epsilon(\psi_1^{s_1=-1})}{\epsilon(\psi^{s=-1})}=\frac{\epsilon(\rho\otimes r(a+2))^{b-2}}{\epsilon(\rho\otimes r(a))^b}=\omega_\rho(-1)^{b-a}=\omega_\rho(-1)^{2B}.$$
		When $A=B+1$, then $a=2B+2$, $b=2$. We have 
		$$\frac{\epsilon(\psi_1^{s_1=-1})}{\epsilon(\psi^{s=-1})}=\frac{1}{\epsilon(\rho\otimes r(2B+2))^2}=\omega_\rho(-1)^{2B}.$$
		Since $B+C\in\ZZ$, we have $\omega_\rho(-1)^{2B}\omega_\rho(-1)^{2C}=1$, which completes the proof when $s(\rho,A,B)=-1$.
	\end{proof}
	
	\par In the following, we use the symbol $\pi(\psi,\varepsilon,(\rho_1,A_1,B_1;\eta_1),\dots,(\rho_l,A_l,B_l;\eta_l))$ to denote $\pi(\psi_+,\varepsilon_+)$ for simplicity, where $\Jord(\psi_+)=\Jord(\psi)\cup\left\{(\rho_1,A_1,B_1),\dots,(\rho_l,A_l,B_l)\right\}$, $\varepsilon_+(\rho_i,A_i,B_i)=\eta_i$, and $\varepsilon_+|_{\Jord(\psi)}=\varepsilon$.
	
	\begin{theorem}\label{theorem: Mp version of Theorem 7.5 in Xu_moeglin}
		For $\varepsilon\in\mS_\psi^\vee$, set $\eta_0=\varepsilon(\rho,A,B)$. Let $\psi'$ be obtained from $\psi$ by removing $(\rho,A,B)$ and let $\varepsilon'$ be the restriction of $\varepsilon$ on $\Jord(\psi')$. Then we have 
		$$\begin{aligned}
			\pi(\psi,&\varepsilon)\\
			=&\bigoplus_{C\in\left(B,A\right]}(-1)^{A-C}[B,-C]\rtimes\Jac_{\rho,C}\circ\Jac_{\rho,C-1}\circ\cdots\circ\Jac_{\rho,B+2}\pi(\psi',\varepsilon',(\rho,A,B+2;\eta_0))\\
			&\bigoplus_{\eta=\pm1}(-1)^{[\frac{A-B+1}{2}]}\eta^{A-B+1}\eta_0^{A-B}\pi(\psi',\varepsilon',(\rho,A,B+1;\eta),(\rho,B,B;\eta\eta_0)).
		\end{aligned}$$
		In particular, when $A=B+1$ and $\eta_0=-1$ (resp. $\eta_0=1$), we replace the term $\pi(\psi',\varepsilon',(\rho,A,B+2;\eta_0))$ above by $0$ (resp. $\pi(\psi',\varepsilon')$).
	\end{theorem}
	\begin{proof}
		Write $\pi(\psi_1,\varepsilon_1)=\pi(\psi',\varepsilon',(\rho,A,B+2;\eta_0))$ and $\pi(\psi_2,\varepsilon_\eta)=\pi(\psi',\varepsilon',(\rho,A,B+1;\eta),(\rho,B,B;\eta\eta_0)))$. From the computations in the proof of \cite[Lemma 7.6]{Xu2017_Moeglin}, we have $$\frac{\varepsilon_\eta(s_{\psi_2})}{\varepsilon(s_\psi)}=\eta^{A-B+1}\eta_0^{A-B}.$$ Note that $\varepsilon_\eta(\rho,A,B+1)\varepsilon_\eta(\rho,B,B)=\eta_0=\varepsilon(\rho,A,B)$. Thus $\varepsilon(s)=\varepsilon_\eta(s_2)$ holds for all $\eta\in\left\{\pm1\right\}$ and $s\in\mS_{\psi}$ (the $s_2$ here is defined by $s_2(\rho,A,B+1)=s_2(\rho,B,B)=s(\rho,A,B)$). Conversely, for $s\in\mS_{\psi_2}$, if $s(\rho,A,B+1)\neq s(\rho,B,B)$, that is, $s$ is not in the image of $\mS_{\psi}\rightarrow\mS_{\psi_2}$, $s\mapsto s_2$, then we have $\varepsilon_+(s)+\varepsilon_-(s)=0$. Therefore, we have  
		$$\begin{aligned}
			|\mS_\psi|^{-1}\sum_{s\in\mS_\psi}\varepsilon(ss_\psi)T_{\psi_2,s_2}&=|\mS_{\psi}|^{-1}\sum_{s\in\mS_\psi}\varepsilon(s_\psi)\frac{\sum_{\eta=\pm1}\varepsilon_\eta(s_2)}{2}T_{\psi_2,s_2}\\
			&=|\mS_{\psi_2}|^{-1}\sum_{\eta=\pm1}\eta^{A-B+1}\eta_0^{A-B}\sum_{s\in\mS_{\psi_2}}\varepsilon_\eta(ss_{\psi_2})T_{\psi_2,s}\\
			&=\sum_{\eta=\pm1}\eta^{A-B+1}\eta_0^{A-B}\pi(\psi,\varepsilon_\eta).
		\end{aligned}$$
		When $A>B+1$, it is easy to see that $$|\mS_\psi|^{-1}\sum_{s\in\mS_\psi}\varepsilon(ss_\psi)T_{\psi_1,s_1}=\pi(\psi_1,\varepsilon_1).$$
		When $A=B+1$, similar to the proof of Proposition \ref{prop: partial Jacquet module of discrete series}, we have $$|\mS_\psi|^{-1}\sum_{s\in\mS_\psi}\varepsilon(ss_\psi)T_{\psi_1,s_1}=\begin{cases}
			\pi(\psi_1,\varepsilon_1)&\text{if }\eta_0=1\\
			0&\text{if }\eta_0=-1.\\
		\end{cases}$$
		In conclusion, we can deduce the theorem from Proposition \ref{prop: Mp version of proposition 5.9 in Xu_moeglin}.
	\end{proof}
	
	\par The main result of this section is the following theorem:
	
	\begin{theorem}\label{theorem: construction of A-packet non-negative DDR case}
		In the set up of Theorem \ref{theorem: Mp version of Theorem 7.5 in Xu_moeglin}, we have 
		$$\begin{aligned}
			\pi(\psi,\varepsilon)=&\soc([B,-A]_\rho\rtimes\pi(\psi',\varepsilon',(\rho,A-1,B+1;\eta_0)))\\
			&\bigoplus_{\eta=\pm1\atop\eta_0=\eta^{A-B+1}(-1)^{\frac{(A-B+1)(A-B)}{2}}}\pi(\psi',\varepsilon',\cup_{B\le C\le A}(\rho,C,C;(-1)^{C-B}\eta)).
		\end{aligned}$$
		In particular, when $A=B+1$ and $\eta_0=-1$, we replace the term $\pi(\psi',\varepsilon',(\rho,A-1,B+1;\eta_0))$ by $0$.
	\end{theorem}
	
	\par Essentially, the proof of the above theorem is no different from that of the \cite[Theorem 4.1]{Mœeglin2009_paquets}. Since the proof is lengthy and technical, we defer it to \S \ref{section: proof of the theorem} below.
	
	\subsection{Extended multi-segments}\label{subsection: extended multi-segments}
	
	\par In this subsection, we will introduce the concept of extended multi-segments following \cite{Atobe2022_A-packet}. As a consequence of Theorem \ref{theorem: construction of A-packet non-negative DDR case}, we can construct $\Pi_\psi$ for non-negative DDR $\psi$ via extended multi-segments.
	
	\begin{definition}\label{def: extended segment}
		An extended segment is a triple $([A,B]_\rho,l,\eta)$, where:
		\begin{enumerate}
			\item[$\bullet$] $\rho$ is an irreducible unitary cuspidal representation of $\GL(d_\rho)$.
			\item[$\bullet$] $A, B\in\frac{1}{2}\ZZ$ and $A-B\in\ZZ_{\ge0}$.
			\item[$\bullet$] $l\in\ZZ$ with $0\le l\le \frac{b}{2}$, where $b=A-B+1$.
			\item[$\bullet$] $\eta\in\left\{\pm1\right\}$.
		\end{enumerate}
	\end{definition}
	\index{extended segment}
	\index{AB extended segment@$([A,B]_\rho,l,\eta)$}
	
	\begin{definition}\label{def: extended multi-segment}
		An extended multi-segment of $\widetilde{G}_{2n}$ is a multi-set of extended segments $$\mE=\cup_\rho\left\{([A_i,B_i]_\rho,l_i,\eta_i)\right\}_{i\in(I_\rho,\succ_\rho)}$$ such that:
		\begin{enumerate}
			\item[$\bullet$] $(I_\rho,\succ_\rho)$ is a totally ordered finite set and $\succ_\rho$ is an admissible order, i.e., if $A_i>A_j$, $B_i>B_j$ for some $i,j\in I_\rho$, then $i\succ_\rho j$. Furthermore, if $B_i<0$ for some $i\in I_\rho$, we assume that $(I_\rho,\succ_\rho)$ satisfies the property that, if $B_i>B_j$, then $i\succ_\rho j$. 
			\item[$\bullet$] $A_i+B_i\ge 0$ for all $\rho$ and $i\in I_\rho$.
			\item[$\bullet$] let $a_i=A_i+B_i+1$, $b_i=A_i-B_i+1$, then 
			\begin{enumerate}
				\item[(1)] if $\rho$ is not self-dual, then $I_\rho=\emptyset$;
				\item[(2)] if $\rho$ is of symplectic type, then $a_i+b_i\equiv 0\mod 2$;
				\item[(3)] if $\rho$ is of orthogonal type, then $a_i+b_i\equiv 1\mod 2$.
			\end{enumerate}
			\item[$\bullet$] $\sum_\rho\sum_{i\in I_\rho}d_{\rho_i}a_ib_i=2n$.
		\end{enumerate}
	\end{definition}
	\index{admissible order}
	\index{extended multi-segment}
	\index{extended multi-segment@$\mE$}
	
	\par Our definition of extended multi-segments is parallel to the definition in \cite{Atobe2022_A-packet}. The only difference is that we do not need the sign condition $\prod(-1)^{[\frac{b_i}{2}]+l_i}\eta_i^{b_i}=\epsilon_G$.
	
	\begin{definition}\label{def: associated parameter of extended multi-segment}
		Let $\mE=\cup_\rho\left\{([A_i,B_i]_\rho,l_i,\eta_i)\right\}_{i\in(I_\rho,\succ_\rho)}$ be an extended multi-segment for $\widetilde{G}_{2n}$. We define the associated enhanced A-parameter $(\psi_\mE,\varepsilon_\mE)$ (with respect to Atobe's normalization, see \S \ref{section:Atobe's normalization} below) by:
		\begin{enumerate}
			\item[$\bullet$] $\psi_\mE=\bigoplus\limits_\rho\bigoplus\limits_{i\in I_\rho}\rho\otimes r(a_i)\otimes r(b_i)$.
			\item[$\bullet$] $\varepsilon_\mE(\rho,a_i,b_i)=\prod\limits_{j\in I_\rho\atop[A_i,B_i]_\rho=[A_j,B_j]_\rho}(-1)^{\left[\frac{b_j}{2}\right]+l_j}\eta_j^{b_j}$.
		\end{enumerate}
		then $\psi_\mE\in\Psi_\gp(\widetilde{G}_n)$ and $\varepsilon_\mE\in\mS_{\psi_\mE}^\vee$.
	\end{definition}
	\index{psi enhanced A-paratemter@$\psi_\mE$}
	\index{eps enhanced A-paratemter@$\varepsilon_\mE$}
	
	\begin{lemma}\label{lemma: segments in extended multi-segments are not linked}
		For two segments $[B,-A]_\rho$ and $[B',-A']_\rho'$, if $A\ge B>A'\ge B'\ge 0$, and if $l<\frac{A-B+1}{2}$, $l'<\frac{A'-B'+1}{2}$, then the segments $[B+l,-A+l]_\rho$ and $[B'+l',-A'+l']_\rho$ are not linked.
	\end{lemma}
	\begin{proof}
		Suppose $[B+l,-A+l]_\rho$ and $[B'+l',-A'+l']_\rho$ are linked. Then we have either $B+l>B'+l'>-A+l>-A'+l'$ or $B'+l'>B+l>-A'+l'>-A+l$.
		\par In the case where $B+l>B'+l'>-A+l>-A'+l'$. We have $\frac{A-B+1}{2}> l-l'> A-A'\Rightarrow 2A'>A+B-1$. Note that $A\ge B\ge A'+1$, this gives a contradiction. Similarly, $B'+l'>B+l>-A'+l'>-A+l$ implies $A'+B'+1>2B$, which is also impossible.
	\end{proof}
	
	\begin{definition}\label{def: associated representation of NDDR extended multi-segment}
		Let $\mE=\cup_\rho\left\{([A_i,B_i]_\rho,l_i,\eta_i)\right\}_{i\in(I_\rho,\succ_\rho)}$ be an extended multi-segment for $\widetilde{G}_{2n}$, assume that $\psi_\mE$ is a non-negative DDR, or equivalently
		\begin{enumerate}
			\item[$\bullet$] for $i,j\in I_\rho$, if $i\succ_\rho j$, then $B_i>A_j$;
			\item[$\bullet$] $B_i\ge 0$ for all $i\in I_\rho$.
		\end{enumerate}
		we define \begin{equation}\label{equation: associated representation for non-negative DDR}
			\pi(\mE)=\soc\left(\bigtimes\limits_{\rho}\bigtimes\limits_{i\in I_\rho}\begin{bmatrix}B_i&\dots&B_i+l_i-1\\\vdots&&\vdots\\-A_i&\dots&-A_i+l_i-1\end{bmatrix}_\rho\rtimes\pi(\phi,\varepsilon)\right)
		\end{equation} with $$\phi=\bigoplus\limits_\rho\bigoplus\limits_{i\in I_\rho}\rho\otimes r(2(B_i+l_i)+1)\oplus\rho\otimes r(2(B_i+l_i+1)+1)\oplus\cdots\oplus\rho\otimes r(2(A_i-l_i)+1)$$ and $\varepsilon(\rho,2(B_i+l_i+k)+1)=(-1)^k\eta_i$ for $0\le k\le b_i-2l_i-1$.
	\end{definition}
	\begin{remark}
		By Lemma \ref{lemma: segments in extended multi-segments are not linked}, the segments that occur in the column of different generalized segments in the right-hand side of (\ref{equation: associated representation for non-negative DDR}) are not linked. Hence, the parabolic induction is isomorphic to a subrepresentation of a certain standard module, which implies that $\pi(\mE)$ is irreducible. 	
	\end{remark}
	\index{pi representation associated to an extended multi-segment@$\pi(\mE)$}
	
	\par When dealing with extended multi-segments and their associated representations, the following two lemmas are useful:
	
	\begin{lemma}\label{lemma: reduction step in non-negative DDR case}
		Let $\mE$ be an extended multi-segment for $\widetilde{G}_{2n}$. Suppose that $\psi_\mE$ is a non-negative DDR. We fix an $([A_i,B_i]_\rho,l_i,\eta_i)\in\mE$ with $l_i>0$. Let $$\mE^-=\mE\cup\left\{([A_i-1,B_i+1]_\rho,l_i-1,\eta_i)\right\}-\left\{([A_i,B_i]_\rho,l_i,\eta_i)\right\}.$$ Then we have $$\pi(\mE)=\soc([B_i,-A_i]_\rho\rtimes\pi(\mE^-)).$$
	\end{lemma}
	\begin{proof}
		This follows directly from Lemma \ref{lemma: segments in extended multi-segments are not linked} and the definition of $\pi(\mE)$.
	\end{proof}
	\index{extended multi-segment shift1@$\mE^-$}
	
	\begin{lemma}\label{lemma: another reduction step in non-negative DDR case}
		Let $\mE$ be an extended multi-segment for $\widetilde{G}_{2n}$. Suppose that $\psi_\mE$ is a non-negative DDR, and we fix an $([A_i,B_i]_\rho,l_i,\eta_i)\in\mE$ with $B_i>0$. Let $$\mE^+=\mE\cup\left\{([A_i-1,B_i-1]_\rho,l_i,\eta_i)\right\}-\left\{([A_i,B_i]_\rho,l_i,\eta_i)\right\}.$$ If $\psi_{\mE^+}$ is also a non-negative DDR (i.e., $A_{i-1}<B_i-1$). Then we have $$\pi(\mE)=\soc([B_i,A_i]_\rho\rtimes\pi(\mE^+)).$$
	\end{lemma}
	\begin{proof}
		Let $D=D_{\rho,A_i}\circ D_{\rho,A_i-1}\circ\cdots\circ D_{\rho,B_i}$. By the theory of derivatives, we only need to prove that $D(\pi(\mE^+))=\pi(\mE)$, which is a special case of Proposition \ref{prop: well-definedness of associated representation of extended multi-segment} below.
	\end{proof}
	\index{extended multi-segment shift2@$\mE^+$}
	
	\begin{definition}\label{def: equivalence of extended multi-segments}
		\par $\quad$
		\begin{enumerate}
			\item[(1)] Two extended segments $([A,B]_\rho,l,\eta)$ and $([A',B']_{\rho'},l',\eta')$ are equivalent if the following holds:
			\begin{enumerate}
				\item[$\bullet$] $[A,B]_\rho=[A',B']_{\rho'}$.
				\item[$\bullet$] $l=l'$.
				\item[$\bullet$] $\eta=\eta'$ whenever $l=l'<\frac{b}{2}$.
			\end{enumerate}
			\item[(2)] Similarly, two extended multi-segments $\mE=\cup_\rho\left\{([A_i,B_i]_\rho,l_i,\eta_i)\right\}_{i\in(I_\rho,\succ_\rho)}$ and $\mE'=\cup_\rho\left\{([A_j',B_j']_\rho,l_j',\eta_j')\right\}_{j\in(J_\rho,\succ'_\rho)}$ are equivalent if there exists an order-preserving bijection $I_\rho\leftrightarrow J_\rho$ for every $\rho$, such that for all corresponding pairs $i\leftrightarrow j$, the extended segments $([A_i,B_i]_\rho,l_i,\eta_i)$ and $([A_j',B_j']_\rho,l_j',\eta_j')$ are equivalent.
		\end{enumerate}
	\end{definition}
	
	\par In particular, when $\psi_\mE$ and $\psi_{\mE'}$ are non-negative DDR, it is clear from the definition that, if $\mE$ and $\mE'$ are equivalent, then $\pi(\mE)\cong\pi(\mE')$.
	
	\par The following is a consequence of Theorem \ref{theorem: construction of A-packet non-negative DDR case}.
	
	\begin{theorem}\label{theorem: A-packets can be constructed via extended multi-segments - NDDR case}
		For $\varepsilon\in\mS_\psi^\vee$, we have $\pi(\psi,\varepsilon)=\bigoplus\pi(\mE)$, where $\mE$ runs over all equivalence classes of extended multi-segments with $(\psi,\varepsilon)=(\psi_\mE,\varepsilon_\mE)$.
	\end{theorem}
	\begin{proof}
		By applying Theorem \ref{theorem: construction of A-packet non-negative DDR case} repeatedly.
	\end{proof}

	\section{Construction of Arthur packets}\label{section: construction of Arthur packets}
	
	\par In this section, we will construct $\Pi_\psi$ for general $\psi$.
	
	\subsection{Atobe's normalization}\label{section:Atobe's normalization}
	
	\par In the first two subsections, we fix a $\psi\in\Psi_\gp(\widetilde{G}_{2n})$. Let $\prec_\rho$ be an admissible order on $\Jord_\rho(\psi)$. That is, $\prec_\rho$ satisfies the following property $(\mP)$, and if $B<0$ for some $(\rho,A,B)\in\Jord(\psi)$, $\prec_\rho$ satisfies $(\mP')$.
	\index{admissible order}
	
	\begin{enumerate}
		\item[$(\mP)$] If $A>A'$ and $B>B'$, then $(\rho,A,B)\succ_\rho(\rho,A',B')$.
		\item[$(\mP')$] If $B>B'$, then $(\rho,A,B)\succ_\rho(\rho,A',B')$.
	\end{enumerate}
	
	\par Following \cite[Definition 3.5]{Atobe2022_A-packet}, we make the following definitions
	
	\begin{definition}\label{def: Z_Ato/W}
		Let $\mZ_{\Ato/\W}(\psi)$ be the set of unordered pairs $\left\{(\rho,a,b),(\rho',a',b')\right\}$ such that $(\rho,a,b),(\rho',a',b')\in\Jord(\psi)$, $\rho=\rho'$, $b\not\equiv b'\mod 2$ and 
		\begin{enumerate}
			\item[$\bullet$] $(\rho,a,b)\succ_\rho(\rho,a',b')\Longrightarrow a'>a$.
			\item[$\bullet$] $(\rho,a,b)\prec_\rho(\rho,a',b')\Longrightarrow a>a'$.
			\item[$\bullet$] $b\in2\ZZ\Longrightarrow b>b'$.
		\end{enumerate}
		For $(\rho,a,b)\in\Jord(\psi)$, let $\mZ_{\Ato/\W}(\psi)_{(\rho,a,b)}:=\left\{(\rho',a',b'):\left\{(\rho,a,b),(\rho',a',b')\right\}\in\mZ_{\Ato/\W}(\psi)\right\}$ and let $\varepsilon_\psi^{\Ato/\W}(\rho,a,b):=(-1)^{|\mZ_{\Ato/\W}(\psi)_{(\rho,a,b)}|}$.
	\end{definition}
	\index{Z Atobe's normalization@$\mZ_{\Ato/\W}(\psi)$, $\mZ_{\Ato/\W}(\psi)_{(\rho,a,b)}$}
	\index{eps Atobe's normalization@$\varepsilon_\psi^{\Ato/\W}$}

	\begin{proposition}\label{prop: properties of Atobe's normalization - classical group case}
		Suppose $\psi\in\Psi_\gp(\SO(2n+1))$. Then we have:
		\begin{enumerate}
			\item[(1)] $\varepsilon_\psi^{\Ato/\W}\in\mS_\psi^\vee$ and $\varepsilon_\psi^{\Ato/\W}(s_\psi)=(-1)^{|\mZ_{\Ato/\W}(\psi)|}$ (when we view $\psi$ as a parameter for $\SO(2n+1)$, the definition of $\mS_\psi$ is different from the definition in \S \ref{subsection: Arthur parameters}. To be more precise, we have $\mS_\psi=\pi_0(S_\psi/Z(\widetilde{G}_{2n}^\vee))$. For more details, see \cite[\S 2]{Xu2017_Moeglin}).
			\item[(2)] If we write $\pi_\Ato(\psi,\varepsilon)=\pi(\psi,\varepsilon\varepsilon_\psi^{\Ato/\W})$ for $\varepsilon\in\mS_\psi^\vee$, then $$S\Theta_{\Ato,\psi}^{\SO(2n+1)}:=\sum_{\varepsilon\in\mS_\psi^\vee}\varepsilon(s_\psi)\pi_\Ato(\psi,\varepsilon)$$ is a stable distribution on $\SO(2n+1)$. In particular, we have $S\Theta_{\Ato,\psi}^{\SO(2n+1)}=(-1)^{|\mZ_{\Ato/\W}(\psi)|}S\Theta_{\psi}^{\SO(2n+1)}$. 
		\end{enumerate} 
	\end{proposition}
	\begin{proof}
		Since $S\Theta_{\psi}^{\SO(2n+1)}=\sum_{\varepsilon\in\mS_\psi^\vee}\varepsilon(s_\psi)\pi(\psi,\varepsilon)$, $(2)$ is a direct consequence of $(1)$. For $(1)$, we have $\prod_{(\rho,a,b)\in\Jord(\psi)}\varepsilon_\psi^{\Ato/\W}(\rho,a,b)=\sum\limits_{(\rho,a,b)\in\Jord(\psi)}(-1)^{|\mZ_{\Ato/\W}(\psi)_{(\rho,a,b)}|}=(-1)^{2|\mZ_{\Ato/\W}(\psi)|}=1$. This implies that $\varepsilon_\psi^{\Ato/\W}\in\mS_\psi^\vee$. On the other hand, we have $$\varepsilon_\psi^{\Ato/\W}(s_\psi)=\prod_{(\rho,a,b)\in\Jord(\psi)\atop b\text{ is even}}(-1)^{|\mZ_{\Ato/\W}(\psi)_{(\rho,a,b)}|}=(-1)^{|\mZ_{\Ato/\W}(\psi)|}.$$
	\end{proof}
	
	\par After Proposition \ref{prop: properties of Atobe's normalization - classical group case}, the following definitions are reasonable:
	
	\begin{definition}\label{def: T psi varepsilon in Atobe's normalization}
		For $\psi\in\Psi_\gp(\widetilde{G}_{2n})$, $s\in S_{\psi,2}$, choose $(\psi,s)\leftrightarrow(\textbf{G}^!,\psi^!)$ as in (\ref{equation: basic bijection}). We define:
		\begin{enumerate}
			\item[(1)] $T_{\psi,s}^\Ato:=\epsilon(\psi^{s=-1})\cdot\mT_{\textbf{G}^!,\widetilde{G}_{2n}}(S\Theta_{\Ato,\psi^!}^{G^!})$. By Proposition \ref{prop: properties of Atobe's normalization - classical group case}, we know that $T_{\psi,s}^\Ato=(-1)^{|\mZ_{\Ato/\W}(\psi^!)|}T_{\psi,s}$, where $$\mZ_{\Ato/\W}(\psi^!)=\left\{\left\{(\rho,a,b),(\rho',a',b')\right\}\in\mZ_{\Ato/\W}(\psi):\underline{s}(\rho,a,b)\underline{s}(\rho',a',b')=1\right\}.$$ Thus $T_{\psi,s}^\Ato$ depends only on the image $\underline{s}$ of $s$ in $\mS_\psi$.
			\item[(2)] $\pi_\Ato(\psi,\varepsilon):=|\mS_\psi|^{-1}\sum_{\underline{s}\in\mS_\psi}\varepsilon(\underline{s}s_\psi)T_{\psi,\underline{s}}^\Ato.$
		\end{enumerate}
	\end{definition}
	\index{pi Atobe's normalization@$\pi_\Ato(\psi,\varepsilon)$}
	\index{T Atobe's normalization@$T_{\psi,s}^\Ato$}
	
	\begin{proposition}\label{prop: properties of Atobe's normalization - Mp case}
		For $\psi\in\Psi_\gp(\widetilde{G}_{2n})$, the following is true:
		\begin{enumerate}
			\item[(1)] $\pi_\Ato(\psi,\varepsilon)=\pi(\psi,\varepsilon\varepsilon_\psi^{\Ato/\W})$.
			\item[(2)] If $\psi$ is a non-negative DDR, then $\pi_\Ato(\psi,\varepsilon)=\pi(\psi,\varepsilon)$.
		\end{enumerate}
	\end{proposition}
	\begin{proof}
		Since $\varepsilon_\psi^{\Ato/\W}$ is trivial when $\psi$ is a non-negative DDR, we only need to prove $(1)$. For $(1)$, we have $$\begin{aligned}
			\varepsilon_\psi^{\Ato/\W}(s)&=\prod_{s(\rho,a,b)=-1}(-1)^{|\mZ_{\Ato/\W}(\psi)_{(\rho,a,b)}|}\\
			&=(-1)^{\sum_{s(\rho,a,b)=-1}|\mZ_{\Ato/\W}(\psi)_{(\rho,a,b)}|}.
		\end{aligned}$$
		To compute $\sum_{s(\rho,a,b)=-1}|\mZ_{\Ato/\W}(\psi)_{(\rho,a,b)}|$, we consider $$\begin{aligned}
			S_1&=\left\{\left\{(\rho,a,b),(\rho,a',b')\right\}\in\mZ_{\Ato/\W}(\psi):s(\rho,a,b)s(\rho',a',b')=-1\right\}\\
			S_2&=\left\{\left\{(\rho,a,b),(\rho,a',b')\right\}\in\mZ_{\Ato/\W}(\psi):s(\rho,a,b)=s(\rho',a',b')=-1\right\}.
		\end{aligned}$$ Each unordered pair in $S_1$ occurs exactly once in the sum and each unordered pair $\left\{(\rho,a,b),(\rho,a',b')\right\}$ in $S_2$ occurs both in $\mZ_{\Ato/\W}(\psi)_{(\rho,a,b)}$ and $\mZ_{\Ato/\W}(\psi)_{(\rho',a',b')}$. Thus, we have $\sum_{s(\rho,a,b)=-1}|\mZ_{\Ato/\W}(\psi)_{(\rho,a,b)}|=|S_1|+2|S_2|$. Note that $\mZ_{\Ato/\W}(\psi)=S_1\sqcup\mZ_{\Ato/\W}(\psi^!)$, we conclude that
		$$\varepsilon_\psi^{\Ato/\W}(s)=(-1)^{|\mZ_{\Ato/\W}(\psi)|-|\mZ_{\Ato/\W}(\psi^!)|}.$$
		By Proposition \ref{prop: properties of Atobe's normalization - classical group case}, we have $\varepsilon_\psi^{\Ato/\W}(s_\psi)=(-1)^{|\mZ_{\Ato/\W}(\psi)|}$. Thus $$\varepsilon_\psi^{\Ato/\W}(ss_\psi)=(-1)^{|\mZ_{\Ato/\W}(\psi)|}(-1)^{|\mZ_{\Ato/\W}(\psi)|-|\mZ_{\Ato/\W}(\psi^!)|}=(-1)^{|\mZ_{\Ato/\W}(\psi^!)|},$$ which completes the proof.
	\end{proof}
	
	\subsection{Good parity case}\label{section: good parity case}

	\par For $\rho\in\Pi_{\unit,\cusp}(\GL(d_\rho))$, write $\Jord_\rho(\psi)$ as $$\Jord_\rho(\psi)=\left\{(\rho,A_1,B_1),(\rho,A_2,B_2),\dots,(\rho,A_m,B_m)\right\}$$ with $(\rho,A_1,B_1)\prec_\rho(\rho,A_2,B_2)\prec_\rho\cdots\prec_\rho(\rho,A_m,B_m)$. Take a sequence of non-negative integers $\textbf{t}_\rho=(t_{\rho,1},t_{\rho,2},\dots,t_{\rho,m})$ such that $$0\le B_1+t_{\rho,1}\le A_1+t_{\rho,1}<B_2+t_{\rho,2}\le A_2+t_{\rho,2}<\cdots<B_m+t_{\rho,m}\le A_m+t_{\rho,m}.$$
	
	\par After taking $\textbf{t}_\rho$ for every $\rho$ such that $\Jord_\rho(\psi)\neq\emptyset$, we set $\textbf{t}=(\textbf{t}_\rho)$ and define $\psi_{\textbf{t}}$ by $\Jord(\psi_{\textbf{t}})=\left\{(\rho,A_i+t_{\rho,i},B_i+t_{\rho,i}):(\rho,A,B)\in\Jord(\psi)\right\}$. It is not hard to see that $\psi_\textbf{t}$ is a non-negative DDR. Thus $\Pi_{\psi_{\textbf{t}}}$ is already known by Theorem \ref{theorem: A-packets can be constructed via extended multi-segments - NDDR case}.
	\index{psi shifted to non-negative DDR@$\psi_{\textbf{t}}$}
	
	\par To relate the representations in $\Pi_{\psi_{\textbf{t}}}$ with those in $\Pi_\psi$, we introduce the following auxiliary symbols:
	
	\begin{enumerate}
		\item[(1)] Define $\psi_i\in\Psi_{\gp}(\widetilde{G}_{2n})$ inductively for $1\le i\le m+1$ by $\psi_i=\psi_{i+1}\oplus\rho\otimes r(a_i+2t_{\rho,i})\otimes r(b_i)\ominus \rho\otimes r(a_i)\otimes r(b_i)$ and $\psi_{m+1}=\psi$.
		\item[(2)] Define $D_{\rho,i,t}=D_{\rho,A_i+t}^{(1)}\circ D_{\rho,A_i+t-1}^{(1)}\circ\cdots\circ D_{\rho,B_i+t}^{(1)}$ (if $B_i+t\le0$, we replace $D_{\rho,1}^{(1)}\circ D_{\rho}^{(1)}$ by $D_{[0,1]_\rho}^{(1)}$) and define $D_{\rho,i}=D_{\rho,i,1}\circ D_{\rho,i,2}\circ\cdots\circ D_{\rho,i,t_{\rho,i}}$.
		\item[(3)] For $s\in\mS_{\psi_i}$, define $s_-\in\mS_{\psi_{i+1}}$ by
		$$s_-(\rho',a',b')=\begin{cases}
			s(\rho',a',b')&\text{if }(\rho',a',b')\neq(\rho,a_i,b_i)\\
			s(\rho,a_i+2t_{\rho,i},b_i)&\text{if }(\rho',a',b')=(\rho,a_i,b_i)\\
			&\text{ and }(\rho,a_i,b_i)\notin\Jord(\psi_i)\\
			s(\rho,a_i+2t_{\rho,i},b_i)s(\rho,a_i,b_i)&\text{if }(\rho',a',b')=(\rho,a_i,b_i)\\
			&\text{ and }(\rho,a_i,b_i)\in\Jord(\psi_i).\\
		\end{cases}$$
		Then $s\mapsto s_-$ defines a projection $\mS_{\psi_i}\twoheadrightarrow\mS_{\psi_{i+1}}$ and an embedding $\mS_{\psi_{i+1}}^\vee\xhookrightarrow{}\mS_{\psi_i}^\vee$. We will view $\mS_{\psi_{i+1}}^\vee$ as a subgroup of $\mS_{\psi_i}^\vee$ via this embedding. 
	\end{enumerate}
	
	\begin{lemma}\label{lemma: one step in reduction to non-negative DDR}
		For $\varepsilon\in\mS_{\psi_i}^\vee$, we have
		$$D_{\rho,i}(\pi_\Ato(\psi_i,\varepsilon))=\begin{cases}
			\pi_\Ato(\psi_{i+1},\varepsilon)&\text{if }\varepsilon\in\mS_{\psi_{i+1}}^\vee\\
			0&\text{otherwise}.
		\end{cases}$$
	\end{lemma}
	\begin{proof}
		For $s\in S_{\psi_i,2}$, choose $(\textbf{G}_i^!,\psi_i^!)\leftrightarrow(\psi_i,s)$ as in (\ref{equation: basic bijection}). By Lemma \ref{lemma: commutation of spectral transfer and partial Jacquet module} and \cite[\S 6]{Xu2017_cuspidal}, we have
		$$D_{\rho,i}(T_{\psi_i,\underline{s}}^{\Ato})=\begin{cases}
			\mT_{\textbf{G}_{i,(tb_id_\rho,0)}^!,\widetilde{G}_{2n_i-2tb_id_\rho}}(D_{\rho,i}^{1}S\Theta_{\Ato,\psi_i^!}^{G_i^!})&\text{if }s(\rho,a_i+2t_{\rho,i},b_i)=1\\
			\omega_\rho(-1)^{tb_i}\mT_{\textbf{G}_{i,(0,tb_id_\rho)}^!,\widetilde{G}_{2n_i-2tb_id_\rho}}(D_{\rho,i}^{-1}S\Theta_{\Ato,\psi_i^!}^{G_i^!})&\text{if }s(\rho,a_i+2t_{\rho,i},b_i)=-1.\\
		\end{cases}$$ Here $D_{\rho,i}^{1}$ and $D_{\rho,i}^{-1}$ mean taking the derivative in the first and second $\SO$ factor respectively. Take $(\textbf{G}^!_{i+1},\psi_{i+1}^!)\leftrightarrow(\psi_{i+1},s_-)$ as in (\ref{equation: basic bijection}). Then Atobe's construction in \cite{Atobe2022_A-packet} implies that $D_{\rho,i}S\Theta_{\Ato,\psi_i^!}^{G_i^!}=S\Theta_{\Ato,\psi_{i+1}^!}^{G_{i+1}^!}$. Thus, we have
		$$D_{\rho,i}(T_{\psi_i,\underline{s}}^{\Ato})=\begin{cases}
			T_{\psi_{i+1},\underline{s}_-}&\text{if }s(\rho,a_i+2t_{\rho,i},b_i)=1\\
			\omega_\rho(-1)^{tb_i}\frac{\epsilon(\rho\otimes r(a_i+2t_i))^{b_i}}{\epsilon(\rho\otimes r(a_i))^{b_i}}T_{\psi_{i+1},\underline{s}_-}&\text{if }s(\rho,a_i+2t_{\rho,i},b_i)=-1.
		\end{cases}$$
		Note that $\frac{\epsilon(\rho\otimes r(a_i+2t_i))^{b_i}}{\epsilon(\rho\otimes r(a_i))^{b_i}}=\omega_\rho(-1)^{tb_i}$. Thus $D_{\rho,i}(T_{\psi_i,\underline{s}}^{\Ato})=T_{\psi_{i+1},\underline{s}_-}^{\Ato}$ for all $s\in S_{\psi_i,2}$. We can deduce the lemma from this fact by the same method as in the proof of Proposition \ref{prop: partial Jacquet module of discrete series}.
	\end{proof}
	
	\par Now, we define the functor $D_{\textbf{t}}$ by $D_{\textbf{t}_\rho}=D_{\rho,m}\circ D_{\rho,m-1}\circ\dots D_{\rho,1}$ and $D_{\textbf{t}}=\circ_\rho D_{\textbf{t}_\rho}$. Then, we have the following proposition:
	\index{D derivative with respect to t@$D_{\textbf{t}}$}
	
	\begin{proposition}\label{prop: reduction to non-negative DDR}
		For $\varepsilon\in\mS_{\psi_\textbf{t}}^\vee$, we have
		$$D_{\textbf{t}}(\pi_\Ato(\psi_{\textbf{t}},\varepsilon))=\begin{cases}
			\pi_\Ato(\psi,\varepsilon)&\text{if }\varepsilon\in\mS_\psi^\vee\\
			0&\text{otherwise}.\\
		\end{cases}$$
		Furthermore, for $\pi\in\Pi_{\psi_{\textbf{t}}}$, either $D_{\textbf{t}}(\pi)=0$ or $D_{\textbf{t}}(\pi)$ is irreducible. If $D_{\textbf{t}}(\pi)\cong D_{\textbf{t}}(\pi')\neq0$, then $\pi\cong\pi'$.
	\end{proposition}
	\begin{proof}
		The first part of the proposition follows from Lemma \ref{lemma: one step in reduction to non-negative DDR}. By Lemma \ref{lemma: commutation of spectral transfer and partial Jacquet module} and \cite[\S 6]{Xu2017_cuspidal}, if $D_{\textbf{t}}(\pi)\neq0$, then $D_{\textbf{t}}(\pi)$ is the highest derivative. Thus, Proposition \ref{prop: highest rho-derivative is irreducible in non-self-dual case} and Proposition \ref{prop: highest [0,zeta]_rho-derivative is irreducible} imply that $D_{\textbf{t}}(\pi)$ is irreducible. Furthermore, there exists a socle functor $S_{\textbf{t}}$, such that $S_{\textbf{t}}D_{\textbf{t}}(\pi)\cong\pi$ whenever $D_{\textbf{t}}(\pi)\neq0$. This completes the proof.
	\end{proof}
	
	\begin{definition}\label{def: associated representation of extended multi-segments}
		Let $\mE=\cup_\rho\left\{([A_i,B_i]_\rho,l_i,\eta_i)\right\}_{i\in(I_\rho,\succ_\rho)}$ be an extended multi-segment such that $\psi_\mE=\psi$. We set $\mE_{\textbf{t}}=\cup_\rho\left\{([A_i+t_{\rho,i},B_i+t_{\rho,i}]_\rho,l_i,\eta_i)\right\}_{i\in(I_\rho,\succ_\rho)}$ and define $\pi(\mE):=D_{\textbf{t}}(\pi(\mE_{\textbf{t}}))$ to be the associated representation of $\mE$. By Theorem \ref{theorem: A-packets can be constructed via extended multi-segments - NDDR case}, we know that $\pi(\mE_{\textbf{t}})\in\Pi_{\psi_{\textbf{t}}}$. Thus, Proposition \ref{prop: reduction to non-negative DDR} implies that $\pi(\mE)$ is either zero or irreducible. Further, Proposition \ref{prop: reduction to non-negative DDR} also implies that, if $\pi(\mE)\neq0$, we have $\pi(\mE)\in\Pi_\psi$. We will prove in \S \ref{section: Consequences of the Adams conjecture} (see Proposition \ref{prop: well-definedness of associated representation of extended multi-segment}) that $\pi(\mE)$ is independent of the choice of $\textbf{t}$.
	\end{definition}
	\index{pi representation associated to an extended multi-segment@$\pi(\mE)$}
	
	\par The following theorem is the main result of this subsection:
	
	\begin{theorem}\label{theorem: A-packets can be constructed via extended multi-segments}
		For $\varepsilon\in\mS_\psi^\vee$, we have $\pi_\Ato(\psi,\varepsilon)=\bigoplus\pi(\mE)$, where $\mE$ runs over all equivalence classes of extended multi-segments with $(\psi,\varepsilon)=(\psi_\mE,\varepsilon_\mE)$.
	\end{theorem}
	\begin{proof}
		By combining Theorem \ref{theorem: A-packets can be constructed via extended multi-segments - NDDR case} and Proposition \ref{prop: reduction to non-negative DDR}.
	\end{proof}
	
	\subsection{General case}\label{section: general case}
	
	\par In this section, we fix a $\psi\in\Psi_{\gp}(\widetilde{G}_{2n})$ and decompose $\psi$ into $\psi=\psi_{\np}^\vee\oplus\psi_{\gp}\oplus\psi_{\np}$ as in (\ref{equation: reduction to good parity}). The main purpose of this subsection is to prove the following proposition:
	
	\begin{proposition}\label{prop: reduction to good parity}
		For $\pi\in\Pi_{\psi_{\gp}}$, $\tau_{\psi_{\np}}\rtimes\pi$ is irreducible ($\tau_{\psi_{\np}}$ is defined as in Proposition \ref{prop: proposition 4.5.3 in ArthurMp}). In particular, by Proposition \ref{prop: proposition 4.5.3 in ArthurMp}, we have $\Pi_{\psi}=\left\{\tau_{\psi_{\np}}\rtimes\pi:\pi\in\Pi_{\psi_{\gp}}\right\}$.
	\end{proposition}
	
	\par By Theorem \ref{theorem: A-packets can be constructed via extended multi-segments}, we write $\pi=\pi(\mE)$ with $\mE$ an extended multi-segment of $\widetilde{G}_{2n}$. Choose $\textbf{t}$ as in \S \ref{section: good parity case} so that $\pi(\mE)=D_{\textbf{t}}(\pi(\mE_{\textbf{t}}))$. Since $\Supp(\tau_{\psi_{\np}})$ is disjoint from the cuspidal lines that support $\Jord(\psi_\gp)$, we have $D_{\textbf{t}}(\tau_{\psi_{\np}}\rtimes\pi(\mE_{\textbf{t}}))=\tau_{\psi_{\np}}\rtimes D_{\textbf{t}}(\pi(\mE_{\textbf{t}}))=\tau_{\psi_{\np}}\rtimes\pi(\mE)$. Thus, by the theory of derivatives (i.e. Propositions \ref{prop: highest rho-derivative is irreducible in non-self-dual case} and \ref{prop: highest [0,zeta]_rho-derivative is irreducible}), it is sufficient to prove that $\tau_{\psi_{\np}}\rtimes\pi(\mE_{\textbf{t}})$ is irreducible. Therefore, we may assume that $\psi_{\gp}$ is a non-negative DDR.
	
	\par Write $\mE=\cup_\rho\left\{([A_i,B_i],l_i,\eta_i)\right\}_{i\in(I_\rho,\succ_\rho)}$. We make induction on $\sum_\rho\sum_{i\in I_\rho}l_i$. Suppose that $\sum_\rho\sum_{i\in I_\rho}l_i>0$. Then Lemma \ref{lemma: reduction step in non-negative DDR case} implies that $\pi(\mE)=\soc([B,-A]_\rho\rtimes\pi(\mE^-))$. Thus $\tau_{\psi_{\np}}\rtimes\pi(\mE)$ can be embedded into $[B,-A]_\rho\rtimes(\tau_{\psi_{\np}}\rtimes\pi(\mE^-))$. By induction hypothesis, $\tau_{\psi_{\np}}\rtimes\pi(\mE^-)$ is irreducible. Consider $D=\Jac_{\rho,-A}\circ\Jac_{\rho,-A+1}\circ\dots\circ\Jac_{\rho,B}$. Then $D([B,-A]_\rho\rtimes(\tau_{\psi_{\np}}\rtimes\pi(\mE^-)))=\tau_{\psi_{\np}}\rtimes\pi(\mE^-))$ is irreducible, which implies that $[B,-A]_\rho\rtimes(\tau_{\psi_{\np}}\rtimes\pi(\mE^-))$ is SI. In particular, $\tau_{\psi_{\np}}\rtimes\pi(\mE)$ is irreducible. In conclusion, we only need to consider the case when $\sum_\rho\sum_{i\in I_\rho}l_i=0$. That is, we may assume that $\pi$ is discrete.
	
	\par For $\rho\in\Pi_{\unit,\cusp}(\GL(d_\rho))$, $A,B\in\frac{1}{2}\ZZ$, $t\in\ZZ_{\ge0}$ with $A-B\in\ZZ_{\ge0}$, we write
	$$S(\rho,A,B)=\begin{bmatrix}
		B&\dots&A\\
		\vdots&&\vdots\\
		-A&\dots&-B\\
	\end{bmatrix}_\rho.$$
	Then we have $\tau_{\psi_{\np}}=\bigtimes_{(\rho,A,B)\in\Jord(\psi_{\np})}S(\rho,A,B)$.
	
	\par We fix a total order $\succ$ on $\Jord(\psi_{\np})$ satisfying the condition that $B_i>B_j\Rightarrow(\rho_i,A_i,B_i)\succ(\rho_j,A_j,B_j)$. Furthermore, we choose a sequence $(t_i)$ of non-negative integers such that $B_i+t_i\ge0$ and $B_i+t_i>A_j+t_j$ holds for all $i\succ j$. Now, we define:
	$$\begin{aligned}
		\sigma_{\ge i}&=(\bigtimes_{j\ge i}S(\rho_j,A_j+t_j,B_j+t_j))\times(\bigtimes_{j<i}S(\rho_j,A_j,B_j)),\\
		\sigma_{> i}&=(\bigtimes_{j> i}S(\rho_j,A_j+t_j,B_j+t_j))\times(\bigtimes_{j\le i}S(\rho_j,A_j,B_j)),\\
	\end{aligned}$$
	and for $(\rho_i,A_i,B_i)\in\Jord(\psi_{\np})$, we define a functor $D_i$ by:
	$$\begin{aligned}
		D_{i,k}&=\begin{cases}
			D_{\rho,A_i+k}^{(2)}\circ D_{\rho,A_i-1+k}^{(2)}\circ\cdots\circ D_{\rho,B_i+k}^{(2)}&\text{if }\rho\text{ is self-dual}\\
			D_{\rho,A_i+k}\circ D_{\rho,A_i-1+k}\circ\dots\circ D_{\rho,B_i+k}&\\
			\circ D_{\rho^\vee,A_i+k}\circ D_{\rho^\vee,A_i-1+k}\circ\cdots D_{\rho^\vee,B_i+k}&\text{if }\rho\text{ is not self-dual},\\
		\end{cases}\\
		D_i&=D_{i,1}\circ D_{i,2}\circ\cdots\circ D_{i,t_m}.\\
	\end{aligned}$$
	In particular, if $B_i+k\le 0$, we replace the term $D_{\rho,1}^{(m)}\circ D_{\rho,0}^{(m)}$ from the definition of $D_{i,k}$ by $D_{[0,1]_\rho}^{(m)}$.
	
	\par By direct computation, we have $D_i(\sigma_{\ge i}\rtimes\pi)=\sigma_{>i}\rtimes\pi$. Note that $D_i(\sigma_{\ge i})$ is the highest derivative of $\sigma_{>i}$. We only need to prove that $(\bigtimes_i S(\rho_i,A_i+t_i,B_i+t_i))\rtimes\pi$ is irreducible. Thus, we may assume that $\Jord(\psi_{\np})=\left\{(\rho_1,A_1,B_1),(\rho_2,A_2,B_2),\cdots,(\rho_m,A_m,B_m)\right\}$ with $$0\le B_1\le A_1<B_2\le A_2<\cdots<B_m\le A_m.$$
	
	\begin{lemma}\label{lemma: reduction to good parity - lemma}
		In the above setting, for $(\rho,A,B)\in\Jord(\psi_{\np})$, $[A,-B]_\rho\rtimes\pi$ is irreducible.
	\end{lemma}
	\begin{proof}
		Write $\sigma_s=\soc([A,-B]_\rho\rtimes\pi)$ and $\sigma_q=\cos([A,-B]_\rho\rtimes\pi)=\soc([B,-A]_{\rho^\vee}\rtimes\pi)$. Consider $D=\Jac_{\rho^\vee,-A}\circ\Jac_{\rho^\vee,-A+1}\circ\cdots\circ\Jac_{\rho^\vee,B}$. Then $D([B,-A]_{\rho^\vee}\rtimes\pi)=\pi$ is irreducible, which implies that $[B,-A]_{\rho^\vee}\rtimes\pi$ is SI (because every irreducible subrepresentation $\tau$ of $[B,-A]_{\rho^\vee}\rtimes\pi$ must satisfy $D(\tau)=\pi$). Thus, we only need to prove that $\sigma_s=\sigma_q$. By Howe's duality (Theorem \ref{theorem: Howe's duality}), we only need to prove that $\theta_{-\alpha}(\sigma_s)=\theta_{-\alpha}(\sigma_q)$ for $\alpha\gg0$. Lemma \ref{lemma: compatibility of theta lifts and socle} implies that $\sigma_s\xhookrightarrow{}[A,-B]_\rho\rtimes\theta_{-\alpha}(\pi)$ and $\sigma_q\twoheadrightarrow{}[A,-B]_\rho\rtimes\theta_{-\alpha}(\pi)$. By Proposition \ref{prop: adams conjecture - discrete case} below, we know that $\theta_{-\alpha}(\pi)$ is elementary. Thus \cite[\S 6]{Mœglin2006_arthur_packets} implies that $[A,-B]_\rho\rtimes\theta_{-\alpha}(\pi)$ is irreducible. This completes the proof.
	\end{proof}
	
	\par Suppose that $A=B$ holds for all $(\rho,A,B)\in\Jord(\psi_\np)$. Then $S(\rho,A,B)=[A,-A]_\rho$. The proposition follows from Corollary \ref{coro: strengthening of proposition 4.5.3 in ArthurMp - tempered case} in this case. Thus, we may assume that $A>B$ for some $(\rho,A,B)\in\Jord(\psi_\np)$.
	
	\par Now, we fix a $(\rho,A,B)\in\Jord(\psi_\np)$ with $A>B$ and let $\sigma'=\bigtimes_{(\rho',A',B')\neq(\rho,A,B)}S(\rho,A,B)$. Suppose that $\tau$ is an irreducible subrepresentation of $\tau_{\psi_{\np}}\rtimes\pi$. Then $\tau$ can be embedded into $$\sigma'\times\soc([B,-A]_\rho\times S(\rho,A-1,B+1))\times [A,-B]_\rho\rtimes\pi.$$
	By Lemma \ref{lemma: reduction to good parity - lemma}, we have $[A,-B]_\rho\rtimes\pi=[B,-A]_{\rho^\vee}\rtimes\pi$. Note that $[B,-A]_{\rho^\vee}$ commutes with $\soc([B,-A]_\rho\times S(\rho,A-1,B+1))$. Thus $\tau$ can be embedded into
	$$[B,-A]_{\rho^\vee}\times[B,-A]_{\rho}\rtimes(\sigma'\times S(\rho,A-1,B+1)\rtimes\pi).$$
	
	\par By induction on $\sum_i(A_i-B_i)$, we may assume that $\pi'=\sigma'\times S(\rho,A-1,B+1)\rtimes\pi$ is irreducible. Consider 
	$$D=\begin{cases}
		D_{\rho,-A}^{(2)}\circ D_{\rho,-A+1}^{(2)}\circ\cdots\circ D_{\rho,B}^{(2)}&\text{if }\rho\text{ is self-dual}\\
		D_{\rho,-A}\circ D_{\rho,-A+1}\circ\cdots\circ D_{\rho,B}&\\
		\circ D_{\rho^\vee,-A}\circ D_{\rho^\vee,-A+1}\circ\cdots\circ D_{\rho^\vee,B}&\text{if }\rho\text{ is not self-dual}.\\
	\end{cases}$$ In particular, if $D_{\rho,-1}^{(m)}\circ D_{\rho,0}^{(m)}$ occurs in the above formula, we replace this term by $D_{[0,-1]_\rho}^{(m)}$ from the definition of $D$. We compute that $D([B,-A]_{\rho^\vee}\times[B,-A]_\rho\rtimes\pi')=\pi'$ is irreducible. Thus $[B,-A]_{\rho^\vee}\times[B,-A]_\rho\rtimes\pi'$ is SI. Therefore, $\tau_{\psi_{\np}}\rtimes\pi$ is also SI. Recall that $\tau_{\psi_{\np}}\rtimes\pi$ is unitary (see Theorem \ref{theorem: pi(psi,varepsilon) is a finite length representation}). Thus, we have completed the proof of Proposition \ref{prop: reduction to good parity}.
	
	\section{Adams conjecture}\label{section: Adams conjecture}
	
	\par In this section, we will prove the Adams conjecture for $(\widetilde{G}_{2n},H_{2m+1})$ with $\alpha\gg0$ (recall that $\alpha=2m-2n$ and $\theta_{-\alpha}=\theta_{V_{2m+1},W_{2n}}$).
	
	\subsection{Discrete case}
	
	\par In this subsection, we fix a $\phi\in\Phi_{\bdd,2}(\widetilde{G}_{2n})$ and a $\varepsilon\in\mS_\phi^\vee$. Let $s_0\in\mS_\phi$ be defined by $s_0(\rho,a)=-1$ for all $(\rho,a)\in\Jord(\phi)$.
	
	\begin{proposition}\label{prop: first occurrence of triv-cuspidal discrete series}
		Suppose that $\pi(\phi,\varepsilon)$ is $\triv$-cuspidal, i.e. $\Jac_{\triv,x}\pi(\psi,\varepsilon)=0$ for all $x\in\RR$. Let $\theta_{-\alpha}(\pi(\phi,\varepsilon))$ be the first occurrence of $\pi(\phi,\varepsilon)$ in the Witt tower $\left\{V_{2m+1}\right\}_{m\ge0}$. Then, the following holds:
		\begin{enumerate}
			\item[(1)] $\alpha=\begin{cases}-b_{\triv,\phi,\varepsilon}&\text{if }\varepsilon(s_0)=1\\
				b_{\triv,\phi,\varepsilon}+2&\text{if }\varepsilon(s_0)=-1.
			\end{cases}$ 
			\par See \S \ref{section: the key proposition} for the definition of $b_{\triv,\phi,\varepsilon}$.
			\item[(2)] Let $\phi_0=\begin{cases}
				\phi\ominus\triv\otimes r(b_{\triv,\phi,\varepsilon})&\text{if }\varepsilon(s_0)=1\\
				\phi\oplus\triv\otimes r(b_{\triv,\phi,\varepsilon}+2)&\text{if }\varepsilon(s_0)=-1
			\end{cases}$ 
			\par be a L-parameter for $\SO(V_{2m+1})$, let $\varepsilon_0$ be the unique element in $\mS_{\phi_0}^\vee$ such that $\varepsilon_0(\rho,a)=\varepsilon(\rho,a)$ for all $\rho\neq\triv$ and $\pi(\phi_0,\varepsilon_0)$ is $\triv$-cuspidal. Then we have $\theta_{-\alpha}(\pi(\phi,\varepsilon))|_{\SO(V_{2m+1})}=\pi(\phi_0,\varepsilon_0)$.
			\item[(3)] $\theta_{-\alpha}(\pi(\phi,\varepsilon))(-I_{2m+1})=\varepsilon(s_0)\epsilon(\phi)$ (the $\epsilon(\phi)$ here is the local root number of $\phi$, see \cite[\S 4.1]{Li2024_arthurpacketsmetaplecticgroups} for the definition of local root number). 
		\end{enumerate}
	\end{proposition}
	\begin{proof}
		This is a special case of \cite[Theorem 1.2, 1.3, 1.4]{Atobe2017_local}.
	\end{proof}
	
	\begin{proposition}\label{prop: adams conjecture - discrete case}
		When $\alpha\gg0$, the following holds:
		\begin{enumerate}
			\item[(1)] $\theta_{-\alpha}(\pi(\phi,\varepsilon))|_{\SO(V_{2m+1})}=\pi_M(\phi_\alpha,\varepsilon_\alpha^\M)$ with $\phi_\alpha=\phi\oplus\triv\otimes r(1)\otimes r(\alpha)$, $$\frac{\varepsilon_\alpha^\M(\rho,a)}{\varepsilon(\rho,a)}=\begin{cases}
				1&\text{if }\rho\neq\triv\\
				-1&\text{if }\rho=\triv
			\end{cases}$$ and $\varepsilon_\alpha^\M(\triv,1,\alpha)=-\varepsilon(s_0)\epsilon(\phi)$. The subscript ``$\M$" here means M\oe glin's normalization; see \cite[\S 6]{Xu2017_Moeglin} for more details.
			\item[(2)] $\theta_{-\alpha}(\pi(\phi,\varepsilon))(-I_{2m+1})=\varepsilon(s_0)\epsilon(\phi)$
		\end{enumerate}
	\end{proposition}\index{phi alpha@$\phi_\alpha$}\index{eps aloha@$\varepsilon_\alpha$}
	\begin{proof}
		The $(2)$ of this proposition is a part of \cite[Theorem 1.3, Theorem 1.4]{Atobe2017_local}. Therefore, we only need to prove the $(1)$ of the proposition. Write $b=b_{\triv,\phi,\varepsilon}$ and $a=a_{\triv,\phi,\varepsilon}$. We prove this proposition by induction on $a$. Suppose first that $a=\infty$, i.e., $\pi(\phi,\varepsilon)$ is $\triv$-cuspidal. Then, Proposition \ref{prop: first occurrence of triv-cuspidal discrete series} and Kudla's filtration imply that:
		\begin{enumerate}
			\item[(1)] If $\varepsilon(s_0)=1$, then $\theta_{-\alpha}(\pi(\phi,\varepsilon))\xhookrightarrow{}[-\frac{\alpha-1}{2},-\frac{1}{2}]_{\triv}\rtimes\theta_0(\pi(\phi,\varepsilon))$.
			\item[(2)] If $\varepsilon(s_0)=-1$, then $\theta_{-\alpha}(\pi(\phi,\varepsilon))\xhookrightarrow{}[-\frac{\alpha-1}{2},-\frac{b-1}{2}]_{\triv}\rtimes\pi(\phi_0,\varepsilon_0)$ (the pair $(\phi_0,\varepsilon_0)$ here is the pair in Proposition \ref{prop: first occurrence of triv-cuspidal discrete series}(2)).
		\end{enumerate}
		By M\oe glin's construction of elementary A-packet in \cite{Mœglin2006_Aubert}, we have:
		\begin{enumerate}
			\item[(1)] When $\varepsilon(s_0)=1$, $\pi_\M(\phi_\alpha,\varepsilon_\alpha^\M)=\soc([-\frac{\alpha-1}{2},-\frac{1}{2}]_\triv\rtimes\pi_\M(\phi,\varepsilon))$ (here we view $\phi$ as a L-parameter for $\SO(V_{2n+1})$).
			\item[(2)] When $\varepsilon(s_0)=-1$, $\pi_\M(\phi_\alpha,\varepsilon_\alpha^\M)=\soc([-\frac{\alpha-1}{2},-\frac{b+3}{2}]_{\triv}\rtimes\pi_\M(\phi_0,\varepsilon_0))$.
		\end{enumerate}
		 By \cite[Theorem 1.3]{Atobe2017_local}, we know that $\theta_0(\pi(\phi,\varepsilon))=\pi(\phi,\varepsilon)$ when $\varepsilon(s_0)=1$. Also, by \cite[Corollary 6.19]{Xu2017_Moeglin}, we have $\pi(\phi,\varepsilon)=\pi_\M(\phi,\varepsilon)$ and $\pi(\phi_0,\varepsilon_0)=\pi_\M(\phi_0,\varepsilon_0)$. This completes the proof for the case $a=\infty$.
		
		\par Now, we consider the case when $a<\infty$. If $a>b+2$, then Proposition \ref{prop: Moeglin's construction in discrete case} implies that $\pi(\phi,\varepsilon)=\soc(|\cdot|^{\frac{a-1}{2}}\rtimes\pi(\phi',\varepsilon'))$ with $a_{\triv,\phi',\varepsilon'}>a_{\triv,\phi,\varepsilon}$. By Lemma \ref{lemma: compatibility of theta lifts and socle}, we have $\theta_{-\alpha}(\pi(\phi,\varepsilon))\xhookrightarrow{}|\cdot|^{\frac{a-1}{2}}\rtimes\pi_\M(\phi'_\alpha,\varepsilon'^{\M}_{\alpha})$. Note that $\soc(|\cdot|^{\frac{a-1}{2}}\rtimes\pi_\M(\phi'_\alpha,\varepsilon'^\M_\alpha))=\pi_\M(\phi_\alpha,\varepsilon^\M_\alpha)$. Thus we have $\theta_{-\alpha}(\pi(\phi,\varepsilon))=\pi_\M(\phi_\alpha,\varepsilon^\M_\alpha)$.

		\par If $a=b+2$, then Proposition \ref{prop: Moeglin's construction in discrete case} implies that $\pi(\phi,\varepsilon)$ is a subrepresentation of $[\frac{a-1}{2},-\frac{b-1}{2}]_\triv\rtimes\pi(\phi',\varepsilon')$. Note that $\pi_\M(\phi_\alpha,\varepsilon^\M_\alpha)$ is the unique subrepresentation of $$[\frac{a-1}{2},-\frac{b-1}{2}]_\triv\rtimes\pi_\M(\phi'_\alpha,\varepsilon'^\M_\alpha)$$ satisfying $\Jac_{\triv,\frac{b-1}{2}}(\pi_\M(\phi_\alpha,\varepsilon^\M_\alpha))=0$. By Lemma \ref{lemma: compatibility of theta lifts and socle}, we only need to prove that $\Jac_{\triv,\frac{b-1}{2}}(\theta_{-\alpha}(\pi(\phi,\varepsilon))=0$. Let $\widetilde{\varepsilon}\in\mS_\phi^\vee$ be defined by $$\frac{\widetilde{\varepsilon}(\rho,x)}{\varepsilon(\rho,x)}=\begin{cases}
		    1&\text{if }(\rho,x)\neq(\triv,b),(\triv,a)\\
		    -1&\text{if }(\rho,x)=(\triv,b)\text{ or }(\triv,a).\\
		\end{cases}$$
		Then $\Jac_{\triv,\frac{b-1}{2}}(\pi(\phi,\widetilde{\varepsilon}))\neq0$ and $\pi(\phi,\widetilde{\varepsilon})$ is a subrepresentation of $[\frac{a-1}{2},-\frac{b-1}{2}]_{\triv}\rtimes\pi(\phi',\varepsilon')$. Now, Lemma \ref{lemma: compatibility of theta lifts and socle} implies that $\theta_{-\alpha}(\pi(\phi,\widetilde{\varepsilon}))$ is the only irreducible subrepresentation of $[\frac{a-1}{2},-\frac{b-1}{2}]_{\triv}\rtimes\pi_\M(\phi'_\alpha,\varepsilon'^\M_\alpha)$ such that $\Jac_{\triv,\frac{b-1}{2}}(\theta_{-\alpha}(\pi(\phi,\widetilde{\varepsilon})))\neq0$. By Howe's duality (Theorem \ref{theorem: Howe's duality}), we have $\theta_{-\alpha}(\pi(\phi,\varepsilon))\neq\theta_{-\alpha}(\pi(\phi,\widetilde{\varepsilon}))$, which completes the proof.
	\end{proof}
	
	\begin{corollary}\label{coro: adams conjecture - discrete case - Whittaker normalized}
		When $\alpha\gg0$, the following holds:
		\begin{enumerate}
			\item[(1)] $\theta_{-\alpha}(\pi(\phi,\varepsilon))|_{\SO(V_{2m+1})}=\pi(\phi_\alpha,\varepsilon_\alpha)$ with:
			\begin{itemize}
				\item[$\bullet$] $\phi_\alpha=\phi\oplus\triv\otimes r(1)\otimes r(\alpha)$.
				\item[$\bullet$] $\varepsilon_\alpha(\rho,a)=\varepsilon(\rho,a)$.
				\item[$\bullet$] $\varepsilon_\alpha(\triv,1,\alpha)=\varepsilon(s_0)$.
			\end{itemize}
			\item[(2)] $\theta_{-\alpha}(\pi(\phi,\varepsilon))(-I_{2m+1})=\varepsilon(s_0)\epsilon(\phi)$
		\end{enumerate}
	\end{corollary}
	\begin{proof}
		Follows from Proposition \ref{prop: adams conjecture - discrete case} and \cite[Corollary 6.19]{Xu2017_Moeglin}.
	\end{proof}
	
	\subsection{Good parity case}
	
	\par In this subsection, we fix an extended multi-segment $\mE=\cup_\rho\left\{([A_i,B_i]_\rho,l_i,\eta_i)\right\}_{i\in(I_\rho,\succ_\rho)}$ for $\widetilde{G}_{2n}$. Let $s_0\in\mS_{\psi_\mE}$ be defined as $s_0(\rho,a,b)=(-1)^{m(\rho,a,b)}$ for all $(\rho,a,b)\in\Jord(\psi_\mE)$, where $m(\rho,a,b)$ is the multiplicity of $(\rho,a,b)$ in $\Jord(\psi_\mE)$.
	
	\begin{lemma}\label{lemma: compatibility of theta lifts and highest derivative}
		Let $\pi\in\Pi_-(\widetilde{G}_{2n})$, $\rho\in\Pi_{\unit,\cusp}(\GL(d_\rho))$ and $x\in\RR$. For $\alpha\gg0$, the following is true:
		\begin{enumerate}
			\item[(1)] When $\rho|\cdot|^x$ is not self-dual, suppose that $D_{\rho,x}^{(k)}(\pi)$ is the highest derivatives of $\pi$, then $D_{\rho,x}^{(k)}(\theta_{-\alpha}(\pi))$ is the highest derivatives of $\theta_{-\alpha}(\pi)$, and we have $\theta_{-\alpha}(D_{\rho,x}^{(k)}(\pi))=D_{\rho,x}^{(k)}(\theta_{-\alpha}(\pi))$.
			\item[(2)] When $\rho$ is self-dual and $\Jac_{\rho,\zeta}(\pi)=0$, $\Jac_{\rho,\zeta}(\theta_{-\alpha}(\pi))=0$ hold for some $\zeta\in\left\{\pm1\right\}$. Then the statement in $(1)$ still holds if we replace $D_{\rho,x}$ by $D_{[0,\zeta]_\rho}$.
		\end{enumerate}
	\end{lemma}
	\begin{proof}
		By Lemma \ref{lemma: compatibility of theta lifts and socle} and the theory of derivatives (see \S \ref{section: derivative and socles}), it is sufficient to prove that, when $\rho|\cdot|^x$ is not self-dual, if $D_{\rho,x}(\pi)=0$, then $D_{\rho,x}(\theta_{-\alpha}(x))=0$. Suppose that $D_{\rho,x}^{(k)}(\theta_{-\alpha}(\pi))$ is the highest derivative of $\theta_{-\alpha}(\pi)$. Then Proposition \ref{prop: highest rho-derivative is irreducible in non-self-dual case} implies that we have $(\rho^\vee|\cdot|^{-x})^k\rtimes\pi'\twoheadrightarrow\theta_{-\alpha}(\pi)$ with $\pi'$ irreducible. Applying Lemma \ref{lemma: compatibility of theta lifts and socle} with $\Theta_{\alpha}=\Theta_{W_{2n},V_{2m+1}}$, we have $(\rho^\vee|\cdot|^{-x})^k\rtimes\Theta_{\alpha}(\pi')\twoheadrightarrow\Theta_{\alpha}(\theta_{-\alpha}(\pi))\twoheadrightarrow\pi$. Thus, there exist an irreducible subquotient $\pi''$ of $\Theta_{\alpha}(\pi')$, such that $(\rho^\vee|\cdot|^{-x})^k\rtimes\pi''\twoheadrightarrow\pi$. That is, we have $\pi\xhookrightarrow{}(\rho|\cdot|^x)^k\rtimes\pi''$, which implies that $k=0$ because $\Jac_{\rho,x}(\pi)=0$.
	\end{proof}
	
	\begin{lemma}\label{lemma: theta lifts of extended multi-segment - lemma}
		For $\phi\in\Phi_{\bdd,2}(\SO(V_{2m+1}))$ with $\phi=\bigoplus_\rho\bigoplus_{i\in I_\rho}\rho\otimes r(a_i)$, we set $\phi_\alpha=\phi\oplus\rho\otimes r(1)\otimes r(\alpha)$. Suppose that $\alpha\gg0$. For $\varepsilon\in\mS_{\phi_\alpha}^\vee$, we set 
		$$\mE_{\varepsilon}=\begin{cases}
			(\cup_\rho\left\{([\frac{a_i-1}{2},\frac{a_i-1}{2}]_\rho,0,\varepsilon(\rho,a_i))\right\}_{i\in I_\rho})\cup\left\{([\frac{\alpha-1}{2},-\frac{\alpha-1}{2}]_{\triv},\frac{\alpha}{2},1)\right\}&\text{if }\varepsilon(\rho,1,\alpha)=1\\
			(\cup_\rho\left\{([\frac{a_i-1}{2},\frac{a_i-1}{2}]_\rho,0,\varepsilon(\rho,a_i))\right\}_{i\in I_\rho})\cup\left\{([\frac{\alpha-1}{2},-\frac{\alpha-1}{2}]_{\triv},\frac{\alpha}{2}-1,1)\right\}&\text{if }\varepsilon(\rho,1,\alpha)=-1
		\end{cases}$$
		with $([\frac{\alpha-1}{2},-\frac{\alpha-1}{2}]_{\triv},l_\alpha,\eta_\alpha)\prec_{\triv}([\frac{a_i-1}{2},\frac{a_i-1}{2}]_{\triv},0,\varepsilon(\rho,a_i))$ for all $i\in I_{\triv}$. Then $\pi_\Ato(\phi_\alpha,\varepsilon)=\pi(\mE_{\varepsilon})$.
	\end{lemma}
	\begin{proof}
		This follows from \cite[Theorem 3.4, 3.6, 3.7]{Atobe2022_A-packet}.
	\end{proof}
	
	\begin{proposition}\label{prop: theta lifts of extended multi-segment}
		When $\alpha\gg0$, the following is true:
		\begin{enumerate}
			\item[(1)] Let $\mE_\alpha=\cup_{\rho}\left\{([A_i,B_i]_\rho,l_i,\eta_i)\right\}_{i\in(I_\rho,\succ_\rho)}\cup\left\{\left(\left[\frac{\alpha-1}{2},-\frac{\alpha-1}{2}\right]_{\triv},l_\alpha,\eta_\alpha\right)\right\}$ with 
			$$(l_\alpha,\eta_\alpha)=\begin{cases}
				(\frac{\alpha}{2},1)&\text{if }\varepsilon_\mE(s_0)=1\\
				(\frac{\alpha}{2}-1,1)&\text{if }\varepsilon_\mE(s_0)=-1
			\end{cases}$$ and $([\frac{\alpha-1}{2},-\frac{\alpha-1}{2}]_{\triv},l_\alpha,\eta_\alpha)\prec_{\triv}([A_i,B_i]_\triv,l_i,\eta_i)$ for all $i\in I_{\triv}$. Then we have $\theta_{-\alpha}(\pi(\mE))|_{\SO(V_{2m+1})}=\pi(\mE_{\alpha})$.
			\item[(2)] $\theta_{-\alpha}(\pi(\mE))(-I_{2m+1})=\varepsilon_\mE(s_0)\epsilon(\psi_\mE)$.
		\end{enumerate}
	\end{proposition}\index{E alpha@$\mE_\alpha$}
	\begin{proof}
		\par We first consider the case when $\psi_\mE$ is a non-negative DDR, and we make induction on $\sum_\rho\sum_{i\in I_\rho} l_i$. Note that $\prod_{0\le k\le b_i-1}\epsilon(\rho\otimes r(2(B_i+k)+1))=\epsilon(\rho\otimes r(a_i)\otimes r(b_i))$. By Corollary \ref{coro: adams conjecture - discrete case - Whittaker normalized} and Lemma \ref{lemma: theta lifts of extended multi-segment - lemma}, the proposition is true when $\sum_\rho\sum_{i\in I_\rho} l_i=0$.
		
		\par Suppose $\sum_\rho\sum_{i\in I_\rho} l_i>0$. By Lemma \ref{lemma: reduction step in non-negative DDR case}, there exists an embedding $\pi(\mE)\xhookrightarrow{}[B_i,-A_i]_\rho\rtimes\pi(\mE^-)$. By Lemma \ref{lemma: compatibility of theta lifts and socle}, we have $\theta_{-\alpha}(\pi(\mE))\xhookrightarrow{}[B_i,-A_i]_\rho\rtimes\theta_{-\alpha}(\pi(\mE^-))$. By the induction hypothesis, we have $\theta_{-\alpha}(\pi(\mE^-))=\pi(\mE_{\alpha}^-)$. Note that $\psi_{\mE_{\alpha}^-}$ is a DDR. By M\oe glin's construction for representations in DDR packets, we have $\pi(\mE_\alpha)=\soc([B_i,-A_i]_\rho\rtimes\pi(\mE_\alpha^-))$ (see, e.g., \cite[Proposition 3.3(i)]{Moeglin2011_adams_conjecture}), which implies that $\theta_{-\alpha}(\pi(\mE))=\pi(\mE_\alpha)$.
		
		\par For general $\mE$, choose $\textbf{t}$ as in \S \ref{section: good parity case} so that $\psi_{\mE_{\textbf{t}}}$ is a non-negative DDR. Then $\pi(\mE)$ is defined to be $D_{\textbf{t}}(\pi(\mE_{\textbf{t}}))$. By Lemma \ref{lemma: compatibility of theta lifts and highest derivative}, we have $\theta_{-\alpha}(D_{\textbf{t}}(\pi(\mE_{\textbf{t}})))=D_{\textbf{t}}(\theta_{-\alpha}(\pi(\mE_{\textbf{t}})))=D_{\textbf{t}}(\pi((\mE_{\textbf{t}})_\alpha))=\pi(\mE_\alpha)$. This completes the proof.
	\end{proof}
	
	\subsection{General case}
	
	\par Now, we can prove the Adams conjecture for $\alpha\gg0$.
	
	\begin{theorem}\label{theorem: adams conjecture}
		For $\psi\in\Psi(\widetilde{G}_{2n})$, write $\psi_\alpha=\psi\oplus\triv\otimes r(1)\otimes r(\alpha)$. Suppose that $\alpha\gg0$. Then, for $\pi\in\Pi_\psi$, we have $\theta_{-\alpha}(\pi)\in\Pi_{\psi_\alpha}$.
	\end{theorem}\index{psi alpha@$\psi_\alpha$}
	\begin{proof}
		By Proposition \ref{prop: theta lifts of extended multi-segment}, the theorem is true when $\psi$ is of good parity. The general case follows from Lemma \ref{lemma: compatibility of theta lifts and socle} and Proposition \ref{prop: reduction to good parity}.
	\end{proof}
	
	\section{Consequences of the Adams conjecture}\label{section: Consequences of the Adams conjecture}
	
	\par In this section, we will derive some consequences of Proposition \ref{prop: theta lifts of extended multi-segment}.
	
	\subsection{Theorem of multiplicity one}
	
	\par In this subsection, we prove that $\Pi_\psi$ is multiplicity-free for $\psi\in\Psi(\widetilde{G}_{2n})$.
	
	\begin{proposition}\label{prop: well-definedness of associated representation of extended multi-segment}
		For an extended multi-segment $\mE$, the associated representation $\pi(\mE)$ is independent of the choice of $\textbf{t}$.
	\end{proposition}
	\begin{proof}
		For another choice $\textbf{t}'$, we have $\pi(\mE_\alpha)=D_{\textbf{t}}(\pi((\mE_{\textbf{t}})_\alpha))=D_{\textbf{t'}}(\pi((\mE_{\textbf{t'}})_\alpha))$. Thus we have $\theta_{-\alpha}(D_{\textbf{t}}(\pi(\mE_{\textbf{t}})))=\theta_{-\alpha}(D_{\textbf{t'}}(\pi(\mE_{\textbf{t'}})))$. Now, the proposition follows from Howe's duality (Theorem \ref{theorem: Howe's duality}).
	\end{proof}
	
	\begin{proposition}\label{prop: associated representation of extended multi-segment is isomorphic iff equivalent}
		Let $\mE$ and $\mE'$ be extended multi-segments. Suppose that $\psi_{\mE}=\psi_{\mE'}$ and $\pi(\mE),\pi(\mE')\neq0$. Then $\pi(\mE)\cong\pi(\mE')$ if and only if $\mE$ is equivalent to $\mE'$.
	\end{proposition}
	\begin{proof}
		Suppose that $\pi(\mE)\cong\pi(\mE')$. Then we have $\pi(\mE_\alpha)\cong\pi(\mE_\alpha')$ for $\alpha\gg0$. Thus $\mE_\alpha$ and $\mE_\alpha'$ are equivalent, which implies that $\mE$ and $\mE'$ are equivalent. The other direction is trivial.
	\end{proof}
	
	\begin{corollary}\label{coro: A-packet is multiplicity free}
		For $\psi\in\Psi(\widetilde{G}_{2n})$, $\Pi_\psi$ is multiplicity free.
	\end{corollary}
	\begin{proof}
		When $\psi$ is of good parity, the corollary follows from Theorem \ref{theorem: A-packets can be constructed via extended multi-segments} and Proposition \ref{prop: associated representation of extended multi-segment is isomorphic iff equivalent}. The general case can be deduced from the good parity case and Proposition \ref{prop: reduction to good parity}.
	\end{proof}
	
	\subsection{A non-vanishing criterion}
	
	\par In \cite{Xu2021_combinatorial}, Xu gives an algorithm to determine when M\oe glin's  constructions are non-zero. Later, Atobe applies Xu's algorithm to $\pi(\mE)$ and gives a non-vanishing criterion for $\pi(\mE)$ (see \cite[\S 4]{Atobe2022_A-packet}). We will generalize Atobe's criterion (\cite[Theorem 4.4]{Atobe2022_A-packet}) to the metaplectic case.
	
	\par We fix an extended multi-segment $\mE=\cup_\rho\left\{([A_i,B_i]_\rho,l_i,\eta_i)\right\}_{i\in (I_\rho,\succ_\rho)}$ for $\widetilde{G}_{2n}$. Take a non-negative integer $t$ such that $t+B_i\ge 0$ for all $\rho$ and $i\in I_\rho$. Define $\mE_t:=\cup_\rho\left\{([A_i+t,B_i+t]_\rho,l_i,\eta_i)\right\}_{i\in (I_\rho,\succ_\rho)}$. Then, we have:
	
	\begin{lemma}\label{lemma: pi(Et alpha) neq 0 if and only if pi((E alpha)t) neq 0}
		Let $t'$ be a non-negative integer such that $t'>\alpha$. Then, the representation $\pi((\mE_t)_\alpha)$ is non-zero if and only if $\pi((\mE_\alpha)_{t+t'})$ is non-zero (recall from Proposition \ref{prop: theta lifts of extended multi-segment} that $(\mE_t)_\alpha$ and $(\mE_\alpha)_{t+t'}$ are extended multi-segments for $\SO(V_{2m+1})$).
	\end{lemma}
	\begin{proof}
		By \cite[Theorem 3.7]{Atobe2022_A-packet}, we know that $\pi((\mE_t)_\alpha)\neq 0\Leftrightarrow\pi(((\mE_t)_\alpha)_{t'})\neq 0$. By \cite[Theorem 4.4]{Atobe2022_A-packet} we have $\pi(((\mE_t)_\alpha)_{t'})\neq 0\Leftrightarrow\pi((\mE_\alpha)_{t+t'})\neq 0$. This completes the proof.
	\end{proof}
	
	\par Now, we can generalize \cite[Theorem 3.7]{Atobe2022_A-packet} to the metaplectic case:
	
	\begin{proposition}\label{prop: reduction to the non-negative case}
		The representation $\pi(\mE)$ is non-zero if and only if $\pi(\mE_t)\neq 0$ and the following condition holds for all $\rho$ and $i\in I_\rho$:
		\begin{equation}\tag{$\star$}
			B_i+l_i\ge\begin{cases}
				0&\text{if }B_i\in\ZZ\\
				\frac{1}{2}&\text{if }B_i\notin\ZZ\text{ and }\eta_i\neq(-1)^{\alpha_i}\\
				-\frac{1}{2}&\text{if }B_i\notin\ZZ\text{ and }\eta_i=(-1)^{\alpha_i},\\
			\end{cases}
		\end{equation}
		where we set:
		$$\alpha_i=\begin{cases}
			\sum_{j\in I_\rho,j\prec_\rho i}(A_j+B_j+1)&\text{if }\rho\neq\triv\\
			1+\sum_{j\in I_\rho,j\prec_\rho i}(A_j+B_j+1).&\text{if }\rho=\triv\\
		\end{cases}$$
	\end{proposition}
	\begin{proof}
		By \cite[Theorem 3.7]{Atobe2022_A-packet}, we know that $\pi(\mE)$ is non-zero if and only if $\pi((\mE_\alpha)_{t+t'})\neq 0$ and the condition $(\star)$ in \cite[Theorem 3.7]{Atobe2022_A-packet} holds for $\mE_\alpha$. By Lemma \ref{lemma: pi(Et alpha) neq 0 if and only if pi((E alpha)t) neq 0}, we know that $\pi((\mE_\alpha)_{t+t'})\neq 0$ if and only if $\pi(\mE_t)\neq 0$. By Proposition \ref{prop: theta lifts of extended multi-segment}, it is straightforward to verify that the condition $(\star)$ above for $\mE$ is equivalent to the condition $(\star)$ for $\mE_\alpha$ in \cite[Theorem 3.7]{Atobe2022_A-packet}. Thus, the proposition follows from \cite[Theorem 3.7]{Atobe2022_A-packet}.
	\end{proof}
	
	\par By Proposition \ref{prop: reduction to the non-negative case}, in the rest of this subsection, we assume that $\mE$ is non-negative (i.e., $B_i\ge 0$ for all $\rho$ and $i\in I_\rho$).
	
	\begin{proposition}\label{prop: necessary conditions for pi(E) neq 0}
		Suppose that $\pi(\mE)\neq 0$. Let $k\succ_\rho k-1$ be two adjacent elements in $I_\rho$. Then, we have:
		\begin{itemize}
			\item[(1)] If $A_k\ge A_{k-1}$ and $B_k\ge B_{k-1}$, then
			\begin{itemize}
				\item[$\bullet$] $\eta_k=(-1)^{A_{k-1}-B_{k-1}}\eta_{k-1}\Longrightarrow A_k-l_k\ge A_{k-1}-l_{k-1}$ and $B_k+l_k\ge B_{k-1}+l_{k-1}$.
				\item[$\bullet$] $\eta_k\neq (-1)^{A_{k-1}-B_{k-1}}\eta_{k-1}\Longrightarrow B_k+l_k>A_{k-1}-l_{k-1}$.
			\end{itemize}
			In particular, if $[A_k,B_k]_\rho=[A_{k-1},B_{k-1}]_\rho$, then $\eta_k=(-1)^{A_{k-1}-B_{k-1}}\eta_{k-1}$ and $l_k=l_{k-1}$.
			\item[(2)] If $[A_{k-1},B_{k-1}]_\rho\subset[A_k,B_k]_\rho$, then
			\begin{itemize}
				\item[$\bullet$] $\eta_k=(-1)^{A_{k-1}-B_{k-1}}\eta_{k-1}\Longrightarrow0\le l_k-l_{k-1}\le b_k-b_{k-1}$.
				\item[$\bullet$] $\eta_k\neq(-1)^{A_{k-1}-B_{k-1}}\eta_{k-1}\Longrightarrow l_k+l_{k-1}\ge b_{k-1}$.
			\end{itemize}
			\item[(3)] If $[A_{k-1},B_{k-1}]_\rho\supset[A_k,B_k]_\rho$, then
			\begin{itemize}
				\item[$\bullet$] $\eta_k=(-1)^{A_{k-1}-B_{k-1}}\eta_{k-1}\Longrightarrow0\le l_{k-1}-l_k\le b_{k-1}-b_k$.
				\item[$\bullet$] $\eta_k\neq(-1)^{A_{k-1}-B_{k-1}}\eta_{k-1}\Longrightarrow l_k+l_{k-1}\ge b_k$.
			\end{itemize}
		\end{itemize}
	\end{proposition}
	\begin{proof}
		This follows directly from Proposition \ref{prop: theta lifts of extended multi-segment} and \cite[Proposition 4.1]{Atobe2022_A-packet}.
	\end{proof}
	
	\par To state the non-vanishing criterion, we first introduce the row exchange operator $R_k$ on an extended multi-segment:
	
	\begin{definition}\label{def: row exchange operators}
		Let $\mE=\cup_\rho\left\{([A_i,B_i]_\rho,l_i,\eta_i)\right\}_{i\in (I_\rho,\succ_\rho)}$ be a non-negative extended multi-segment for $\widetilde{G}_{2n}$ (i.e., $B_i\ge 0$ for all $\rho$ and $i\in I_\rho$). Let $k\succ_\rho k-1$ be two adjacent elements in $I_\rho$. Let $\succ_\rho'$ be the order on $I_\rho$ defined by $k-1\succ_\rho' k$ and $i\succ_\rho j\Longleftrightarrow i\succ_\rho' j$ for $(i,j)\neq (k,k-1)$. If $\succ_\rho'$ is also an admissible order on $I_\rho$ (recall from Definition \ref{def: extended multi-segment} that admissible means $A_i>A_j,B_i>B_j\Rightarrow i\succ'j$; thus, this condition is equivalent to $[A_{k-1},B_{k-1}]_\rho\subset[A_k,B_k]_\rho$ or $[A_{k-1},B_{k-1}]_\rho\supset[A_k,B_k]_\rho$), we define 
		$$R_{k-1}(\mE):=(\cup_{\rho'\neq \rho}\left\{([A_i,B_i]_{\rho'},l_i,\eta_i)\right\}_{i\in (I_{\rho'},\succ_{\rho'})})\cup\left\{([A_i,B_i]_\rho,l_i',\eta_i')\right\}_{(i\in I_\rho,\succ_\rho')}$$
		by:
		\begin{itemize}
			\item[(a)] When $[A_{k-1},B_{k-1}]_\rho\subset[A_k,B_k]_\rho$, we set $l_i',\eta_i'$ as follows:
			\begin{itemize}
				\item[$\bullet$] $(l_i',\eta_i')=(l_i,\eta_i)$ for $i\notin\left\{k-1,k\right\}$.
				\item[$\bullet$] $(l_{k-1}',\eta_{k-1}')=(l_{k-1},(-1)^{A_k-B_k}\eta_{k-1})$.
				\item[$\bullet$] $l_k'=l_k+(-1)^{A_{k-1}-B_{k-1}}\eta_{k-1}\eta_k(b_{k-1}-2l_{k-1})$ in $(\ZZ/b_k\ZZ)/\left\{\pm1\right\}$. Here, we identify $\left\{l\in\ZZ:0\le l\le\frac{b}{2}\right\}$ with $(\ZZ/b\ZZ)/\left\{\pm1\right\}$.
				\item[$\bullet$] $\eta_k'=\begin{cases}
					(-1)^{A_{k-1}-B_{k-1}}\eta_k&\text{if }\eta_k=(-1)^{A_{k-1}-B_{k-1}}\eta_{k-1}\text{, }b_k-2l_k<2(b_{k-1}-2l_{k-1})\\
					(-1)^{A_{k-1}-B_{k-1}-1}\eta_k&\text{otherwise}.
				\end{cases}$
			\end{itemize}
			\item[(b)] When $[A_{k-1},B_{k-1}]_\rho\supset[A_k,B_k]_\rho$, we set $l_i',\eta_i'$ as follows:
			\begin{itemize}
			\item[$\bullet$] $(l_i',\eta_i')=(l_i,\eta_i)$ for $i\notin\left\{k-1,k\right\}$.
			\item[$\bullet$] $(l_k',\eta_k')=(l_k,(-1)^{A_{k-1}-B_{k-1}}\eta_k)$.
			\item[$\bullet$] $l_{k-1}'=l_{k-1}+(-1)^{A_{k-1}-B_{k-1}}\eta_{k-1}\eta_k(b_k-2l_k)$ in $(\ZZ/b_{k-1}\ZZ)/\left\{\pm1\right\}$.
			\item[$\bullet$] $\eta_{k-1}'=\begin{cases}
				(-1)^{A_k-B_k}\eta_{k-1}&\text{if }\eta_k=(-1)^{A_{k-1}-B_{k-1}}\eta_{k-1}\\
				&\text{ and }b_{k-1}-2l_{k-1}<2(b_k-2l_k)\\
				(-1)^{A_k-B_k-1}\eta_{k-1}&\text{otherwise}.
			\end{cases}$
			\end{itemize}
		\end{itemize}
		If $\succ_\rho'$ is not admissible, we define $R_{k-1}(\mE):=\mE$.
	\end{definition}\index{R row exchange operator@$R_k(\mE)$}
	
	\begin{theorem}\label{theorem: row exchange does not change pi(E)}
		Let $\mE=\cup_\rho\left\{([A_i,B_i]_\rho,l_i,\eta_i)\right\}_{i\in (I_\rho,\succ_\rho)}$ be a non-negative extended multi-segment for $\widetilde{G}_{2n}$. Then $\pi(\mE)\cong\pi(R_k(\mE))$ for all $\rho$ and $k\in I_\rho$.
	\end{theorem}
	\begin{proof}
		This follows from Proposition \ref{prop: theta lifts of extended multi-segment} and \cite[Theorem 4.3]{Atobe2022_A-packet}.
	\end{proof}
	
	\par For $\rho\in\Pi_{\unit,\cusp}(\GL(d_\rho))$, we denote by $I_\rho^{2,\mathrm{adj}}$ the set of triples $(i,j,\succ_\rho')$, where $\succ_\rho'$ is an admissible order on $I_\rho$, and $i\succ_\rho'j$ are adjacent elements in $I_\rho$ with respect to $\succ_\rho'$. For $(i,j,\succ_\rho')\in I_\rho^{2,\mathrm{adj}}$, we can apply a sequence of operators $R_k$ with $k\in I_\rho$ to change the admissible order from $\succ_\rho$ to $\succ_\rho'$. We denote by $\mE_{\succ_\rho'}$ the resulting extended multi-segment after applying these operators $R_k$ to $\mE$. Now, we can state and prove the non-vanishing criterion.
	\index{I rho 2 adj@$I_\rho^{2,\mathrm{adj}}$}
	\index{E another admissible order@$\mE_{\succ_\rho'}$}
	
	\begin{theorem}\label{theorem: non-vanishing criterion}
		The representation $\pi(\mE)$ is nonzero if and only if for every $(i,j,\succ_\rho')\in I_\rho^{2,\mathrm{adj}}$, the three necessary conditions in Proposition \ref{prop: necessary conditions for pi(E) neq 0} are satisfied for $\mE_{\succ_\rho'}$ with respect to $i\succ_\rho'j$.
	\end{theorem}
	\begin{proof}
		Note that for $\alpha\gg0$, we have that $[\frac{\alpha-1}{2},-\frac{\alpha-1}{2}]_{\triv}\supset[A_i,B_i]_{\triv}$, $0\le l_\alpha-l_i\le\alpha-b_i$, and $l_i+l_\alpha\ge b_i$ for all $i\in I_{\triv}$. Thus, the theorem follows from Proposition \ref{prop: theta lifts of extended multi-segment} and \cite[Theorem 4.4]{Atobe2022_A-packet}.
	\end{proof}
	
	\section{Proof of Theorem \ref{theorem: construction of A-packet non-negative DDR case}}\label{section: proof of the theorem}
	
	\par This section is devoted to proving Theorem \ref{theorem: construction of A-packet non-negative DDR case}. The content of this section is parallel to \cite[\S 5]{Mœeglin2009_paquets}. Thus, readers familiar with \cite{Mœeglin2009_paquets} may skim this section.
	
	\par In this section, we fix $\psi\in\Psi_\gp(\widetilde{G}_{2n})$ to be a non-negative DDR, and fix an $\varepsilon\in\mS_\psi^\vee$. We fix a $(\rho,A_i,B_i)\in\Jord(\psi)$ with $A_i>B_i$ and set $\eta_0=\varepsilon(\rho,A_i,B_i)$. Let $\psi'$ be obtained from $\psi$ by removing $(\rho,A_i,B_i)$ and let $\varepsilon'$ be the restriction of $\varepsilon$ on $\Jord(\psi')$.
	
	\par We prove the theorem by induction on $\sum_{(\rho,A_i,B_i)\in\Jord(\psi)}(A_i-B_i)$. That is, we assume that the theorem is true for all parameters $\psi''$ with $\sum_{(\rho,A_i'',B_i'')\in\Jord(\psi'')}(A_i''-B_i'')<\sum_{(\rho,A_i,B_i)\in\Jord(\psi)}(A_i-B_i)$. In particular, Theorem \ref{theorem: A-packets can be constructed via extended multi-segments - NDDR case} is true for all such parameters.
	
	\subsection{Preparations}
	
	\par In this subsection, we develop some technical lemmas that will be used in the proof of Theorem \ref{theorem: construction of A-packet non-negative DDR case}.
	
	\begin{lemma}\label{lemma: highest derivative of representations in a A packet}
		Let $\psi\in\Psi_{\gp}(\widetilde{G}_{2n})$, $\pi\in\Pi_\psi$, $\rho\in\Pi_{\unit,\cusp}(\GL(d_\rho))$ and $x\in\RR$. Suppose that $D_{\rho,x}^{(k)}(\pi)$ is the highest derivative. Then $$k\le |\left\{(\rho,A,B)\in\Jord(\psi):B=x\right\}|.$$
	\end{lemma}
	\begin{proof}
		We prove that $D_{\rho,x}^{(k)}(T_{\psi,s})=0$ whenever $k>|\left\{(\rho,A,B)\in\Jord(\psi):B=x\right\}|$. Let $(\textbf{G}^!,\psi^!)\leftrightarrow(\psi,s)$ be as in (\ref{equation: basic bijection}). By Lemma \ref{lemma: commutation of spectral transfer and partial Jacquet module}, we have $$D_{\rho,x}^{(k)}\mT_{\textbf{G}^!,\widetilde{G}_{2n}}^\vee=\sum_{k_1+k_{-1}=k}\mT_{\textbf{G}_{(k_1d_\rho,k_{-1}d_\rho)}^!,\widetilde{G}_{2n-2kd_\rho}}\circ D_{\rho,x}^{(k_1),1}\circ D_{\rho,x}^{(k_{-1}),-1}.$$ Here $D_{\rho,x}^{(k_1),1}$ and $D_{\rho,x}^{(k_{-1}),-1}$ mean taking the $\rho$-derivative in the first and second $\SO$ factor respectively. By \cite[Proposition 8.3]{Xu2017_Moeglin}, if $D_{\rho,x}^{(k)}(T_{\psi,s})\neq0$, we have $$k_{\zeta}\le|\left\{(\rho,A,B)\in\Jord(\psi):B=x\text{ and }s(\rho,A,B)=\zeta\right\}|$$ for $\zeta\in\left\{\pm1\right\}$. Thus $$k=k_1+k_{-1}\le|\left\{(\rho,A,B)\in\Jord(\psi):B=x\right\}|.$$ This completes the proof.
	\end{proof}
	
	\begin{lemma}\label{lemma: Proposition 4.4 of Moeglin}
		Let $\varepsilon\in\mS_\psi^\vee$. Suppose that $B_i>0$, $B_i>A_{i-1}+1$, and Theorem \ref{theorem: A-packets can be constructed via extended multi-segments - NDDR case} is true for $\psi$. Then we have 
		$$\pi(\psi,\varepsilon)=\soc([B_i,A_i]_\rho\rtimes\pi(\psi',\varepsilon',(\rho,A_i-1,B_i-1;\eta_0))).$$
		In particular, by the induction hypothesis, for $B+1<C\le A$, we have
		$$\begin{aligned}
			\Jac_{\rho,C}\circ\Jac_{\rho,C+1}\circ\cdots&\circ\Jac_{\rho,B+2}(\pi(\psi',\varepsilon',(\rho,A,B+2;\eta_0)))\\
			&=\soc([C+1,A]_\rho\rtimes\pi(\psi',\varepsilon',(\rho,A-1,B+1;\eta_0))).
		\end{aligned}$$
	\end{lemma}
	\begin{proof}
		It follows from Lemma \ref{lemma: another reduction step in non-negative DDR case} that $\pi(\psi,\varepsilon)=\soc([B_i,A_i]_\rho\rtimes\pi(\psi',\varepsilon',(\rho,A_i-1,B_i-1;\eta_0)))$. Thus, we have $\pi(\psi',\varepsilon',(\rho,A,B+2;\eta_0))=\soc([B+2,A]_\rho\rtimes\pi(\psi',\varepsilon',(\rho,A-1,B+2;\eta_0)))$. By taking partial Jacquet module on both sides of the equation, we have $$\begin{aligned}
			\Jac_{\rho,C}\circ\Jac_{\rho,C+1}\circ\cdots\circ&\Jac_{\rho,B+2}(\pi(\psi',\varepsilon',(\rho,A,B+2;\eta_0)))\\
			&=\soc([C+1,A]_\rho\rtimes\pi(\psi',\varepsilon',(\rho,A-1,B+1;\eta_0))).
		\end{aligned}$$
	\end{proof}
	
	\begin{lemma}\label{lemma: lemma on the main term}
		Let $\varepsilon\in\mS_\psi^\vee$. Then we have $$\Jac_{\rho,-A}\circ\Jac_{\rho,-A+1}\circ\cdots\circ\Jac_{\rho,B}\pi(\psi,\varepsilon)=\pi(\psi',\varepsilon',(\rho,A-1,B+1;\eta_0)).$$
		In particular, by the theory of derivatives, for an irreducible constituent $\pi$ of $\pi(\psi,\varepsilon)$, if $\Jac_{\rho,-A}\circ\Jac_{\rho,-A+1}\circ\cdots\circ\Jac_{\rho,B}\pi\neq0$, then $\pi\xhookrightarrow{}\soc([B_i,-A_i]\rtimes\pi(\psi',\varepsilon',(\rho,A-1,B+1;\eta_0)))$.
	\end{lemma}
	\begin{proof}
		By Lemma \ref{lemma: Proposition 4.4 of Moeglin} and the induction hypothesis, we have
		$$\begin{aligned}
			\Jac_{\rho,-A}\circ\Jac_{\rho,-A+1}\circ\cdots\circ&\Jac_{\rho,B}\pi(\psi,\varepsilon)\\
			&=\sum_{C\in(B,A]}(-1)^{A-C}\Jac_{\rho,-A}\circ\Jac_{\rho,-A+1}\circ\cdots\circ\Jac_{\rho,-(C+1)}Y_C,
		\end{aligned}$$ where $Y_C=\soc([C+1,A]_\rho\rtimes\pi(\psi',\varepsilon',(\rho,A-1,B+1;\eta_0)))$. It is clear that $\Jac_{\rho,-A}\circ\Jac_{\rho,-A+1}\circ\cdots\circ\Jac_{\rho,-(C+1)}Y_C=0$ unless $A=C$, in which case we have $\Jac_{\rho,-A}\circ\Jac_{\rho,-A+1}\circ\cdots\circ\Jac_{\rho,-(C+1)}Y_C=\pi(\psi',\varepsilon',(\rho,A-1,B+1;\eta_0))$. This completes the proof.
	\end{proof}
	\begin{remark}
		It can be deduced directly from Theorem \ref{theorem: Mp version of Theorem 7.5 in Xu_moeglin} and the above proof that $$\soc([B,-A]_\rho\rtimes\pi(\psi',\varepsilon',(\rho,A-1,B+1;\eta_0)))\le\pi(\psi,\varepsilon).$$ Thus, in order to prove Theorem \ref{theorem: construction of A-packet non-negative DDR case}, it is sufficient to prove the following things:
		\begin{enumerate}
			\item[(1)] For an irreducible constituent $\pi$ of $\pi(\psi,\varepsilon)$, if $\Jac_{\rho,-A}\circ\Jac_{\rho,-A+1}\circ\cdots\circ\Jac_{\rho,B}\pi=0$, then $\pi$ is isomorphic to $\pi(\psi',\varepsilon',\cup_{B\le C\le A}(\rho,C,C;(-1)^{C-B}\eta))$ for some suitable $\eta$.
			\item[(2)] For suitable $\eta$, the representation $\pi(\psi',\varepsilon',\cup_{B\le C\le A}(\rho,C,C;(-1)^{C-B}\eta))$ occurs with multiplicity one in $\pi(\psi,\varepsilon)$.
		\end{enumerate}
	\end{remark}
	
	\par The next lemmas provide some general facts about parabolic induction.
	
	\begin{lemma}\label{lemma: inversion of scole}
		Let $[x,y]_\rho$, $[x',y']_\rho$ be segments with $\rho\in\Pi_{\unit,\cusp}(\GL(d_\rho))$, and let $\pi\in\Pi_-(\widetilde{G}_{2n})$. Suppose the following conditions hold:
		\begin{enumerate}
			\item[(1)] $-y\notin[x,y]$, $-y'\notin[x',y']$.
			\item[(2)] $[x,y]\subset[x',y']$ and $x',y'\notin[x,y]$.
			\item[(3)] For every $z\in[x,y]$ and $z'\in[z',y']$, we have $\Jac_{\rho,y}\circ\Jac_{\rho,y+\zeta}\circ\dots\circ\Jac_{\rho,z}(\pi)=0$ and $\Jac_{\rho,y'}\circ\Jac_{\rho,y'-\zeta'}\circ\dots\circ\Jac_{\rho,z'}(\pi)=0$, where $\zeta=\Sgn(y-z)$ and $\zeta'=\Sgn(y'-z')$.
		\end{enumerate}
		Then we have $$\soc([x,y]_\rho\rtimes\soc([x',y']_\rho\rtimes\pi))=\soc([x',y']_\rho\rtimes\soc([x,y]_\rho\rtimes\pi)).$$
	\end{lemma}
	\begin{proof}
		Note that $[x,y]_\rho$ commutes with $[x',y']_\rho$. It is sufficient to prove that $[x,y]_\rho\times[x',y']_\rho\rtimes\pi$ is SI. By the conditions of the lemma, using the same argument as in the proof of Lemma \ref{lemma: computation of partial Jacquet module of [0,zeta]_rho rtimes sigma}, we have $\Jac_{[x',y']_\rho}([x',y']_\rho\times[x,y]_\rho\rtimes\pi)$ and $\Jac_{[x,y]_\rho}([x,y]_\rho\rtimes\pi)=\pi$. This completes the proof.
	\end{proof}
	
	\begin{lemma}\label{lemma: subquotient of parabolic induction - lemma}
		Let $[x,y]_\rho$, $[x',y']_\rho$ be segments with $\rho\in\Pi_{\unit,\cusp}(\GL(d_\rho))$, and let $\pi'\in\Pi_-(\widetilde{G}_{2n})$.  Suppose that the following conditions hold:
		\begin{enumerate}
			\item[(1)] $\pi$ is an irreducible subrepresentation of $[x',y']_\rho\rtimes\pi'$.
			\item[(2)] The parabolic inductions $[x,y]_\rho\times[x',y']_\rho$ and $[x,y]_\rho\times[x',y']_\rho^\vee$ are irreducible.
			\item[(3)] For all $z'\in[x',y']$, $\Jac_{\rho,y'}\circ\Jac_{\rho,y'-\zeta'}\circ\cdots\circ\Jac_{\rho,z'}(\pi')=0$, where $\zeta'=\Sgn(y'-z')$.
			\item[(4)] $\left\{x,-y\right\}\cap[x',y']=\emptyset$ or $\left\{y',-y'\right\}\cap[x,y]=\emptyset$.
			\item[(5)] $-y'\notin[x',y']$.
		\end{enumerate}
		Then $\Jac_{\rho,y'}\circ\Jac_{\rho,y'-\zeta'}\circ\cdots\circ\Jac_{\rho,z'}([x,y]_\rho\rtimes\pi')=0$ for all $z'\in[x',y']$ and $\Jac_{\rho,y'}\circ\Jac_{\rho,y'-\zeta'}\circ\cdots\circ\Jac_{\rho,x'}([x,y]_\rho\times[x',y']_\rho\rtimes\pi')=[x,y]_\rho\rtimes\pi'$. In particular, we have $\pi=\soc([x',y']_\rho\rtimes\pi')$ (set $[x,y]=\emptyset$).
	\end{lemma}
	\begin{proof}
		This is a direct consequence of Lemma \ref{lemma: computation of partial Jacquet module}.
	\end{proof}
	
	\begin{lemma}\label{lemma: subquotient of parabolic induction}
		In the context of Lemma \ref{lemma: subquotient of parabolic induction - lemma}, we set $D=\Jac_{\rho,y'}\circ\Jac_{\rho,y'-\zeta'}\circ\cdots\circ\Jac_{\rho,x'}$. Then, $\sigma\mapsto D(\sigma)$ gives a bijection between the irreducible subquotients of $[x,y]_\rho\rtimes\pi$ and those of $[x,y]_\rho\rtimes\pi'$.
	\end{lemma}
	\begin{proof}
		By Lemma \ref{lemma: subquotient of parabolic induction - lemma}, we have $\pi=\soc([x',y']_\rho\rtimes\pi')=\cos([x',y']_\rho^\vee\rtimes\pi')$. Thus, there exists a canonical intertwining operator $M':[x',y']_\rho^\vee\rtimes\pi'\rightarrow [x',y']_\rho\rtimes\pi'$ whose image is $\pi$. Let $M$ be the intertwining operator defined by $$\begin{aligned}
			[x',y']_\rho^\vee\times[x,y]_\rho\rtimes\pi'\stackrel{\sim}{\longrightarrow}&[x,y]_\rho\times[x',y']_\rho^\vee\rtimes\pi'\\
			&\stackrel{[x,y]_\rho\rtimes M'}{\longrightarrow}[x,y]_\rho\times[x',y']_\rho\rtimes\pi'\stackrel{\sim}{\longrightarrow}[x',y']_\rho\times[x,y]_\rho\rtimes\pi'.
		\end{aligned}$$
		Then, the image of $M$ is isomorphic to $[x,y]_\rho\rtimes\pi$. We filter $[x,y]_\rho\rtimes\pi'$ by subrepresentations $V_i$ such that $\left\{V_i/V_{i+1}\right\}$ is the set of irreducible subquotients of $[x,y]_\rho\rtimes\pi'$. Note that $M$ is the canonical intertwining operator. Thus $M$ can be restricted to maps $[x',y']_\rho^\vee\rtimes V_i\rightarrow[x',y']_\rho\rtimes V_i$. Therefore, for every irreducible subquotient $\sigma$ of $[x,y]_\rho\rtimes\pi'$, $M$ induces a map $[x',y']_\rho^\vee\rtimes\sigma\rightarrow[x',y']_\rho\rtimes\sigma$. By Lemma \ref{lemma: subquotient of parabolic induction - lemma}, the parabolic induction $[x',y']_\rho\rtimes\sigma$ is SI, and $\soc([x',y']_\rho\rtimes\sigma)=\cos([x',y']_\rho^\vee\rtimes\sigma)$. Write $\widetilde{\sigma}=\soc([x',y']_\rho\rtimes\sigma)$. Then $M([x',y']_\rho^\vee\rtimes\sigma)=0$ or $\widetilde{\sigma}$. By Lemma \ref{lemma: subquotient of parabolic induction - lemma}, we have $D([x,y]_\rho\rtimes\pi)=[x,y]_\rho\rtimes\pi'$. Since we have proved at the beginning that the image of $M$ is $[x,y]_\rho\rtimes\pi$, we must have $M([x,y]^\vee\rtimes\sigma)=\widetilde{\sigma}$ for all $\sigma$ and $[x,y]_\rho\rtimes\pi=\sum_{\sigma\in\JH([x,y]_\rho\rtimes\pi')}\widetilde{\sigma}$ in $\mR_-(\widetilde{G})$. This completes the proof.
	\end{proof}
	
	\par The last lemma in this subsection is about the irreducibility of parabolic induction.
	
	\begin{lemma}\label{lemma: irreducibility of rho||x rtimes pi in non-negative DDR case}
		Let $x\in\frac{1}{2}\ZZ$ with $x\ge0$ and $\rho\in\Pi_{\unit,\cusp}(\GL(d_\rho))$. Suppose that $2x-1\notin[B,A]$ for all $(\rho,A,B)\in\Jord_\rho(\psi)$, and Theorem \ref{theorem: A-packets can be constructed via extended multi-segments - NDDR case} is true for $\psi$. Then, for every $\pi\in\Pi_\psi$, $\rho|\cdot|^x\rtimes\pi$ is irreducible. 
	\end{lemma}
	\begin{proof}
		M\oe glin has already proved the lemma for the classical group (see \cite[Proposition 6.3.1]{Mœeglin2009_paquets}). Thus, as in the proof of Proposition \ref{prop: non-unitary irreducibility for discrete series}, we can deduce the lemma by theta correspondence and Proposition \ref{prop: theta lifts of extended multi-segment} above.
	\end{proof}
	
	\subsection{Start of the proof}\label{subsection: start of the proof}
	
	\par In this subsection, we assume that the $(\rho,A,B)$ we fixed is the only triple of $\Jord(\psi)$ satisfying $A>B$. That is, $\psi'$ is discrete. Further, we assume that $a_{\rho,\psi',\varepsilon'}\ge 2B+1$ (see \S \ref{section: the key proposition} for the definition of $a_{\rho,\psi',\varepsilon'}$). Since $\psi$ is a non-negative DDR, this implies that $a_{\rho,\psi',\varepsilon'}>2A+1$.
	
	\subsubsection{The $A=B+1$ case}\label{subsubsection: the A=B+1 case}

	\par When $A=B+1$, Theorem \ref{theorem: construction of A-packet non-negative DDR case} can be reformulated as follows:
	
	\begin{enumerate}
		\item[(1)] If $\eta_0=1$, then $\pi(\psi,\varepsilon)=\soc([B,-A]_\rho\rtimes\pi(\psi',\varepsilon'))$.
		\item[(2)] If $\eta_0=-1$, then $\pi(\psi,\varepsilon)=\oplus_{\eta=\pm1}\pi(\psi',\varepsilon'(\rho,B+1,B+1;\eta),(\rho,B,B;-\eta))$.
	\end{enumerate}
	
	\par Note that $(2)$ is just a particular case of Theorem \ref{theorem: Mp version of Theorem 7.5 in Xu_moeglin}. Thus, we only need to consider the case when $\eta_0=1$. In this case, by Theorem \ref{theorem: Mp version of Theorem 7.5 in Xu_moeglin}, we have $$\pi(\psi,\varepsilon)=[B,-(B+1)]_\rho\rtimes\pi(\psi',\varepsilon')\ominus_{\eta=\pm1}\pi(\psi',\varepsilon',(\rho,B+1,B+1;\eta),(\rho,B,B;\eta)).$$
	By Proposition \ref{prop: cuspidal support of discrete series}, we have 
	$$\begin{aligned}
		\oplus_{\eta=\pm1}\pi(\psi',\varepsilon',(\rho,B+1,B+1;\eta),(\rho,B,B;\eta))&=\soc([B+1,-B]_\rho\rtimes\pi(\psi',\varepsilon'))\\
		&=\cos([B,-(B+1)]_\rho\rtimes\pi(\psi',\varepsilon')).
	\end{aligned}$$
	Thus, it is sufficient to analyze the irreducible subquotients of $[B,-(B+1)]_\rho\rtimes\pi(\psi',\varepsilon')$. Suppose that $\tau$ is such a subquotient. Then the key proposition (Proposition \ref{proposition: key proposition}) implies that there exists a totally ordered multi-set $\mE$ of real numbers, such that $\left\{|x|:x\in\mE\right\}=\left\{|x|:x\in[B+1,-B]\right\}$ and $\tau$ is a subrepresentation of $\sigma\rtimes\pi(\phi,\varepsilon)$, where $\sigma$ is an irreducible subrepresentation of $\times_{x\in\mE}\rho|\cdot|^x$.
	
	\par Suppose that $\tau\le\pi(\psi,\varepsilon)$. Then, Lemma \ref{lemma: highest derivative of representations in a A packet} and the condition that $a_{\rho,\psi',\varepsilon'}\ge 2B+1$ imply that $\Jac_{\rho,x}(\sigma)=0$ unless $x=B$. By Zelevinsky's classification, we can write $\sigma$ as $\sigma=\soc([x_1,y_1]_\rho\times[x_2,y_2]_\rho\times\cdots\times[x_r,y_r]_\rho)$ with $x_1\le x_2\le\cdots\le x_r$ and $y_i\le x_i$. Note that $\Jac_{\rho,x_1}(\sigma)=0$. Thus, we must have $x_1=B$. Then, there are only two possibilities:
	\begin{enumerate}
		\item[(1)] $\sigma=(\rho|\cdot|^B\times[B+1,-(B-1)]_\rho)$.
		\item[(2)] $\sigma=([B,-B]_\rho\times\rho|\cdot|^{B+1})$.
	\end{enumerate}
	Since $\rho|\cdot|^B\times[B+1,-(B-1)]_\rho$ is irreducible, the first case is excluded. In conclusion, $\tau$ is a subrepresentation $[B,-B]_\rho\rtimes\rho|\cdot|^{B+1}\rtimes\pi(\psi',\varepsilon')$. By the non-unitary irreducibility of $\pi(\psi',\varepsilon')$ (see Proposition \ref{prop: non-unitary irreducibility for discrete series}), $\rho|\cdot|^{B+1}\rtimes\pi(\psi',\varepsilon')$ is irreducible. Thus $\tau$ is a subrepresentation of $[B,-B]_\rho\rtimes\rho|\cdot|^{-(B+1)}\rtimes\pi(\psi',\varepsilon')$.
	Note that $[B,-B]_\rho\rtimes\rho|\cdot|^{-(B+1)}=[B,-(B+1)]_\rho+\cos([B,-B]_\rho\rtimes\rho|\cdot|^{-(B+1)})$ with $\Jac_{\rho,-(B+1)}(\cos([B,-B]_\rho\rtimes\rho|\cdot|^{-(B+1)}))\neq0$ in $\mR(\GL)$. Therefore, $\tau$ is a subrepresentation of $[B,-(B+1)]_\rho\rtimes\pi(\psi',\varepsilon')$. Consider $D=\Jac_{\rho,-(B+1)}\circ\Jac_{\rho,-B}\circ\cdots\circ\Jac_{\rho,B}$, then $D([B,-(B+1)]_\rho\rtimes\pi(\psi',\varepsilon'))=\pi(\psi',\varepsilon')$ is irreducible. This implies that $[B,-(B+1)]_\rho\rtimes\pi(\psi',\varepsilon')$ is SI, which completes the proof in the $A=B+1$ case.
	
	\subsubsection{First construction}\label{subsubsection: first construction}
	
	\par From now on, we assume that $A>B+1$. The aim of this subsection is to prove the following lemma:
	
	\begin{lemma}\label{lemma: first construction}
		Suppose that $\pi$ is an irreducible constituent of $\pi(\psi,\varepsilon)$. Then, there exist an integer $t_1\in[0,\frac{A-B}{2}]$ and a sign $\lambda\in\left\{\pm1\right\}$ such that $\pi$ is a subquotient of an induced representation of the form:
		$$\bigtimes_{j\in[0,t_1-1]}(\rho|\cdot|^{A-2j}\times\cdots\times\rho|\cdot|^{-(A-2j-1)})\rtimes\pi(\psi',\varepsilon',\cup_{C\in[B,A-2t_1]}(\rho,C,C;(-1)^{[C]}\lambda)),$$ where $[0,t_1-1]$ is empty if $t_1=0$ and $[B,A-2t_1-2]$ is empty if $t_1=\frac{A-B}{2}$. Furthermore, $\lambda$ and $t_1$ necessarily satisfy $\eta_0=\prod_{C\in[B,A-2t_1]}(-1)^{[C]}\lambda$.
	\end{lemma}
	\begin{proof}
		By Theorem \ref{theorem: Mp version of Theorem 7.5 in Xu_moeglin} and Lemma \ref{lemma: Proposition 4.4 of Moeglin}, either there exists a $\eta=\pm1$ such that $\pi$ is a subquotient of $\pi(\psi',\varepsilon',(\rho,A,B+1;\eta),(\rho,B,B;\eta\eta_0))$, or there exists $j\in[1,A-B]$ such that $\pi$ is a subquotient of the induced representation
		$$[B,-(B+j)]_\rho\times[B+j+1,A]_\rho\rtimes\pi(\psi',\varepsilon',(\rho,A-1,B+1;\eta_0)).$$
		In the second case, $\pi$ is actually a subquotient of the induced representation
		$$\rho|\cdot|^A\times\rho|\cdot|^{A-1}\times\cdots\times\rho|\cdot|^{-B}\rtimes\pi(\psi',\varepsilon',(\rho,A-1,B+1;\eta_0)).$$
		By applying Lemma \ref{lemma: Proposition 4.4 of Moeglin} again, we have $$\pi(\psi',\varepsilon',(\rho,A-1,B+1;\eta_0)=\soc([B+1,A-1]\rtimes\pi(\psi',\varepsilon',(\rho,A-2,B;\eta_0))).$$ Thus $\pi$ is an irreducible subquotient of the induced representation $$\rho|\cdot|^A\times\rho|\cdot|^{A-1}\times\cdots\times\rho|\cdot|^{-(A-1)}\rtimes\pi(\psi',\varepsilon',(\rho,A-2,B;\eta_0)).$$ By induction on $\sum_{(\rho,A_i,B_i)\in\Jord(\psi)}(A_i-B_i)$, we may assume that the lemma is true for $\pi(\psi',\varepsilon',(\rho,A-2,B;\eta_0))$. This directly implies that the lemma is true for $\pi$.
		
		\par Now, we assume that $\pi$ is a subquotient of $\pi(\psi',\varepsilon',(\rho,A,B+1;\eta),(\rho,B,B;\eta\eta_0))$. By the induction hypothesis, there exist an integer $t'\in[0,\frac{A-(B+1)}{2}]$ and a sign $\lambda\in\left\{\pm1\right\}$ such that $\pi$ is a subquotient of an induced representation of the form $$\bigtimes_{j\in[0,2t'-1]}(\cdots)\rtimes\pi(\psi',\varepsilon',\cup_{C\in(B,A-2t_1]}(\rho,C,C;(-1)^{[C]}\lambda),(\rho,B,B;\eta\eta_0)),$$ with $(\cdots)=\rho|\cdot|^{A-2j}\times\cdots\times\rho|\cdot|^{-(A-2j-1)}$ and $\eta=\prod_{C\in(B,A-2t']}(-1)^{[C]}\lambda$.
		
		\par If $(-1)^{B+1}\lambda=-\eta\eta_0$ or $B+1=A-2t'$, the lemma is deduced by setting $t_1=t'$. Thus, we may assume that $(-1)^{B+1}\lambda=\eta\eta_0$ and $B+1<A-2t'$. By Proposition \ref{prop: cuspidal support of discrete series}, we know that $\pi(\psi',\varepsilon',\cup_{C\in(B,A-2t_1]}(\rho,C,C;(-1)^{[C]}\lambda),(\rho,B,B;\eta\eta_0))$ can be embedded into $$[B+1,-B]_\rho\rtimes\pi(\psi',\varepsilon',\cup_{C\in[B+2,A-2t_1]}(\rho,C,C;(-1)^{[C]}\lambda)).$$ Further, by applying Proposition \ref{prop: cuspidal support of discrete series} again, $\pi(\psi',\varepsilon',\cup_{C\in[B+2,A-2t_1]}(\rho,C,C;(-1)^{[C]}\lambda))$ can be embedded into $$\rho|\cdot|^{B+2}\times\rho|\cdot|^{B+3}\times\cdots\times\rho|\cdot|^{A-2t'}\rtimes\pi(\psi',\varepsilon',\cup_{C\in[B+1,A-2t_1-1]}(\rho,C,C;(-1)^{[C]}\lambda))$$ and $\pi(\psi',\varepsilon',\cup_{C\in[B+1,A-2t_1-1]}(\rho,C,C;(-1)^{[C]}\lambda))$ can be embedded into $$\rho|\cdot|^{B+1}\times\rho|\cdot|^{B+2}\times\cdots\times\rho|\cdot|^{A-2t'-1}\rtimes\pi(\psi',\varepsilon',\cup_{C\in[B,A-2t_1-2]}(\rho,C,C;(-1)^{[C]}\lambda)).$$ Note that $\prod_{C\in[B,A-2t'-2]}(-1)^{[C]}\lambda=\prod_{C\in[B+2,A-2t']}(-1)^{[C]}=\frac{\eta}{\eta\eta_0}=\eta_0$. Thus we obtain the lemma by setting $t_1=t'+1$.
	\end{proof}
	
	\subsubsection{Second construction}
	
	\par We denote by $\widetilde{\psi}$ the parameter deduced from $\psi$ by removing $(\rho,A,B)$ and all triples $(\rho,C,C)$ for $C<B$. Let $\widetilde{\varepsilon}$ be the restriction of $\varepsilon$ on $\Jord(\widetilde{\psi})$. For $D\in\frac{1}{2}\ZZ$ with $A-D\in\ZZ_{\ge0}$ and $\lambda\in\left\{\pm1\right\}$, we let $(\psi_{D,\lambda},\varepsilon_{D,\lambda})$ be obtained from $(\widetilde{\psi},\widetilde{\varepsilon})$ by adding triples $(\rho,C,C)$ for all $D-C\in\ZZ_{\ge0}$ and set $\varepsilon_{\lambda,D}(\rho,C,C)=(-1)^{[C]}\lambda$. In particular, we write $(\psi_{-\frac{1}{2},\lambda},\varepsilon_{-\frac{1}{2},\lambda})=(\widetilde{\psi},\widetilde{\varepsilon})$ if $A\notin\ZZ$ and $(\psi_{-1,\lambda},\varepsilon_{-1,\lambda})=(\widetilde{\psi},\widetilde{\varepsilon})$ if $A\in\ZZ$.
	
	\begin{lemma}\label{lemma: second construction}
		Let $\pi$ be an irreducible subquotient of $\pi(\psi,\varepsilon)$. Then, there exist a totally ordered multi-set $\mE$ of real numbers, a $D\in\frac{1}{2}\ZZ$ and $\lambda\in\left\{\pm1\right\}$ such that the following holds:
		\begin{enumerate}
			\item[(1)] $(\psi_{D,\lambda},\varepsilon_{D,\lambda})$ is well-defined.
			\item[(2)] $\prod_{C\le D}(-1)^{[C]}\lambda=\eta_0\prod_{E<B;(\rho,E,E)\in\Jord(\psi)}\varepsilon(\rho,E,E)$ if $\rho\neq\triv$ (this is a metaplectic feature, but it does not affect M\oe glin's arguments).
			\item[(3)] $\pi\xhookrightarrow{}\times_{x\in\mE}\rho|\cdot|^x\rtimes\pi(\psi_{D,\lambda},\varepsilon_{D,\lambda})$.
			\item[(4)] $\mE\cup-\mE=(\bigcup_{E\in[B,A]}[-E,E])\cup(\bigcup_{E<B;(\rho,E,E)\in\Jord(\psi)}[-E,E])-(\bigcup_{C\le D}[-C,C])$.
		\end{enumerate}
		In particular, we may choose $\lambda$ such that $b_{\rho,\psi_{D,\lambda},\varepsilon_{D,\lambda}}=2D+1$, in which case we will write $(\psi_t,\varepsilon_t)=(\psi_{D,\lambda},\varepsilon_{D,\lambda})$ with $t=\frac{A-D}{2}$.
	\end{lemma}
	\begin{proof}
		By Proposition \ref{prop: cuspidal support of discrete series} and Lemma \ref{lemma: first construction}, there exist $\mE$, $D$, and $\lambda$ satisfying $(1)$, $(3)$ and for which $\pi$ is a subquotient of $\times_{x\in\mE}\rho|\cdot|^x\rtimes\pi(\psi_{D,\lambda},\varepsilon_{D,\lambda})$. Now the lemma follows from the key proposition (Proposition \ref{proposition: key proposition}).
	\end{proof}
	
	\par Furthermore, under the condition that $b_{\rho,\psi_{D,\lambda},\varepsilon_{D,\lambda}}=2D+1$, the $D, \lambda$ in Lemma \ref{lemma: second construction} are uniquely determined by $\pi$:
	
	\begin{lemma}\label{lemma: second construction - converse}
		Let $\pi$ be an irreducible subquotient of $\pi(\psi,\varepsilon)$. Suppose that there exist a set $\mE'$ of real numbers, a $D'\in\frac{1}{2}\ZZ$ and $\lambda'\in\left\{\pm1\right\}$ such that $\pi$ is a subquotient of $\times_{x\in\mE'}\rho|\cdot|^x\rtimes\pi(\psi_{D',\lambda'},\varepsilon_{D',\lambda'})$ and $b_{\rho,\psi_{D',\lambda'},\varepsilon_{D',\lambda'}}=2D'+1$. Then we have:
		\begin{enumerate}
			\item[(1)] $\mE'\cup-\mE'=(\bigcup_{E\in[B,A]}[-E,E])\cup(\bigcup_{E<B;(\rho,E,E)\in\Jord(\psi)}[-E,E])-(\bigcup_{C\le D}[-C,C])$.
			\item[(2)] $D'$ and $\lambda'$ are uniquely determined by $\pi$.
		\end{enumerate}
	\end{lemma}
	\begin{proof}
		Take $\mE$, $D$ and $\lambda$ satisfying the four conditions of Lemma \ref{lemma: second construction}. By the key proposition (Proposition \ref{proposition: key proposition}), we can order $\mE'$ such that $\pi$ is a subrepresentation $\times_{x\in\mE'}\rho|\cdot|^x\rtimes\pi(\psi_{D',\lambda'},\varepsilon_{D',\lambda'})$. Then, we have $\circ_{x\in\mE'}\Jac_{\rho,x}\pi\neq 0$. This implies that $\circ_{x\in\mE'}\Jac_{\rho,x}(\times_{x\in\mE}\rho|\cdot|^x\rtimes\pi(\psi_{D,\lambda},\varepsilon_{D,\lambda}))\neq0$, which is possible only when $\mE\cup-\mE=\mE'\cup-\mE'$. Now, we have $m'\pi(\psi_{D',\lambda'},\varepsilon_{D',\lambda'})=\circ_{x\in\mE'}\Jac_{\rho,x}\pi=\circ_{x\in\mE'}\Jac_{\rho,x}(\times_{x\in\mE}\rho|\cdot|^x\rtimes\pi(\psi_{D,\lambda},\varepsilon_{D,\lambda}))=m\pi(\psi_{D,\lambda},\varepsilon_{D,\lambda})$ with $m,m'\in\ZZ_{\ge 1}$, which directly implies $D'=D$ and $\lambda'=\lambda$.
	\end{proof}
	
	\par We end this subsection with a technical lemma:
	
	\begin{lemma}\label{lemma: first step - lemma 1}
		Suppose that $\mE$, $D$ and $\lambda$ satisfy the four conditions in Lemma \ref{lemma: second construction} with $b_{\rho,\psi_{D,\lambda},\varepsilon_{D,\lambda}}=2D+1$. If $-A\in\mE$, then we have $\Jac_{\rho,-A}\circ\Jac_{\rho,-A+1}\circ\cdots\circ\Jac_{\rho,B}\pi\neq0$.
	\end{lemma}
	\begin{proof}
		By Lemma \ref{lemma: second construction}, there exists an irreducible representation $\sigma$ of $\GL(d_\sigma)$ such that $\Supp(\sigma)=\mE$ and $\pi\xhookrightarrow{}\sigma\rtimes\pi(\psi_{D,\lambda},\varepsilon_{D,\lambda})$. By Zelevinsky's classification, we may write $\sigma=\soc([x_1,y_1]_\rho\times\cdots\times[x_r,y_r]_\rho)$ with $x_1\le\cdots\le x_r$ and $x_i\ge y_i$. Note that $-A$ is the minimal element of $\mE$, thus $y_i=-A$ for some $i$ and $\sigma=[x_i,-A]\times\soc([x_1,y_1]_\rho\times\cdots\times[x_{i-1},y_{i-1}]_\rho\times[x_{i+1},y_{i+1}]_\rho\times\cdots\times[x_r,y_r]_\rho)$. In particular, we have $\Jac_{\rho,-A}\circ\Jac_{\rho,-A+1}\circ\cdots\circ\Jac_{\rho,x_i}\pi\neq0$. By Lemma \ref{lemma: highest derivative of representations in a A packet}, the only possibility is $x_i=B$. This completes the proof.
	\end{proof}
	
	\subsubsection{Maximal Order}
	
	\par In this subsection, we fix $\pi$ to be an irreducible constituent of $\pi(\psi,\varepsilon)$ and we take $\mE$, $D$ and $\lambda$ satisfying the four conditions of Lemma \ref{lemma: second construction} and we assume that $b_{\rho,\psi_{D,\lambda},\varepsilon_{D,\lambda}}=2D+1$. We denote by $>_\mE$ the total order on $\mE$, reserving $>$ for the usual order on real numbers.
	
	\par We say that $>_\mE$ is a maximal order if the following conditions hold:
	\begin{enumerate}
		\item[(1)] $\mE=\cup_{i=1}^r[x_i,y_i]$ with $x_i<_\mE x_i+1<_\mE\cdots<_\mE y_i$ and $y_i<_\mE x_{i+1}$ for all $1\le i\le r$.
		\item[(2)] $x_1\ge x_2\ge\cdots\ge x_r$ and $x_i\le y_i$. For $i<j$ with $x_i=x_j$, we have $y_i\ge y_j$.
		\item[(3)] $\pi$ can be embedded into $\soc([x_1,y_1]_\rho\times\cdots\times[x_r,y_r]_\rho)\rtimes\pi(\psi_{D,\lambda},\varepsilon_{D,\lambda})$.
	\end{enumerate}
	In particular, when the condition $(1)$ is satisfied, we write $\sigma_{>_\mE}=\soc([x_1,y_1]_\rho\times\cdots\times[x_r,y_r]_\rho)$ (by Zelevinsky's classification, if the condition $(2)$ holds, we know that $\sigma_{>_\mE}$ is irreducible). Thus the condition $(3)$ can be reformulated as $\pi\xhookrightarrow{}\sigma_{>_\mE}\rtimes\pi(\psi_{D,\lambda},\varepsilon_{D,\lambda})$. Now, we will prove that the maximal order on $\mE$ is uniquely determined by $\mE$:
	
	\begin{lemma}\label{lemma: maximal order}
		In the above setting, there exists a unique maximal order on $\mE$. Precisely, we have $x_1=B$, $x_{i+1}=x_i-1$ and $y_1>y_2>\cdots>y_r$. It is not hard to see that these conditions uniquely determine the maximal order $>_\mE$ on $\mE$.
	\end{lemma}
	\begin{proof}
		The existence of maximal order is provided by Zelevinsky's classification. We only need to prove that the maximal order is uniquely determined by $\pi$. Fix a maximal order $>_\mE$ on $\mE$. We first prove that $x_{i+1}=x_i-1$ for all $i$. If $x_{i+1}\neq x_i-1$, we denote $\mE_i$ the set obtained from $\mE$ by removing $x_{i+1}$. We assume that $i$ is minimal with the property $x_{i+1}\neq x_i-1$. If $x_i\neq x_{i+1}$, then the restriction of $>_\mE$ on $\mE_i$ still satisfies the conditions $(1)$, $(2)$ above. If $x_i=x_{i+1}$, then $y_i\ge y_{i+1}$ by the condition $(1)$ above. Therefore $[x_i,y_i]_\rho$ and $[x_{i+1}+1,y_{i+1}]$ is not linked, and the representation $\sigma_{>_{\mE_i}}$ is still irreducible. In all cases, we have $\sigma_{>_\mE}\xhookrightarrow{}\rho|\cdot|^{x_{i+1}}\times\sigma_{>_{\mE_i}}$. This implies that $\Jac_{\rho,x_1}\Jac_{\rho,x_{i+1}}(\pi)\neq0$. In particular, we have $\Jac_{\rho,x_1}(\pi),\Jac_{\rho,x_{i+1}}(\pi)\neq0$. By Lemma \ref{lemma: highest derivative of representations in a A packet}, we have $x_1=x_{i+1}=B$. However, Lemma \ref{lemma: highest derivative of representations in a A packet} also implies that $\Jac_{\rho,B}\Jac_{\rho,B}(\pi)=0$. Thus we have $x_1=B$ and $x_{i+1}=x_i-1$ for all $i$.
		
		\par Now, we show that $y_i>y_{i+1}$ for all $i$. Otherwise, we fix a minimal $i$ such that $y_i\le y_{i+1}$. Then $[x_i,y_i]_\rho$ and $[x_{i+1},y_{i+1}]_\rho$ are not linked. Thus $\sigma_{>_\mE}$ can be embedded into $$\rho|\cdot|^{x_{i+1}}\times\soc(\cdots\times[x_{i-1},y_{i-1}]_\rho\times[x_{i+1}+1,y_{i+1}]_\rho\times[x_i,y_i]_\rho\times[x_{i+2},y_{i+2}]_\rho\times\cdots).$$ From this we deduce that $\Jac_{\rho,x_{i+1}}(\pi)\neq 0$ and thus we have $x_{i+1}=B$, which is impossible.
	\end{proof}
	
	\subsection{An important particular case}\label{subsection: an important particular case}
	
	\par In this subsection, we keep the hypotheses of \S \ref{subsubsection: first construction} and assume further that $b_{\rho,\psi',\varepsilon'}=2B-1$. We write $$C_{\min}=\begin{cases}
		0&\text{if } A, B\in\ZZ\\
		-\frac{1}{2}&\text{if } A,B\notin\ZZ\\
	\end{cases}.$$
	Further, if $A,B\notin\ZZ$, we set $$\varepsilon(\rho,-\frac{1}{2},-\frac{1}{2})=\begin{cases}
		1&\text{if }\rho\neq\triv\\
		-1&\text{if }\rho=\triv\\
	\end{cases}.$$
	
	\subsubsection{Complementary terms}\label{subsubsection: complementary terms}
	
	\par We fix an irreducible constituent $\pi$ of $\pi(\psi,\varepsilon)$. By Lemma \ref{lemma: second construction - converse}, under the condition that $b_{\rho,\psi_{D,\lambda},\varepsilon_{D,\lambda}}=2D+1$, the $\lambda$ and $D$ in Lemma \ref{lemma: second construction} are uniquely determined by $\pi$. In this subsection, we will determine the $\lambda$ and $D$ for the complementary terms of $\pi(\psi,\varepsilon)$, i.e., the terms that occur in $$\bigoplus_{\eta=\pm1\atop\eta_0=\eta^{A-B+1}(-1)^{\frac{(A-B+1)(A-B)}{2}}}\pi(\psi',\varepsilon',\cup_{B\le C\le A}(\rho,C,C;(-1)^{C-B}\eta)).$$
	
	\par We take the first construction in Lemma \ref{lemma: first construction} for $\pi$, from which we get $t_1$ and $\lambda_1$. Then, we do the second construction in Lemma \ref{lemma: second construction}, from which we get $\mE, D$ and $\lambda$. Write $t=\frac{A-D}{2}$. By the proof of Lemma \ref{lemma: second construction}, we can choose $D$ satisfying the following conditions:
	\begin{enumerate}
		\item[(1)] If $\varepsilon(\rho,B-1,B-1)(-1)^{[B]}\lambda=-1$, then $t=t_1$.
		\item[(2)] If $\varepsilon(\rho,B-1,B-1)(-1)^{[B]}\lambda=1$, then $t=t_1+\inf(B,A-2t_1-B+1)$.
	\end{enumerate}
	
	\par Suppose that $B>0$ and $D=A$, which implies $t_1=0$ and $\varepsilon(\rho,B-1,B-1)(-1)^{[B]}\lambda=-1$. It follows from the definition of $t_1$ that $\pi$ is in the complementary term $$\pi(\psi',\varepsilon',\cup_{B\le C\le A}(\rho,C,C;(-1)^{B-C}\eta))$$ with $\eta=-\varepsilon(\rho,B-1,B-1)$ (note that $\prod_{B\le C\le A}(-1)^{B-C}\eta=\eta^{A-B+1}(-1)^{\frac{(A-B+1)(A-B)}{2}}$, the sign condition in Theorem \ref{theorem: construction of A-packet non-negative DDR case} does hold for $\pi$).
	
	\par Suppose that $B>0$ and $D=A-2\inf(B,A-B+1)$. We further assume that $\lambda=-\varepsilon(\rho,C_{\min},C_{\min})$ if $B<A-B+1$. In this case, we have $t=\inf(B,A-B+1)$. Note that we have $t_1\le[\frac{A-B}{2}]\le\frac{A-B+1}{2}$. Thus, if $A-B+1\le B$, we have $t=A-B+1>t_1$, hence $t=t_1+A-2t_1-B+1=A-B-t_1+1$, hence $t_1=0$. If $B<A-B+1$, we have $t=B$ and either $t_1=0$ or $A-2t_1-B+1<B$. If $A,B\in\ZZ$, by the choice of $\lambda$, we must have $B\le A-2t_1-B+1$. If $A,B\notin\ZZ$, then $t_1=t-\inf(A-2t_1-B+1)\notin\ZZ$, which gives a contradiction. In conclusion, we have $t_1=0$. Thus $\pi$ is in the complementary term $\pi(\psi',\varepsilon',\cup_{B\le C\le A}(\rho,C,C;(-1)^{B-C}\eta))$ with $\eta=\varepsilon(\rho,B-1,B-1)$.
	
	\par In conclusion, we have:
	\begin{enumerate}
		\item[(1)] When $B>0$ and $\eta=\varepsilon(\rho,B-1,B-1)$, then $\pi=\pi(\psi',\varepsilon',\cup_{B\le C\le A}(\rho,C,C;(-1)^{B-C}\eta))$ if and only if $D=A-2\inf (B,A-B+1)$ and $$\lambda=\begin{cases}
			\varepsilon(\rho,C_{\min},C_{\min})&\text{if }A, B\in\ZZ\text{ and } A-B+1\le B\\
			-\varepsilon(\rho,C_{\min},C_{\min})&\text{if }A, B\notin\ZZ\text{ or } A-B+1>B\\
		\end{cases}.$$
		\item[(2)] When $B>0$ and $\eta=-\varepsilon(\rho,B-1,B-1)$, then $\pi=\pi(\psi',\varepsilon',\cup_{B\le C\le A}(\rho,C,C;(-1)^{B-C}\eta))$ if and only if $D=A$ and $\lambda=(-1)^{[B-1]}\varepsilon(\rho,B-1,B-1)$.
		\item[(3)] When $B=0$, then $\pi=\pi(\psi',\varepsilon',\cup_{B\le C\le A}(\rho,C,C;(-1)^{B-C}\eta))$ if and only if $D=A$.
	\end{enumerate}
	In particular, the arguments above imply that if $\pi$ is one of the complementary terms, we must have $t_1=0$.
	
	\begin{lemma}\label{lemma: complementary terms}
		Suppose that $\prod_{C\in[B,A]}(-1)^{[C]}\lambda=\eta_0$. Then the representation $\pi_\lambda:=\pi(\psi',\varepsilon',\cup_{C\in[B,A]}(\rho,C,C,(-1)^{[C]}\lambda))$ occurs with multiplicity exactly $1$ in $\pi(\psi,\varepsilon)$.
	\end{lemma}
	\begin{proof}
		By the proof of Lemma \ref{lemma: first construction}, if $\pi$ is an irreducible subquotient of $$\bigoplus_{C\in\left(B,A\right]}(-1)^{A-C}[B,-C]\rtimes\Jac_{\rho,C}\circ\Jac_{\rho,C-1}\circ\cdots\circ\Jac_{\rho,B+2}\pi(\psi',\varepsilon',(\rho,A,B+2;\eta_0)),$$ we have $t\ge t_1>0$. By the arguments above, $\pi$ can not be isomorphic to a representation of the form $\pi_\lambda$. Thus, it is sufficient to prove that $\pi_\lambda$ occurs with multiplicity $1$ in $$\bigoplus_{\eta=\pm1}(-1)^{[\frac{A-B+1}{2}]}\eta^{A-B+1}\eta_0^{A-B}\pi(\psi',\varepsilon',(\rho,A,B+1;\eta),(\rho,B,B;\eta\eta_0)).$$
		
		\par By induction hypothesis, for $\eta$ and $\lambda'$ such that $\prod_{B+1\le C\le A}(-1)^{[C]}\lambda'=\eta$, the representation $$\pi_{\eta,\lambda'}:=\pi(\psi',\varepsilon',\cup_{C\in[B+1,A]}(\rho,C,C;(-1)^{[C]}\lambda'),(\rho,B,B;\eta\eta_0))$$ occurs with multiplicity $1$ in $\pi(\psi',\varepsilon',(\rho,A,B+1;\eta),(\rho,B,B;\eta\eta_0))$. The $\pi_\lambda$ can only be one of these $\pi_{\eta,\lambda'}$ with $\lambda(-1)^{[B]}=\eta\eta_0$ and $(-1)^{[B+1]}\lambda'=-\eta\eta_0$. This implies that $\lambda=\lambda'$.
		
		\par Suppose that $A-B+1$ is even. We then have $\prod_{C\in[B,A]}(-1)^{[C]}=\eta_0$ and for each $\lambda$ there exists a unique $\eta$ such that $\pi_\lambda=\pi_{\eta,\lambda}$. Note that $\prod_{C\in[B,A]}(-1)^{[C]}=(-1)^{\frac{([A]+[B])(A-B+1)}{2}}=(-1)^{\frac{([A]-[B])(A-B+1)}{2}}=(-1)^{\frac{A-B+1}{2}}$ in this case. We have $(-1)^{[\frac{A-B+1}{2}]}\eta^{A-B+1}\eta_0^{A-B}=1$, which completes the proof in this case.
		
		\par Suppose that $A-B+1$ is odd. Then $\lambda=\eta_0(-1)^{\frac{([A]+[B])(A-B+1)}{2}}$. Set $\eta=(-1)^{[B]}\lambda\eta_0$. Then $\prod_{B+1\le C\le A}(-1)^{[C]}\lambda=(-1)^{[B]}\lambda\eta_0=\eta$. Thus $\pi_{\eta,\lambda}$ is well-defined and $\pi_\lambda=\pi_{\eta,\lambda}$. Now, we have $$(-1)^{[\frac{A-B+1}{2}]}\eta^{A-B+1}\eta_0^{A-B}=(-1)^{\frac{A-B}{2}}(-1)^{[B]}\lambda\eta_0=(-1)^{\frac{[A]+[B]}{2}}(-1)^{\frac{([A]+[B])(A-B+1)}{2}}=1.$$ This completes the proof. 
	\end{proof}
	
	\subsubsection{Non-vanishing of certain Jacquet module}
	
	\par For an irreducible constituent $\pi$ of $\pi(\psi,\varepsilon)$, we choose $\mE, D$ and $\lambda$ satisfying the four conditions of Lemma \ref{lemma: second construction} with $b_{\rho,\psi_{D,\lambda},\varepsilon_{D,\lambda}}=2D+1$ and let $>_\mE$ be the maximal order on $\mE$. By Lemma \ref{lemma: maximal order}, we can write $\mE$ in the form of a tableau
	$$\begin{ytableau}
		\scriptstyle x_1=B&\cdots&\cdots&\cdots&\cdots&\cdots&y_1\\
		x_2&\cdots&\cdots&\cdots&y_2&\none&\none\\
		\vdots&\vdots&\vdots&\vdots&\none&\none&\none\\
		x_r&\cdots&y_r&\none&\none&\none&\none
	\end{ytableau}$$ such that each row and column of the tableau are segments and $x_1<_\mE x_1+1<_\mE\cdots<_\mE y_1<_\mE x_2<_\mE x_2+1<_\mE\cdots<_\mE y_2<_\mE\cdots<_\mE y_r$.
	We denote by $s$ the number of columns and $z_1,\dots,z_s$ the last elements of each column. Note that for $x\ge z$ with $x\ge 0$, we have $$[x,z]\cup[-z,-x]=\begin{cases}
		[x,-x]\cup[z,-z]&\text{if }z\le 0\\
		[x,-x]-[z-1,-(z-1)]&\text{if }z>0\\
	\end{cases}.$$ Thus, we have:
	$$\mE\cup-\mE=\bigcup_{x\in[B,y_1]}[x,-x]\cup\bigcup_{z_s\le0}[z_s,-z_s]-\bigcup_{z_s>0}[-(z_s-1),z_s-1].$$
	By the conditions of Lemma \ref{lemma: second construction}, we have also $\mE\cup-\mE=\bigcup_{D+1\le z\le A}[-z,z]$. By comparing the two expressions, we deduce that
	$$[B,y_1]\cup\bigcup_{z_s\le0}\left\{-z_s\right\}=[D+1,A]\cup\bigcup_{z_s>0}\left\{z_s-1\right\}.$$
	
	\begin{lemma}\label{lemma: non-vanishing of certain Jacquet module}
		Write $t=\frac{A-D}{2}$. Suppose that $t>0$. Then, we can choose $\mE$ satisfying one of the following properties:
		\begin{enumerate}
			\item[(1)] $-A\in\mE$.
			\item[(2)] $t=B$ and $\mE=\begin{ytableau}
				\scriptstyle x_1=B&\cdots&A\\
				\vdots&&\vdots\\
				\scriptstyle -B+1&\cdots&\scriptstyle D+1\\
			\end{ytableau}$.
			\item[(3)] $t=A-B+1$ and $\mE=\begin{ytableau}
			\scriptstyle x_1=B&\cdots&A\\
			\vdots&&\vdots\\
			\scriptstyle-B+1&\cdots&\scriptstyle-(D+1)\\
			\end{ytableau}$.
			\item[(4)] $\mE=\mE_D:=\begin{ytableau}
			\scriptstyle x_1=B&\cdots&\cdots&\cdots&A\\
			\vdots&&\vdots&&\vdots\\
			\scriptstyle B-2t+1&\cdots&\scriptstyle B-t&\cdots&\scriptstyle D+1\\
			\vdots&&\vdots&\none&\none\\
			\scriptstyle-B+1&\cdots&\scriptstyle-(B-t)&\none&\none
			\end{ytableau}$.
		\end{enumerate}
		In particular, if $\mE$ has the property of having a minimal number of strictly positive elements and $\Jac_{\rho,-A}\circ\Jac_{\rho,-A+1}\circ\cdots\circ\Jac_{\rho,B}\pi=0$, then we have $t\le B$.
	\end{lemma}
	\begin{proof}
		Note that $(2)$ and $(3)$ are just special cases of $(4)$. Thus, by Lemma \ref{lemma: first step - lemma 1} above, it is sufficient to prove that, if $\mE$ has a minimal number of strictly positive elements and $-A\notin\mE$, then $\mE$ satisfies the condition $(4)$ above. Since $A\in\mE$, by \ref{lemma: maximal order}, we must have $y_1=A$.
		
		\par We first prove that, for $1\le i\le r$, if $y_i>0$, then either $y_{i+1}=y_i-1$ or $y_i=D+1$. Suppose that $y_{i+1}\neq y_i-1$. Denote by $\mE_i$ the set obtained from $\mE$ by removing $y_i$. Then, we have $\pi\xhookrightarrow{}\sigma_{>_{\mE_i}}\times\rho|\cdot|^{y_i}\rtimes\pi(\psi_{D,\lambda},\varepsilon_{D,\lambda})$. If $y_i\neq D+1$, then Proposition \ref{prop: non-unitary irreducibility for discrete series} implies that $\rho|\cdot|^{y_i}\rtimes\pi(\psi_{D,\lambda},\varepsilon_{D,\lambda})$ is irreducible. Thus we can replace $\mE$ by $\mE'=\mE\cup\left\{-y_i\right\}-\left\{y_i\right\}$, which contradicts the condition that $\mE$ has a minimal number of strictly positive elements.
		
		\par By the discussion above, we have $$[B,A]\cup\bigcup_{z_s\le 0}\left\{-z_s\right\}=[D+1,A]\cup\bigcup_{z_s>0}\left\{z_s-1\right\}.$$ Note that the number of columns of $\mE$ is $A-B+1=b$ and $z_1<z_2<\cdots<z_b$. If $z_b>0$ with $z_b\neq D+1$, then the previous paragraph implies that $z_{i+1}=z_i+1$ for all $1\le i\le b$. Thus we have $$\cup_{z_s\le 0}\left\{-z_s\right\}=\begin{cases}
		[D+1,A]\cup[z_1-1,z_b-1]-[B,A]&\text{if }z_1>0\\
		[D+1,A]\cup[0,z_b-1]-[B,A]&\text{if }z_1\le 0,\\
		\end{cases}$$ which leads to a contradiction. Thus, we have $z_b=D+1$ or $z_b\le 0$. We can further deduce from the previous paragraph that the set $\left\{z_s-1:z_s>0\right\}$ is a segment of the form $[z,D]$, where $0\le z\le D+1$ (we set $[D+1,D]=\emptyset$ here). Thus, we have 		$$[B,A]\cup\bigcup_{z_s\le 0}\left\{-z_s\right\}=[z,A].$$
		Hence $z\le B$ and $\cup_{z_s\le 0}\left\{-z_s\right\}=[z,B-1]$. Therefore, $\mE$ has the form
		$$\begin{ytableau}
			\scriptstyle x_1=B&\cdots&\cdots&\cdots&A\\
			\vdots&&\vdots&&\vdots\\
			\scriptstyle B-2t+1&\cdots&z&\cdots&\scriptstyle D+1\\
			\vdots&&\vdots&\none&\none\\
			\scriptstyle -B+1&\cdots&-z&\none&\none
		\end{ytableau}$$
		In particular, we have $z-(B-2t+1)=-z-(-B+1)\Rightarrow z=B-t$. This completes the proof.
	\end{proof}
	
	\begin{corollary}\label{coro: B<1 case}
		If $B<1$, then Theorem \ref{theorem: construction of A-packet non-negative DDR case} holds for $\pi(\psi,\varepsilon)$. 
	\end{corollary}
	\begin{proof}
		It is sufficient to prove the $(1)$, $(2)$ in the remark after Lemma \ref{lemma: lemma on the main term}. When $B<1$, the $(1)$ is a consequence of Lemma \ref{lemma: non-vanishing of certain Jacquet module}, while the $(2)$ is just Lemma \ref{lemma: complementary terms}.
	\end{proof}
	
	\par By the advantage of Corollary \ref{coro: B<1 case}, we will assume $B\ge1$ in the following subsections.	
	
	\subsubsection{The $D<B-1$ case}
	
	\par We keep the notation of \S \ref{subsubsection: complementary terms}. We further assume that $D<B-1$ in this subsection.
	
	\begin{lemma}\label{lemma: D<B-1 case}
		Suppose that there exists a $C\in(B,A]$ such that $\pi$ is an irreducible subquotient of $$X_C=[B,-C]\rtimes\Jac_{\rho,C}\circ\Jac_{\rho,C-1}\circ\cdots\circ\Jac_{\rho,B+2}\pi(\psi',\varepsilon',(\rho,A,B+2;\eta_0)).$$
		Then we have $\circ_{x\in\mE_D}\Jac_{\rho,x}\pi=0$ (for definition of $\mE_D$, see Lemma \ref{lemma: non-vanishing of certain Jacquet module}). Here $\circ_{x\in\mE_D}\Jac_{\rho,x}$ means taking Jacquet module with respect to elements of the tableau $\mE_D$ row by row from left to right. Note that $\Jac_{\rho,x}\Jac_{\rho,y}=\Jac_{\rho,y}\Jac_{\rho,x}$ whenever $|x-y|\neq 1$. Thus, $\circ_{x\in\mE_D}\Jac_{\rho,x}$ can also be computed by taking Jacquet module with respect to elements of the tableau $\mE_D$ column by column from top to down.
	\end{lemma}
	\begin{proof}
		Let $l\in\ZZ_{\ge1}$ and set 
		$$\mA_l:=\begin{ytableau}
			B&\scriptstyle B+1&\cdots&A\\
			\vdots&&&\vdots\\
			\scriptstyle B-l+1&\scriptstyle B-l+2&\cdots&\scriptstyle A-l+1\\
			\scriptstyle B-l&\none&\none&\none\\
			\vdots&\none&\none&\none\\
			\scriptstyle -B+1&\none&\none&\none\\
		\end{ytableau}.$$
		If $l\le C-B+1$, we set 
		$$\mA_{C,l}:=\begin{ytableau}
			\scriptstyle B+2&\cdots&&\cdots\cdots&A\\
			\vdots&&&&\vdots\\
			\scriptstyle B+3-l&\cdots&\cdots&\cdots&\scriptstyle A-l+1\\
			\scriptstyle B+2-l&\cdots&\scriptstyle C-l&\none&\none
		\end{ytableau}.$$
		We first prove the following facts:
		\begin{enumerate}
			\item[(a)] If $l\le C-B+1$, we have $$\circ_{x\in\mA_l}\Jac_{\rho,x}X_C=[-B,-C+l]\rtimes\circ_{x\in\mA_{C,l}}\Jac_{\rho,x}\pi(\psi',\varepsilon',(\rho,A,B+2;\eta_0)),$$ where the first factor of induced representation does not appear if $l=C-B+1$.
			\item[(b)] If $l>C-B+1$, we have $\circ_{x\in\mA_l}\Jac_{\rho,x}X_C=0$.
		\end{enumerate}
		 We first compute that 
		$$\Jac_{\rho,-B+1}\circ\Jac_{\rho,-B+1}\circ\cdots\circ\Jac_{\rho,B}X_C=[-B,-C]_\rho\rtimes Y_C,$$ where $Y_C=\Jac_{\rho,C}\circ\Jac_{\rho,C-1}\circ\cdots\circ\Jac_{\rho,B+2}\pi(\psi',\varepsilon',(\rho,A,B+2;\eta_0))$. Let $\mB_l$ be obtained from $\mA_l$ by removing the first column. Then the above computation implies that $\circ_{x\in\mA_l}\Jac_{\rho,x}X_C=\circ_{x\in\mB_l}\Jac_{\rho,x}([-B,-C]\rtimes Y_C)$.
		
		\par For $l'\le C-B+1$, we set
		$$\mC_{C,l'}:=\begin{ytableau}
			\scriptstyle B+2&\dots&C&\none&\none&\none\\
			\scriptstyle B+1&\dots&\scriptstyle C-1&\scriptstyle C+1&\dots&A\\
			\vdots&&\vdots&\vdots&&\vdots\\
			\scriptstyle B+2-l'&\dots&\scriptstyle C-l'&\scriptstyle C+2-l'&\dots&\scriptstyle A+1-l'\\
		\end{ytableau}.$$ Further, for $l''\le l'$, we set:
		$$\mC_{C,l',l''}:=\begin{ytableau}
			\scriptstyle B+2&\dots&C&\none&\none&\none&\none\\
			\scriptstyle B+1&\dots&\scriptstyle C-1&\none&\scriptstyle C+1&\dots&A\\
			\vdots&&\vdots&\none&\vdots&&\vdots\\
			\scriptstyle B-l''+2&\dots&\scriptstyle C-l''+1&\none&\scriptstyle C-l''+3&\dots&\scriptstyle A-l''+2\\
			\scriptstyle B-l''+1&\dots&\scriptstyle C-l''&\scriptstyle C-l''+1&\scriptstyle C-l''+2&\dots&\scriptstyle A-l''+1\\
			\vdots&&\vdots&\vdots&\vdots&&\vdots\\
			\scriptstyle B-l'+2&\dots&\scriptstyle C-l'&\scriptstyle C-l'+1&\scriptstyle C-l'+2&\dots&\scriptstyle A-l'+1\\
		\end{ytableau}.$$
		Here, the gap in the tableau means we skip the missing terms when computing the partial Jacquet module. Since $\Jac_{\rho,x}\Jac_{\rho,y}=\Jac_{\rho,y}\Jac_{\rho,x}$ when $|x-y|\neq 1$, we have 
		$$\begin{aligned}
			\circ_{x\in\mB_{l'}}\Jac_{\rho,x}&([-B,-C]\rtimes Y_C)\\
			&=\sum_{l''\le l'}[-B,-C+l'']_\rho\rtimes\circ_{x\in\mC_{C,l',l''}}\Jac_{\rho,x}\pi(\psi',\varepsilon',(\rho,A,B+2;\eta_0)).
		\end{aligned}$$
		Consider $$\mC'':=\begin{ytableau}
			\scriptstyle B+2&\dots&C&\none\\
			\scriptstyle B+1&\dots&C-1&\none\\
			\vdots&&\vdots&\none\\
			\scriptstyle B-l''+1&\dots&\scriptstyle C-l''&\scriptstyle C-l''+1\\
		\end{ytableau}.$$ For $\sigma\in\Pi(\GL(d))$ with $\Supp(\pi)=\mC''$. By Zelevinsky's classification, if $\circ_{x\in\mC''}\Jac_{\rho,x}(\sigma)\neq0$, then there exists an $x\neq B+2$ such that $\Jac_{\rho,x}(\sigma)\neq0$. Thus we have $$\circ_{x\in\mC''}\Jac_{\rho,x}\pi(\psi',\varepsilon',(\rho,A,B+2;\eta_0))=0.$$ Note that $\circ_{x\in\mC_{C,l',l''}}$ factors through $\circ_{x\in\mC''}$ when $l''>0$. We conclude that
		$$\begin{aligned}
			\circ_{x\in\mB_{l'}}\Jac_{\rho,x}&([-B,-C]\rtimes Y_C)\\
			&=[-B,-C+l']_\rho\rtimes\circ_{x\in\mC_{C,l'}}\Jac_{\rho,x}\pi(\psi',\varepsilon',(\rho,A,B+2;\eta_0)).
		\end{aligned}$$ Note that $\circ_{x\in\mC_{C,l'}}\Jac_{\rho,x}=\circ_{x\in\mA_{C,l'}}\Jac_{\rho,x}$ if we compute the Jacquet module columnwise. We obtain the $(a)$ above. Suppose that $l>C-B+1$. Note that $\circ_{x\in\mB_l}\Jac_{\rho,x}$ factors through $\circ_{x\in\mB_{C-B+2}}\Jac_{\rho,x}$ in this case. It is sufficient to assume that $l=C-B+2$. Consider $$\mA':=\begin{ytableau}
			\scriptstyle B+2&\dots&\scriptstyle C&\none&\none&\none&\none\\
			\scriptstyle B+1&\dots&\scriptstyle C-1&\none&\scriptstyle C+1&\dots&\scriptstyle A\\
			\vdots&&\vdots&\none&\vdots&&\vdots\\
			\scriptstyle 2B-C+1&\dots&\scriptstyle B-1&\none&\scriptstyle B+1&\dots&\scriptstyle A+B-C\\
			\scriptstyle 2B-C&\dots&\scriptstyle B-2&\scriptstyle B-1&B&\dots&\underset{-C-1}{\scriptstyle A+B}\\
		\end{ytableau}.$$ Then the computation about $\circ_{x\in\mC_{l',l''}}\Jac_{\rho,x}$ implies that $$\circ_{x\in\mA_l}\Jac_{\rho,x}X_C=\circ_{x\in\mA'}\Jac_{\rho,x}\pi(\psi',\varepsilon',(\rho,A,B+2;\eta_0))=0.$$ This completes the proof of $(b)$.
		
		\par Now, set $l=A-D$. Then $l>A-B+1\ge C-B+1$. Note that $\circ_{x\in\mE_D}\Jac_{\rho,x}$ factors through $\circ_{x\in\mA_l}\Jac_{\rho,x}$. Thus we have $\circ_{x\in\mE_D}\Jac_{\rho,x}\pi=0$ by the $(b)$ above.
	\end{proof}
	
	\begin{corollary}\label{coro: D<B-1 case}
		If $D<B-1$, then either $\pi$ satisfies $\Jac_{\rho,-A}\circ\Jac_{\rho,-A+1}\circ\cdots\circ\Jac_{\rho,B}\pi\neq0$ or is one of the complementary terms. 
	\end{corollary}
	\begin{proof}
		By Lemma \ref{lemma: first step - lemma 1}, Lemma \ref{lemma: non-vanishing of certain Jacquet module} and Lemma \ref{lemma: D<B-1 case} above, if we have $\Jac_{\rho,-A}\circ\Jac_{\rho,-A+1}\circ\cdots\circ\Jac_{\rho,B}\pi=0$, then $\pi$ is a subquotient of the representation $$X_\eta:=\pi(\psi',\varepsilon',(\rho,A,B+1;\eta),(\rho,B,B;\eta\eta_0))$$ for suitable $\eta\in\left\{\pm1\right\}$.
		
		\par Denote by $(\psi'',\varepsilon'')$ the pair obtained from $(\psi',\varepsilon')$ by removing $(\rho,B-1,B-1;\varepsilon(\rho,B-1,B-1))$. Using the same argument as in the proof of Proposition \ref{prop: cuspidal support of discrete series}, we have:
		$$\begin{aligned}
			\Jac_{\rho,-B+1}\circ\Jac_{\rho,-B+2}\circ\cdots&\circ\Jac_{\rho,B}X_\eta\\
			&=\begin{cases}
				\pi(\psi'',\varepsilon'',(\rho,A,B+1;\eta))&\text{if }\eta\eta_0=\varepsilon(\rho,B-1,B-1)\\
				0&\text{if }\eta\eta_0\neq\varepsilon(\rho,B-1,B-1).\\
			\end{cases}
		\end{aligned}$$ Thus we have $\eta\eta_0=\varepsilon(\rho,B-1,B-1)$, and there exists an irreducible subquotient $\pi'$ of $\pi(\psi'',\varepsilon'',(\rho,A,B+1;\eta))$ such that $\pi\xhookrightarrow{}[B,-B+1]_\rho\rtimes\pi'$. We associate to $\pi'$ data $D'$, $\lambda'$, $\mE'$ as in \S \ref{subsubsection: complementary terms}. Then, Lemma \ref{lemma: second construction - converse} implies that $D=D'$, $\lambda=\lambda'$ and $\mE'=\mE\cup[B,-B+1]$. In particular, we have $-A\notin\mE'$. Thus, by the induction hypothesis, $\pi'$ is one of the complementary terms of $\pi(\psi'',\varepsilon'',(\rho,A,B+1;\eta))$. That is, there exists an $\kappa\in\left\{\pm1\right\}$ such that $\pi'=\pi(\psi'',\varepsilon'',\cup_{C\in[B+1,A]}(\rho,A,B+1;(-1)^{C-B-1}\kappa))$. By Proposition \ref{prop: cuspidal support of discrete series}, we have
		$$\begin{aligned}
			&[B,-B+1]_\rho\rtimes\pi'\\
			&=\bigoplus_{\kappa'=\pm1}\pi(\psi'',\varepsilon'',(\rho,B-1,B-1,\kappa'),(\rho,B,B,\kappa'),\underset{C\in[B+1,A]}{\cup}(\rho,A,B+1;(-1)^{C-B-1}\kappa)).
		\end{aligned}$$
		Since $\Jac_{\rho,B-1}\pi=0$ and $\Jac_{\rho,B+1}\pi=0$, there exists $\kappa'$ such that $\kappa'\varepsilon(\rho,B-2,B-2)=-1$ and $\kappa'=-\kappa$. Thus, we have $\pi=\pi(\psi',\varepsilon',\cup_{B\le C\le A}(\rho,C,C;(-1)^{C-B}\kappa'))$. This completes the proof.
	\end{proof}
	
	\subsubsection{The $D\ge B-1$ case}
	
	\par We keep the notation of \S \ref{subsubsection: complementary terms}. Let $\mT=\mE_D$, and we denote by $\sigma_\mT$ the irreducible representation associated to $\mT$, i.e., $\sigma_\mT=\soc([B,A]_\rho\times\cdots\times[B-2t+1,D+1]_\rho\times[B-2t,B-t-1]_\rho\times\cdots\times[-B+1,-(B-t)]_\rho)$.
	
	\begin{lemma}\label{lemma: parabolic induction is SI when D ge B-1}
		When $D\ge B-1$, the parabolic induction $\sigma_\mT\rtimes\pi(\psi_{D,\lambda},\varepsilon_{D,\lambda})$ is SI.
	\end{lemma}
	\begin{proof}
		For $1\le i\le 2B$, denote by $[x_i,y_i]$ the i-th row of $\mT$. Let $\sigma_{\ge i}=\soc(\times_{j\ge i}[x_j,y_j]_\rho)$. Then, it is sufficient to prove that $$\Jac_{\rho,y_i}\circ\Jac_{\rho,y_i-1}\circ\cdots\circ\Jac_{\rho,x_i}\sigma_{\ge i}\rtimes\pi(\psi_{D,\lambda},\varepsilon_{D,\lambda})=\sigma_{\ge i+1}\rtimes\pi(\psi_{D,\lambda},\varepsilon_{D,\lambda}).$$
		
		\par If the above formula does not hold, there must exist some $z\in[x_i,y_i]$, such that $\Jac_{\rho,y_i}\circ\Jac_{\rho,y_i-1}\circ\cdots\circ\Jac_{\rho,z}(\sigma_{\ge i+1}\rtimes\pi(\psi_{D,\lambda},\varepsilon_{D,\lambda}))\neq0$. When $i\le 2t$, then $y_i=D+1-2t-i\ge B$. In particular, $y_i\notin \mT_{\ge i+1}\cup-\mT_{\ge i+1}$. This leads to a contradiction. When $i>2t$, $\sigma_{\ge i}$ is a Speh representation. In particular, we have $\Jac_{\rho,x}^\op(\sigma_{\ge i})\neq 0$ if and only if $x=-(B-t)$. Now, the lemma follows from the fact that $B-t$ is the unique maximal element in $\mT_{\ge i}\cup-\mT_{\ge i}$.
	\end{proof}
	
	\par Let $D$, $\lambda$ and $\mE$ satisfying the $(1)$, $(2)$, $(4)$ of Lemma \ref{lemma: second construction} with $b_{\rho,\psi_{D',\lambda'},\varepsilon_{D',\lambda'}}=2D'+1$. When $D\ge B-1$, the above lemma implies that $\sigma_\mT\rtimes\pi(\psi_{D,\lambda},\varepsilon_{D,\lambda})$ is SI. In this case, we write $\pi_{D,\lambda}=\pi(\psi_{D,\lambda},\varepsilon_{D,\lambda})$ and $\pi_{\mT,D,\lambda}=\soc(\sigma_\mT\rtimes\pi_{D,\lambda})$.
	
	\par For $B+1\le C\le A$, we set $X_C=[B,-C]_\rho\rtimes Y_C$ with $$\begin{aligned}
		Y_C&=\Jac_{\rho,C}\circ\cdots\circ\Jac_{\rho,B+2}\pi(\psi',\varepsilon',(\rho,A,B+2;\eta_0))\\
		&=\soc([C+1,A]_\rho\rtimes\pi(\psi',\varepsilon',(\rho,A-1,B+1;\eta_0))).
	\end{aligned}$$ For $\eta\in\left\{\pm1\right\}$, we set $X_\eta=\pi(\psi',\varepsilon',(\rho,A,B+1;\eta),(\rho,B,B;\eta\eta_0))$.
	
	\begin{lemma}\label{lemma: D ge B-1 case}
		Suppose that $D\ge B-1$, $B+1\le C\le A$, $B>0$,$\eta\in\left\{\pm1\right\}$, $t=\frac{A-D}{2}\le B$. Then the following is true:
		\begin{enumerate}
			\item[(1)] $\pi_{D,\lambda}$ is not a subquotient of $\circ_{x\in\mT_D}\Jac_{\rho,x}X_C$ unless $C-B+1=2t$.
			\item[(2)] Suppose that $C-B+1=2t$. Denote by $\mT_D'$ the following tableau:$$\begin{ytableau}
				\scriptstyle B+1&\cdots&\cdots&\cdots&\cdots&\cdots&\scriptstyle A-1\\
				\vdots&&\vdots&&\vdots&&\vdots\\
				\scriptstyle B-2t+3&\cdots&\scriptstyle B-t+1&\cdots&B&\cdots&\scriptstyle D+1\\
				\scriptstyle B-2t+2&\cdots&\scriptstyle B-t&\cdots&\scriptstyle B-1&\none&\none\\
				\vdots&&\vdots&\none&\none&\none&\none\\
				\scriptstyle-B+2&\cdots&\scriptstyle-B+t&\none&\none&\none&\none
			\end{ytableau}.$$ Then $\circ_{x\in\mT_D}\Jac_{\rho,x}X_C=\circ_{x\in\mT_D'}\Jac_{\rho,x}\pi(\psi',\varepsilon',(\rho,A-1,B+1;\eta_0))$.
			\item[(3)] Suppose that $C-B+1=2t$. Let $\mT_{A-2,B,D}$ be the analogue of $\mT_D$ by replacing $(A,B,t)$ with $(A-2,B,t-1)$. Then, the multiplicity of $\pi_{D,\lambda}$ in $\circ_{x\in\mT_D}\Jac_{\rho,x}X_C$ equals the multiplicity of $\pi_{D,\lambda}$ in $\circ_{x\in\mT_{A-2,B,D}}\Jac_{\rho,x}\pi(\psi',\varepsilon',(\rho,A-2,B;\eta_0))$.
			\item[(4)] The multiplicity of $\pi_{D,\lambda}$ in $\circ_{x\in\mT_D}\Jac_{\rho,x}X_\eta$ is $0$ if $\eta\eta_0\neq\varepsilon(\rho,B-1,B-1)$ and is equal to the multiplicity of $\pi_{D,\lambda}$ in $\circ_{x\in\mT_{A-2,B-1,D}}\Jac_{\rho,x}\pi(\psi'',\varepsilon'',(\rho,A-2,B-1;\eta))$, where $(\psi'',\varepsilon'')$ is obtained from $(\psi',\varepsilon')$ by removing $(\rho,B-1,B-1)$.
			\item[(5)] The multiplicity of $\pi_{D,\lambda}$ in $\circ_{x\in\mT_D}\Jac_{\rho,x}\pi(\psi,\varepsilon)$ is $0$ unless $t=B$ and $\lambda=-\varepsilon(\rho,C_{\min},C_{\min})$, in which case $\pi_{\mT_D,D,\lambda}$ is isomorphic to a complementary term.
		\end{enumerate}
	\end{lemma}
	\begin{proof}
		Set $l=2t$. By the computations in the proof of Lemma \ref{lemma: D<B-1 case}, we have:
		\begin{enumerate}
			\item[(a)] If $l\le C-B+1$, we have $$\circ_{x\in\mA_l}\Jac_{\rho,x}X_C=[-B,-C+l]\rtimes\circ_{x\in\mA_{C,l}}\Jac_{\rho,x}\pi(\psi',\varepsilon',(\rho,A,B+2;\eta_0)),$$ where the first factor of induced representation does not appear if $l=C-B+1$.
			\item[(b)] If $l>C-B+1$, we have $\circ_{x\in\mA_l}\Jac_{\rho,x}X_C=0$.
		\end{enumerate}
		Note that $\circ_{x\in\mT_D}\Jac_{\rho,x}$ factors through $\circ_{x\in\mA_l}\Jac_{\rho,x}$. Thus, when $l>C-B+1$, we have $\circ_{x\in\mT_D}\Jac_{\rho,x}X_C=0$. If $l\le C-B+1$, set
		$$\mT_{D,l}':=\begin{ytableau}
			B&\cdots&\cdots&\cdots&\cdots&\cdots&\scriptstyle A-2\\
			\vdots&&\vdots&&\vdots&&\vdots\\
			\scriptstyle B-2t+3&\cdots&\scriptstyle B-t+1&\cdots&\scriptstyle C-l+1&\cdots&\scriptstyle D+1\\
			\scriptstyle B-2t+2&\cdots&\scriptstyle B-t&\cdots&\scriptstyle C-l&\none&\none\\
			\vdots&&\vdots&\none&\none&\none&\none\\
			\scriptstyle -B+2&\cdots&\scriptstyle -B+t&\none&\none&\none&\none
		\end{ytableau}.$$
		Then we have $$\circ_{x\in\mT_D}\Jac_{\rho,x}X_C=[-B,-C+l]\rtimes\circ_{x\in\mT_{D,l}'}\Jac_{\rho,x}\pi(\psi',\varepsilon',(\rho,A-2,B;\eta_0)).$$
		Suppose that $\pi_{D,\lambda}$ is a subquotient of $\circ_{x\in\mT_D}\Jac_{\rho,x}X_C$. Then, there exists a subquotient $\pi'$ of $\pi(\psi',\varepsilon',(\rho,A-2,B;\eta_0))$ such that $\pi_{D,\lambda}$ is a subquotient of $[-B,-C+l]\rtimes\circ_{x\in\mT_{D,l}'}\Jac_{\rho,x}(\pi')$. Choose $\mE'$, $D'$ and $\lambda'$ for $\pi'$. Then, since $\circ_{x\in\mT_{D,l}'}\Jac_{\rho,x}(\pi')\neq 0$, we have $\mE'\cup-\mE'=\mT_{D,l'}\cup-\mT_{D,l'}$ and $\pi_{D,\lambda}$ is a subquotient of $[-B,-C+l]\rtimes\pi_{D',\lambda'}$. Note that $\mT_{D,l'}\cup-\mT_{D,l'}=\mT_{A-2,B,D}\cup-\mT_{A-2,B,D}\cup([B,C-l]\cup[-B,-C+l])$. Thus we conclude that $l=C-B+1$, $D=D'$ and $\lambda=\lambda'$. This completes the proof of $(1)$. And we can deduce $(2)$ easily from the $(a)$ above.
		
		\par For $(3)$, we denote by $\mT''$ the tableau obtained from $\mT_{D,l}'$ by removing $B-t+1,B-t+2,\dots,B-1$ from the row that starts with $B-2t+2$. We observe that $\circ_{x\in\mT_{D,l}'}\Jac_{\rho,x}=\Jac_{\rho,C-l}\circ\cdots\circ\Jac_{\rho,B-t+1}\circ(\circ_{x\in\mT''}\Jac_{\rho,x})$. By the key proposition (Proposition \ref{section: the key proposition}), we know that $\rho|\cdot|^x\rtimes\pi_{D,\lambda}$ is irreducible for all $x\le B-1$. In particular, we have $\rho|\cdot|^{B-t+1}\times\cdots\times\rho|\cdot|^{B-1}\rtimes\pi_{D,\lambda}\cong\rho|\cdot|^{-B+1}\times\cdots\times\rho|\cdot|^{-B+t-1}\rtimes\pi_{D,\lambda}$. Thus, the multiplicity of $\pi_{D,\lambda}$ in $\Jac_{\rho,B-1}\circ\cdots\circ\Jac_{\rho,B-t+1}\circ(\circ_{x\in\mT''}\Jac_{\rho,x})$ equals the multiplicity of $\pi_{D,\lambda}$ in $\Jac_{\rho,-B+t-1}\circ\cdots\circ\Jac_{\rho,-B+1}\circ(\circ_{x\in\mT''}\Jac_{\rho,x})=\circ_{x\in\mT_{A-2,B,D}}\Jac_{\rho,x}\pi(\psi',\varepsilon',(\rho,A-2,B;\eta_0))$. This completes the proof of $(3)$.
		
		\par For $(4)$, by the induction hypothesis and Proposition \ref{prop: cuspidal support of discrete series}, we have:
		$$\Jac_{\rho,-B+1}\circ\cdots\circ\Jac_{\rho,B}X_\eta=\begin{cases}
			\pi(\psi'',\varepsilon'',(\rho,A,B+1;\eta))&\text{if }\eta\eta_0=\varepsilon(\rho,B-1,B-1)\\
			0&\text{if }\eta\eta_0\neq\varepsilon(\rho,B-1,B-1).\\
		\end{cases}$$ Now, $(4)$ follows from the fact that $(\circ_{B\le x\le  A-1}\Jac_{\rho,x})\circ(\circ_{B+1\le x\le A}\Jac_{\rho,x})\pi(\psi'',\varepsilon'',(\rho,A,B+1;\eta))=\pi(\psi'',\varepsilon'',(\rho,A-2,B-1;\eta))$.
		
		\par For $(5)$, we first prove that, for an irreducible constituent $\pi$ of $\pi(\psi,\varepsilon)$, if $\circ_{x\in\mT_D}\Jac_{\rho,x}\pi=\pi_{D,\lambda}$, then we have $\Jac_{\rho,-A}\circ\Jac_{\rho,-A+1}\circ\cdots\circ\Jac_{\rho,B}(\pi)=0$. Note that $\circ_{x\in\mT_D}\Jac_{\rho,x}\pi=\circ_{x\in\mT_{A-1,B,D-1}}\Jac_{\rho,x}\circ(\Jac_{\rho,D+1}\circ\cdots\circ\Jac_{\rho,A})\pi$. Thus, there exists an irreducible representation $\sigma'$ of $\GL$ with $\Supp(\sigma')=\mT_{A-1,B,D-1}$ such that $\pi\xhookrightarrow{}\sigma'\rtimes\soc([A,D+1]_\rho\rtimes\pi_{D,\lambda})$. By Proposition \ref{prop: cuspidal support of discrete series}, $\soc([A,D+1]_\rho\rtimes\pi_{D,\lambda})$ is discrete. Thus, a direct computation implies that $\Jac_{\rho,-A}\circ\Jac_{\rho,-A+1}\circ\cdots\circ\Jac_{\rho,B}(\pi)=0$.
		
		\par Now, we start the computation of the multiplicity of $\pi_{D,\lambda}$ in $\circ_{x\in\mT_D}\Jac_{\rho,x}\pi(\psi,\varepsilon)$. By $(3)$ and $(4)$, we know it equals the multiplicity of $\pi_{D,\lambda}$ in 
		$$\begin{aligned}
			&(-1)^{D-B+1}\circ_{x\in\mT_{A-2,B,D}}\Jac_{\rho,x}\pi(\psi',\varepsilon',(\rho,A-2,B;\eta_0))\\&+(-1)^{[\frac{A-B+1}{2}]}\eta^{A-B+1}\eta_0^{A-B}\circ_{x\in\mT_{A-2,B-1,D}}\Jac_{\rho,x}\pi(\psi'',\varepsilon'',(\rho,A-2,B-1;\eta)),
		\end{aligned}$$ where $\eta\eta_0=\varepsilon(\rho,B-1,B-1)$.
		By the argument above and the induction hypothesis, the multiplicity of $\pi_{D,\lambda}$ in $\circ_{x\in\mT_{A-2,B,D}}\Jac_{\rho,x}\pi(\psi',\varepsilon',(\rho,A-2,B;\eta_0))$ equals the multiplicity of $\pi_{D,\lambda}$ in $\oplus_{\eta'=\pm1;\eta_0=\eta'^{A-B-1}(-1)^{\frac{(A-B-1)(A-B-2)}{2}}}\circ_{x\in\mT_{A-2,B,D}}\Jac_{\rho,x}\pi_{\eta'}$, where $\pi_{\eta'}=\pi(\psi',\varepsilon',\cup_{B\le C\le A-2}(\rho,C,C;(-1)^{C-B}\eta'))$.
		
		\par When $\eta'\neq\varepsilon(\rho,B-1,B-1)$, the $D',\lambda'$ associated to  $\pi_{\eta'}$ are $D'=A-2$ and $\lambda'=(-1)^{[B-1]}\varepsilon(\rho,B-1,B-1)$. It is clear that $\pi_{D,\lambda}$ occurs in $\circ_{x\in\mT_{A-2,B,D}}\pi_{\eta'}$ if and only if $D=D'=A-2$ and $\lambda=\lambda'$, in which case the multiplicity of $\pi_{D,\lambda}$ is $1$.
		
		\par When $\eta'=\varepsilon(\rho,B-1,B-1)$, we have $D'=A-2-2\inf(B,A-1-B)$. If $\pi_{D,\lambda}$ occurs in $\circ_{x\in\mT_{A-2,B,D}}\pi_{\eta'}$, we must have $D=D'$. Thus $t=1+\inf(B,A-1-B)$. Since we have $t\le B$, we must have $A-B\le B$. Hence $D=A-2-2(A-B-1)=B-(A-B)<B-1$. This gives a contradiction. Thus $\pi_{D,\lambda}$ will not occur in $\circ_{x\in\mT_{A-2,B,D}}\pi_{\eta'}$.
		
		\par Also, the multiplicity of $\pi_{D,\lambda}$ in $\circ_{x\in\mT_{A-2,B-1,D}}\Jac_{\rho,x}\pi(\psi'',\varepsilon'',(\rho,A-2,B-1;\eta))$ equals the multiplicity of $\pi_{D,\lambda}$ in $\oplus_{\eta''=\pm1;\eta=\eta''^{A-B}(-1)^{\frac{(A-B)(A-B-1)}{2}}}\circ_{x\in\mT_{A-2,B-1,D}}\Jac_{\rho,x}\pi_{\eta''}$, where $\pi_{\eta''}=\pi(\psi'',\varepsilon'',\cup_{B-1\le C\le A-2}(\rho,C,C;(-1)^{C-B+1}\eta'))$.
		
		\par We first consider the case when $B>1$. For $\eta''\neq\varepsilon(\rho,B-2,B-2)$, $\pi_{D,\lambda}$ occurs in $\circ_{x\in\mT_{A-2,B,D}}\Jac_{\rho,x}\pi_{\eta''}$ with multiplicity $1$ if $D=D''=A-2$ and $\lambda=\lambda''$, while $\pi_{D,\lambda}$ does not occur in $\circ_{x\in\mT_{A-2,B,D}}\Jac_{\rho,x}\pi_{\eta''}$ otherwise. Note that $\pi_{\eta''}$ appears if and only if $$(-\varepsilon(\rho,B-2,B-2))^{A-B}(-1)^{\frac{(A-B)(A-B-1)}{2}}=\eta=\eta_0\varepsilon(\rho,B-1,B-1),$$
		if and only if 
		$$\begin{aligned}
			\varepsilon(\rho,B-1,B-1)^{A-B-1}&(-1)^{\frac{(A-B)(A-B-1)}{2}}\\
			&=(-\varepsilon(\rho,B-1,B-1))^{A-B-1}(-1)^{\frac{(A-B-2)(A-B-1)}{2}}=\eta_0.
		\end{aligned}$$
		When the above equation holds, we have $\eta^{A-B+1}\eta_0^{A-B}=(-1)^{\frac{(A-B-2)(A-B-1)}{2}}=(-1)^{[\frac{A-B}{2}]}$. It is easy to verify that $(-1)^{[\frac{A-B}{2}]}(-1)^{[\frac{A-B+1}{2}]}+(-1)^{A-B-1}=0$. In conclusion, the multiplicity of $\pi_{D,\lambda}$ in $(-1)^{D-B+1}\circ_{x\in\mT_{A-2,B,D}}\Jac_{\rho,x}\pi_{\eta'}+(-1)^{[\frac{A-B+1}{2}]}\eta^{A-B+1}\eta_0^{A-B}\circ_{x\in\mT_{A-2,B-1,D}}\Jac_{\rho,x}\pi_{\eta''}$ is $0$ for $\eta'=-\varepsilon(\rho,B-1,B-1)$ and $\eta''=-\varepsilon(\rho,B-2,B-2)$.
		
		\par For $\eta''=\varepsilon(\rho,B-2,B-2)$, we have $D''=A-2-2\inf(B-1,A-B)$. If $\pi_{D,\lambda}$ occurs in $\circ_{x\in\mT_{A-2,B-1,D}}\Jac_{\rho,x}\pi_{\eta''}$, then we have $D=D''=A-2-2\inf(B-1,A-B)$. Note that $A-B+1<B$ implies that $D<B-1$. Thus we have $B\le A-B+1$ with $t=B$, in which case $\pi_{D,\lambda}$ occurs with multiplicity $1$ in $(-1)^{[\frac{A-B+1}{2}]}\eta^{A-B+1}\eta_0^{A-B}\circ_{x\in\mT_{A-2,B-1,D}}\Jac_{\rho,x}\pi_{\eta''}$. Thus, we have concluded the proof of $(5)$ in the case when $B>1$. When $B=1$, for any $\eta''\in\left\{\pm1\right\}$, we have $D''=A-2$. By a similar argument as above, we can also deduce $(5)$.
	\end{proof}
	
	\subsubsection{End of the proof in this particular case}
	
	\begin{proposition}\label{prop: end of the proof in particular case}
		Theorem \ref{theorem: construction of A-packet non-negative DDR case} is true under the hypotheses of this section. That is, Theorem \ref{theorem: construction of A-packet non-negative DDR case} is true if $(\rho,A,B)$ is the only triple of $\Jord(\psi)$ satisfying $A>B$ and $b_{\rho,\psi',\varepsilon'}=2B-1$.
	\end{proposition}
	\begin{proof}
		By the remark after Lemma \ref{lemma: lemma on the main term} and Lemma \ref{lemma: complementary terms}, it is sufficient to prove that, for an irreducible constituent $\pi$ of $\pi(\psi,\varepsilon)$, if $\Jac_{\rho,-A}\circ\Jac_{\rho,-A+1}\circ\cdots\circ\Jac_{\rho,B}(\pi)=0$, then $\pi$ is isomorphic to one of the complementary terms. By Lemma \ref{lemma: second construction}, there exists a triple $(\mE,\lambda,D)$ satisfying the four conditions of Lemma \ref{lemma: second construction} with an additional condition that $b_{\rho,\psi_{D,\lambda},\varepsilon_{D,\lambda}}=2D+1$. If $D<B-1$, Corollary \ref{coro: D<B-1 case} gives the desired result. When $D\ge B-1$, by Lemma \ref{lemma: first step - lemma 1} and Lemma \ref{lemma: non-vanishing of certain Jacquet module}, we may further assume that $\mE=\mE_D$. Now, Lemma \ref{lemma: parabolic induction is SI when D ge B-1} implies that $\pi=\soc(\sigma_{\mE_D}\rtimes\pi_{D,\lambda})$. Thus $\pi_{D,\lambda}$ is a subquotient $\circ_{x\in\mE_D}\Jac_{\rho,x}\pi(\psi,\varepsilon)$. Now, the $(5)$ of Lemma \ref{lemma: D ge B-1 case} implies that $\pi=\pi_{\mE_D,D,\lambda}$ is isomorphic to a complementary term. This completes the proof.
	\end{proof}
	
	\subsection{Extension}
	
	\par In this subsection, we extend Proposition \ref{prop: end of the proof in particular case} above to a more general setting. We still assume that $\psi'$ is discrete, i.e., $(\rho,A,B)$ is the only triple of $\Jord(\psi)$ satisfying $A>B$. In Proposition \ref{prop: end of the proof in particular case}, we assume that $b_{\rho,\psi',\varepsilon'}=2B-1$. In this subsection, we only assume a weaker condition that $a_{\rho,\psi',\varepsilon'}\ge2B+1$, which implies that $a_{\rho,\psi',\varepsilon'}>2A+1$ since $\psi$ is a non-negative DDR. 
	
	\begin{proposition}\label{prop: extension}
		Theorem \ref{theorem: construction of A-packet non-negative DDR case} is true under the hypotheses of this section.
	\end{proposition}
	\begin{proof}
		Write $b=b_{\rho,\psi',\varepsilon'}$. Consider 
		$$\mT:=\begin{ytableau}
		    B&\cdots&A\\
		    \vdots&&\vdots\\
		    \frac{b+3}{2}&\cdots&\underset{+\frac{b+3}{2}}{\scriptstyle A-B}\\
		\end{ytableau}.$$ Then, by Proposition \ref{prop: reduction to non-negative DDR}, we have $$\circ_{x\in\mT}\Jac_{\rho,x}\pi(\psi,\varepsilon)=\pi(\psi',\varepsilon',(\rho,\frac{b+1}{2}+A-B,\frac{b+1}{2};\eta_0)).$$
		By Proposition \ref{prop: end of the proof in particular case}, we know that Theorem \ref{theorem: construction of A-packet non-negative DDR case} is true for $\pi(\psi',\varepsilon',(\rho,\frac{b+1}{2},\frac{b+1}{2};\eta_0))$. On the other hand, using the same argument as in the proof of Lemma \ref{lemma: non-vanishing of certain Jacquet module}, we can prove that: for an irreducible constituent $\pi$ of $\pi(\psi,\varepsilon)$, if $\Jac_{\rho,-A}\circ\Jac_{\rho,-A+1}\circ\cdots\circ\Jac_{\rho,B}(\pi)=0$, then $\circ_{x\in\mT}\Jac_{\rho,x}\pi\neq0$.
		
		\par For $\eta\in\left\{\pm1\right\}$ satisfying $\eta_0=\eta^{A-B+1}(-1)^{\frac{(A-B+1)(A-B)}{2}}$, let $X_\eta$ be the corresponding complementary term of $\pi(\psi,\varepsilon)$. Then, a direct computation shows that $\circ_{x\in\mT}X_\eta$ is a complementary term of $\pi(\psi',\varepsilon',(\rho,\frac{b+1}{2},\frac{b+1}{2};\eta_0))$. Thus, the multiplicity of $X_\eta$ is at most one. Conversely, there exists an irreducible constituent $\pi$ of $\pi(\psi,\varepsilon)$ such that $\circ_{x\in\mT}\pi=X'_\eta$ is a complementary term of $\pi(\psi',\varepsilon',(\rho,\frac{b+1}{2},\frac{b+1}{2};\eta_0))$. By the theory of derivatives and socles, we have $\pi=\soc(\sigma_{\mT}\rtimes X_\eta')$. By Proposition \ref{prop: cuspidal support of discrete series}, we have $\soc(\sigma_{\mT}\rtimes X_\eta')=X_\eta$. Thus, the multiplicity of $X_\eta$ in $\pi(\psi,\varepsilon)$ is exactly one.
		
		\par For an irreducible constituent $\pi$ of $\pi(\psi,\varepsilon)$ satisfying $\Jac_{\rho,-A}\circ\Jac_{\rho,-A+1}\circ\cdots\circ\Jac_{\rho,B}(\pi)=0$, write $\pi'=\circ_{x\in\mT}\Jac_{\rho,x}\pi$. By the theory of derivatives, we know that $\pi'$ is irreducible. By the remark after Lemma \ref{lemma: lemma on the main term}, it is sufficient to prove that $\pi'$ is a complementary term of $\pi(\psi',\varepsilon',(\rho,\frac{b+1}{2},\frac{b+1}{2};\eta_0))$.
		
		\par Set $s:=B-\frac{b+1}{2}$. If $\pi''$ is not a complementary term, then $\pi''\le\soc([B-s,-A+s]\rtimes\pi(\psi',\varepsilon',(\rho,A-s-1,B-s+1;\eta_0)))$. Thus, we have $\circ_{-A+s\le x\le B-s}\Jac_{\rho,x}\pi''\neq0$. Denote by $\mT'$ the tableau obtained from $\mT$ by removing the first and last column. Since $\Jac_{\rho,x}\Jac_{\rho,y}=\Jac_{\rho,y}\Jac_{\rho,x}$ whenever $|x-y|\neq1$, we have $(\circ_{-A+s\le x\le B-s}\Jac_{\rho,x})\circ(\circ_{x\in\mT}\Jac_{\rho,x})\pi=(\circ_{A-s+1\le x\le A}\Jac_{\rho,x})\circ(\circ_{x\in\mT'}\Jac_{\rho,x})\circ(\circ_{-A+s\le x\le B}\Jac_{\rho,x})\pi=\pi'$.
		
		\par By the key proposition (Proposition \ref{proposition: key proposition}), we know that $\rho|\cdot|^x\rtimes\pi''$ are irreducible for all $A-s+1\le x\le A$. For an irreducible representation $\sigma$, we have $\Jac_{\rho,x}\sigma=\pi''$ if and only if $\Jac_{\rho,-x}\sigma=\pi''$. Thus, we have $(\circ_{A-s+1\le x\le A}\Jac_{\rho,x})\circ(\circ_{x\in\mT'}\Jac_{\rho,x})\circ(\circ_{-A+s\le x\le B}\Jac_{\rho,x})\pi=(\circ_{-A\le x\le -A+s-1}\Jac_{\rho,x})\circ(\circ_{x\in\mT'}\Jac_{\rho,x})\circ(\circ_{-A+s\le x\le B}\Jac_{\rho,x})\pi=(\circ_{x\in\mT'}\Jac_{\rho,x})\circ(\circ_{-A\le x\le B}\Jac_{\rho,x})\pi=\pi'$. In particular, we have $\circ_{-A\le x\le B}\Jac_{\rho,x}\pi\neq 0$, which gives a contradiction.
	\end{proof}	
	
	\subsection{Reduction}
	
	\par In this subsection, we will finish the proof of Theorem \ref{theorem: construction of A-packet non-negative DDR case} in the general case.
	
	\subsubsection{First reduction}
	
	\par In this subsection, we assume that $\Jord(\psi)$ contains at least two triples $(\rho,A,B)$, $(\rho',A',B')$ such that $A>B$ and $A'>B'$. We set $\eta_0=\varepsilon(\rho,A,B)$, $\eta_0'=\varepsilon(\rho,A',B')$ and denote by $(\psi'',\varepsilon'')$ the pair obtained from $(\psi,\varepsilon)$ by removing $(\rho,A,B)$ and $(\rho',A',B')$. Without loss of generality, we assume that $A>A'$, which implies that $A\ge B>A'\ge B'\ge0$ since $\psi$ is a non-negative DDR. In particular, we have $[B',-A']\subset[B,-B]$.
	
	\par Write $\lambda_{\eta'}:=(-1)^{[\frac{A'-B'+1}{2}]}\eta'^{A'-B'+1}\eta_0'^{A'-B'}$, $\delta_{C'}:=[B',-C']_{\rho'}$, $\delta_{C'}':=[C'+1,A']_{\rho'}$, and $\delta:=[B,-A]_\rho$. By Theorem \ref{theorem: Mp version of Theorem 7.5 in Xu_moeglin}, we have
	$$\begin{aligned}
		\pi(\psi,\varepsilon)=&\bigoplus_{C'\in\left(B',A'\right]}(-1)^{A-C}\delta_{C'}\rtimes\soc(\delta_{C'}'\rtimes\pi(\psi'',\varepsilon'',(\rho,A,B;\eta_0),(\rho',A'-1,B'+1;\eta_0)))\\
		&\bigoplus_{\eta'=\pm1}\lambda_{\eta'}\pi(\psi'',\varepsilon'',(\rho,A,B;\eta_0),(\rho,A',B'+1;\eta'),(\rho,B',B';\eta'\eta_0')).
	\end{aligned}$$
	Furthermore, applying the induction hypothesis to $\pi(\psi,\varepsilon)$, we can write $\pi(\psi,\varepsilon)$ as a sum of the following terms:
	\begin{align}
		&\begin{aligned}
			\oplus_{C'}&(-1)^{A-C}\delta_{C'}\rtimes\soc(\delta_{C'}'\rtimes\soc(\delta\rtimes\pi_1)))\\
			&\text{with }\pi_1=\pi(\psi'',\varepsilon'',(\rho,A-1,B+1;\eta_0),(\rho',A'-1,B'+1;\eta_0').
		\end{aligned}\label{equation: 9.1}\\
		&\begin{aligned}
			\oplus_{C',\eta}&(-1)^{A-C}\delta_{C'}\rtimes\soc(\delta_{C'}'\rtimes\pi_2))\\
			&\text{with }\pi_2=\pi(\psi'',\varepsilon'',\cup_{C\in[B,A]}(\rho,C,C;(-1)^{C-B}\eta),(\rho',A'-1,B'+1;\eta_0').
		\end{aligned}\label{equation: 9.2}\\
		&\begin{aligned}
			\oplus_{\eta'}&\lambda_{\eta'}\soc(\delta\rtimes\pi_3))\\
			&\text{with }\pi_3=\pi(\psi'',\varepsilon'',(\rho,A-1,B+1;\eta_0),(\rho,A',B'+1;\eta'),(\rho,B',B';\eta'\eta_0').
		\end{aligned}\label{equation: 9.3}\\
		&\begin{aligned}
			\oplus_{\eta',\eta}&\lambda_{\eta'}\soc(\delta\rtimes))\\
			&\text{ with }\pi_4=\pi(\psi'',\varepsilon'',\cup_{C\in[B,A]}(\rho,C,C;(-1)^{C-B}\eta),(\rho,A',B'+1;\eta'),(\rho,B',B';\eta'\eta_0').\\
		\end{aligned}\label{equation: 9.4}
	\end{align}
	By Theorem \ref{theorem: Mp version of Theorem 7.5 in Xu_moeglin}, the sum of (\ref{equation: 9.2}) and (\ref{equation: 9.4}) is $$\oplus_\eta\pi(\psi'',\varepsilon'',\cup_{C\in[B,A]}(\rho,C,C;(-1)^{C-B}\eta),(\rho,A',B';\eta_0')),$$ which are exactly the complementary terms of $\pi(\psi,\varepsilon)$.
	On the other hand, Lemma \ref{lemma: inversion of scole} implies that 
	$$\begin{aligned}
		&\soc(\delta_{C'}'\rtimes\soc(\delta\rtimes\pi(\psi'',\varepsilon'',(\rho,A-1,B+1;\eta_0),(\rho',A'-1,B'+1;\eta_0'))))\\
		&=\soc(\delta\rtimes\soc(\delta_{C'}'\rtimes\pi(\psi'',\varepsilon'',(\rho,A-1,B+1;\eta_0),(\rho',A'-1,B'+1;\eta_0'))))
	\end{aligned}$$
	Further, by Lemma \ref{lemma: subquotient of parabolic induction}, $D=\Jac_{\rho,-A}\circ\cdots\circ\Jac_{\rho,B}$ gives a bijection between the irreducible subquotients of $\delta_C\rtimes\soc(\delta\rtimes\soc(\delta_{C'}'\rtimes\pi(\psi'',\varepsilon'',(\rho,A-1,B+1;\eta_0),(\rho',A'-1,B'+1;\eta_0'))))$ and those of $\delta_C\rtimes\soc(\delta_{C'}'\rtimes\pi(\psi'',\varepsilon'',(\rho,A-1,B+1;\eta_0),(\rho',A'-1,B'+1;\eta_0')))$. In particular, for an irreducible constituent $\pi$ of $\pi(\psi,\varepsilon)$, that occurs in (\ref{equation: 9.1}) or (\ref{equation: 9.3}), we have $D(\pi)\neq0$. Thus, by the remark after Lemma \ref{lemma: lemma on the main term}, we know that Theorem \ref{theorem: construction of A-packet non-negative DDR case} is true for $\pi(\psi,\varepsilon)$.
	
	\subsubsection{Second reduction}
	
	\par By the previous subsection, it is sufficient to consider the case that $(\rho,A,B)$ is the only triple satisfying $A>B$, i.e., $\psi'$ is discrete. Write $a=a_{\rho,\psi',\varepsilon'}$, $b=b_{\rho,\psi',\varepsilon'}$. If $a\ge 2B+1$, Theorem \ref{theorem: construction of A-packet non-negative DDR case} follows from Proposition \ref{prop: extension}. Thus, we may assume that $a\le 2B-1$.
	
	\par In this subsection, we first consider the case when $a>b+2$ or $a=2$. We denote by $(\psi'',\varepsilon'')$ the pair obtained from $(\psi,\varepsilon)$ by removing $(\rho,A,B)$ and $(\rho,\frac{a-1}{2},\frac{a-1}{2})$. Then, for $C\in(B,A]$, write $\pi_x=\pi(\psi'',\varepsilon'',(\rho,A-1,B+1;\eta_0),(\rho,\frac{x-1}{2},\frac{x-1}{2};\eta_a)$, we have $$\begin{aligned}
		[B,-C]_\rho\rtimes\soc([C+1,A]_\rho\rtimes\pi_a))&=[B,-C]_\rho\rtimes\soc([C+1,A]_\rho\rtimes\soc(\rho|\cdot|^{\frac{a-1}{2}}\rtimes\pi_{a-2})))\\
		&=[B,-C]_\rho\rtimes\soc(\rho|\cdot|^{\frac{a-1}{2}}\rtimes\soc([C+1,A]_\rho\rtimes\pi_{a-2})))
	\end{aligned}$$
	By the proof of Lemma \ref{lemma: subquotient of parabolic induction}, we have
	$$\begin{aligned}
		[B,-C]_\rho\rtimes&\soc(\rho|\cdot|^{\frac{a-1}{2}}\rtimes\soc([C+1,A]_\rho\rtimes\pi_{a-2})))\\
		&=\soc(\rho|\cdot|^{\frac{a-1}{2}}\times[B,-C]_\rho\rtimes\soc([C+1,A]_\rho\rtimes\pi_{a-2})))
	\end{aligned}$$
	Applying a similar argument to $\pi(\psi',\varepsilon',\cup_{B\le C\le A}(\rho,C,C;(-1)^{C-B}\eta))$, we conclude that
	$$\pi(\psi,\varepsilon)=\soc(\rho|\cdot|^{\frac{a-1}{2}}\rtimes\pi(\psi'',\varepsilon'',(\rho,A,B;\eta_0),(\rho,\frac{a-3}{2},\frac{a-3}{2};\eta_a))).$$
	By induction on $a-b$, we may assume that Theorem \ref{theorem: construction of A-packet non-negative DDR case} is true for $$\pi(\psi'',\varepsilon'',(\rho,A,B;\eta_0),(\rho,\frac{a-3}{2},\frac{a-3}{2};\eta_a)).$$
	Then, by Proposition \ref{prop: cuspidal support of discrete series} and Lemma \ref{lemma: inversion of scole}, it is not hard to see that Theorem \ref{theorem: construction of A-packet non-negative DDR case} is also true for $\pi(\psi,\varepsilon)$.
	
	\subsubsection{Third reduction}
	
	\par Now, we consider the case when $a=b+2$ with $b>1$. We denote by $\varepsilon_-$ the character obtained by changing $\varepsilon$ on the two blocks $(\rho,\frac{a-1}{2},\frac{a-1}{2})$ and $(\rho,\frac{b-1}{2},\frac{b-1}{2})$ to its opposite and we denote by $(\widetilde{\psi},\widetilde{\varepsilon})$ the pair obtained from $(\psi,\varepsilon)$ by removing $(\rho,\frac{a-1}{2},\frac{a-1}{2})$ and $(\rho,\frac{b-1}{2},\frac{b-1}{2})$.
	
	\par Using a similar argument as in the previous subsection, we can deduce that $$\pi(\psi,\varepsilon)\oplus\pi(\psi,\varepsilon_-)=\soc([\frac{a-1}{2},-\frac{b-1}{2}]_\rho\rtimes\pi(\widetilde{\psi},\widetilde{\varepsilon})).$$
	By induction on the rank of $\widetilde{G}_{2n}$, we may assume that Theorem \ref{theorem: construction of A-packet non-negative DDR case} is true for $\pi(\widetilde{\psi},\widetilde{\varepsilon})$. Hence, setting $\delta=[B,-A]_\rho$, we have $$\begin{aligned}
		\pi(\widetilde{\psi},\widetilde{\varepsilon})=&\soc(\delta\rtimes\pi(\widetilde{\psi}',\widetilde{\varepsilon}',(\rho,A-1,B+1;\eta_0)))\\
		&\oplus_\eta\pi(\widetilde{\psi}',\widetilde{\varepsilon}',\cup_{C\in[B,A]}(\rho,C,C;(-1)^{C-B}\eta)).
	\end{aligned}$$
	By Proposition \ref{prop: cuspidal support of discrete series} and Lemma \ref{lemma: inversion of scole}, we have:
	$$\begin{aligned}
		&\soc([\frac{a-1}{2},-\frac{b-1}{2}]_\rho\rtimes\soc(\delta\rtimes\pi(\widetilde{\psi}',\widetilde{\varepsilon}',(\rho,A-1,B+1;\eta_0))))\\
		&=\soc(\delta\rtimes\pi(\psi',\varepsilon',(\rho,A-1,B+1;\eta_0)))\oplus\soc(\delta\rtimes\pi(\psi',\varepsilon_-',(\rho,A-1,B+1;\eta_0))).
	\end{aligned}$$
	and 
	$$\begin{aligned}
		&\soc([\frac{a-1}{2},-\frac{b-1}{2}]_\rho\rtimes\pi(\widetilde{\psi}',\widetilde{\varepsilon}',\cup_{C\in[B,A]}(\rho,C,C;(-1)^{C-B}\eta)))\\
		&=\pi(\psi',\varepsilon',\cup_{C\in[B,A]}(\rho,C,C;(-1)^{C-B}\eta))\oplus\pi(\psi',\varepsilon_-',\cup_{C\in[B,A]}(\rho,C,C;(-1)^{C-B}\eta)).
	\end{aligned}$$
	By induction on $b$, we may assume that Theorem \ref{theorem: construction of A-packet non-negative DDR case} is true for $\pi(\psi,\varepsilon_-)$, from which we can deduce that Theorem \ref{theorem: construction of A-packet non-negative DDR case} is true for $\pi(\psi,\varepsilon)$.
	
	\subsubsection{Final case}
	
	\par Now, there remains only one case that $a=3$ and $b=1$. In this subsection, we will handle this case and thus complete the proof of Theorem \ref{theorem: construction of A-packet non-negative DDR case}. In this case, we still have 
	$$\pi(\psi,\varepsilon)\oplus\pi(\psi,\varepsilon_-)=\soc([1,0]_\rho\rtimes\pi(\widetilde{\psi},\widetilde{\varepsilon})).$$ Thus, it is sufficient to prove that, for an irreducible constituent $\pi$ of $\pi(\psi,\varepsilon)$, if $\Jac_{\rho,-A}\circ\Jac_{\rho,-A+1}\circ\cdots\circ\Jac_{\rho,B}\pi=0$, then $\pi$ is isomorphic to one of the complementary terms.
	
	\begin{lemma}\label{lemma: final case}
		For an irreducible constituent $\pi$ of $\pi(\psi,\varepsilon)$, the following is true:
		\begin{enumerate}
			\item[(1)] There exist $D\in\ZZ_{\ge0}$, $\lambda\in\left\{\pm1\right\}$ and a totally ordered set $\mE$ of integers, such that $\mE\cup-\mE=[0,0]\cup[-1,1]\cup(\bigcup_{E\in[B,A]}[-E,E])-(\bigcup_{C\le D}[-C,-C])=\bigcup_{z\in[D+1,A]}[-z,z]-\bigcup_{z\in(1,B)}[-z,z]$, $A-D=2t_1+B$ and $\pi\xhookrightarrow{}\times_{x\in\mE}\rho|\cdot|^x\rtimes\pi_{D,\lambda}$. The $t_1$ here is taken from Lemma \ref{lemma: first construction}.
			\item[(2)] If $-A\in\mE$, then $\Jac_{\rho,-A}\circ\Jac_{\rho,-A+1}\circ\cdots\circ\Jac_{\rho,B}\pi\neq0$.
		\end{enumerate}
	\end{lemma}
	\begin{proof}
		Denote by $(\widetilde{\psi},\widetilde{\varepsilon})$ the pair obtained from $(\psi,\varepsilon)$ by removing $(\rho,0,0)$ and $(\rho,1,1)$. As in the previous subsection, we know that $\pi(\psi,\varepsilon)$ is a subrepresentation of $[1,0]_\rho\rtimes\pi(\widetilde{\psi},\widetilde{\varepsilon})$. Now, we can deduce $(1)$ by applying Lemma \ref{lemma: second construction} to $\pi(\widetilde{\psi},\widetilde{\varepsilon})$.
		
		\par For $(2)$, using a similar argument as in the proof of Lemma \ref{lemma: first step - lemma 1}, if $-A\in\mE$, then there exists a $x\in\mE$, such that $\Jac_{\rho,-A}\circ\Jac_{\rho,-A+1}\circ\cdots\circ\Jac_{\rho,x}\pi\neq0$. In particular, we have $\Jac_{\rho,x}\pi\neq0$. Thus, $x=B$ or $x=1$. Note that $\Jac_{\rho}\Jac_{\rho,1}\pi\le2\pi(\widetilde{\psi},\widetilde{\varepsilon})$ and $\Jac_{\rho,-1}\pi(\widetilde{\psi},\widetilde{\varepsilon})=0$. We conclude that $x=B$, which completes the proof.
	\end{proof}
	
	\par Write $\mE=\cup_{i=1}^r[x_i,y_i]$ with respect to a maximal order on $\mE$, and we assume that $\mE$ satisfies the property of having the minimal number of positive elements and $-A\notin\mE$.
	
	\par Suppose first that $x_1=1$. Then $B\notin\mE$. Thus $x_{i+1}=x_i-1$ holds for all $1\le i\le r$. As in the proof of Lemma \ref{lemma: maximal order}, we deduce from this that $x_{i+1}=x_i-1$ for all $i$ and $y_1>y_2>\cdots>y_r$. Further, similar to the proof of Lemma \ref{lemma: non-vanishing of certain Jacquet module}, we deduce that, for $1\le i\le r$, if $y_i>0$, then either $y_{i+1}=y_i-1$ or $y_i=D+1$.
	
	\par We denote by $z_1,\dots,z_A$ the elements at the end of each column. Then, as in the proof of Lemma \ref{lemma: non-vanishing of certain Jacquet module}, we have $$[1,A]\cup(\bigcup_{z_s\le0}\left\{-z_s\right\})\cup(1,B)=[D+1,A]\cup(\bigcup_{z_s>0}\left\{z_s-1\right\}),$$ where $\bigcup_{z_s>0}\left\{z_s-1\right\}$ is a segment of the form $[z,D]$. In particular, the union on the right-hand side is multiplicity free. Thus $B=2$ and $\bigcup_{z_s\le0}\left\{-z_s\right\}\subset\left\{0\right\}$, from which we deduced that the number of rows in $\mE$ is at most $2$. Recall that the number of rows of $\mE$ is at least $A-D=2t_1+B$. Thus we have $t_1=0$, which implies that $\pi$ is a complementary term.
	
	\par Now, suppose that $x_1=B$ with $B>2$. Then, the first row of $\mE$ is $B,\dots,A$. By Lemma \ref{lemma: another reduction step in non-negative DDR case}, $\Jac_{\rho,A}\circ\cdots\circ\Jac_{\rho,B}\pi$ is an irreducible constituent $\pi'$ of $\pi(\psi',\varepsilon',(\rho,A-1,B-1;\eta_0))$. If $\Jac_{\rho,-A+1}\circ\cdots\circ\Jac_{\rho,B-1}\pi'=\pi''\neq0$, then $\pi''\le\pi(\psi',\varepsilon',(\rho,A-2,B;\eta_0))$. By Lemma \ref{lemma: irreducibility of rho||x rtimes pi in non-negative DDR case}, we know that $\rho|\cdot|^A\rtimes\pi''$ is irreducible. Therefore, we have
	$$\begin{aligned}
		&\Jac_{\rho,-A+1}\circ\cdots\circ\Jac_{\rho,B-1}\pi'\\
		&=(\Jac_{\rho,-A+1}\circ\cdots\circ\Jac_{\rho,B-1})\circ(\Jac_{\rho,A}\circ\cdots\circ\Jac_{\rho,B})\pi'\\
		&=(\Jac_{\rho,A}\circ\cdots\circ\Jac_{\rho,B+1})\circ(\Jac_{\rho,-A+1}\circ\cdots\circ\Jac_{\rho,B})\pi'\\
		&=\Jac_{\rho,-A}\circ(\Jac_{\rho,A-1}\cdots\circ\Jac_{\rho,B+1})\circ(\Jac_{\rho,-A+1}\circ\cdots\circ\Jac_{\rho,B})\pi'\\
		&=(\Jac_{\rho,A-1}\cdots\circ\Jac_{\rho,B+1})\circ(\Jac_{\rho,-A}\circ\cdots\circ\Jac_{\rho,B})\pi'=0.
	\end{aligned}$$ In conclusion, we only need to consider the case when $B=2$.
	
	\par Suppose first that $\mE$ can still be written as a tableau of the form:
	$$\begin{ytableau}
		B&\dots&\dots&\dots&A\\
		\vdots&\vdots&\vdots&\vdots&\vdots\\
		x_{2t_1+2}&\dots&\dots&\dots&\scriptstyle D+1\\
		\vdots&\vdots&\vdots&\vdots&\none\\
		x_r&\cdots&y_r&\none&\none
	\end{ytableau},$$ where $x_{2t_1+2}=B-(2t_1+2)+1=-2t_1+1$. Note that we have $$[2,A]\cup(\bigcup_{z_s\le0}\left\{-z_s\right\})=[D+1,A]\cup(\bigcup_{z_s>0}\left\{z_s-1\right\}),$$ and the right-hand side has no multiplicity. Thus we have $z_s\ge-1$ for all $s$. In particular, we have $-2t_1+1\ge-1$, which means $t_1=0$ or $1$. If $t_1=0$, then $\pi$ is a complementary term. Thus we only need to eliminate the case $t_1=1$. Suppose that $t_1=1$, then we have $x_r=z_1=-1$. Thus, $\mE$ is rectangular with $4$ rows. In particular, we have $\Jac_{\rho,1}(\sigma_{\mE})=\Jac_{\rho,-1}^{\op}(\sigma_{\mE})=0$, which contradicts to the fact that $\Jac_{\rho,1}\pi\neq0$.
	
	\par If $\mE$ can not be written as a tableau as before, we verify that $\mE$ can be written as a union of a tableau $$\mT:=\begin{ytableau}
		B&\dots&\dots&\dots&A\\
		\vdots&\vdots&\vdots&\vdots&\vdots\\
		x_{2t_1+2}&\dots&\dots&\dots&\scriptstyle D+1\\
		\vdots&\vdots&\vdots&\vdots&\none\\
		x_r&\cdots&y_r&\none&\none
	\end{ytableau}$$ with the set $\left\{1,0\right\}$. Thus we have $$\left\{0,1\right\}\cup[2,A]\cup(\bigcup_{z_s\le0}\left\{-z_s\right\})=[D+1,A]\cup(\bigcup_{z_s>0}\left\{z_s-1\right\}),$$ which forces $\bigcup_{z_s\le0}\left\{-z_s\right\}$ to be empty. Thus, we have $x_{2t_1+2}\ge 1\Rightarrow t_1=0$. This completes the proof.
	
	\printindex
	
	\newpage
	\printbibliography
	
	\vspace{1em}
	\begin{flushleft} \small
		Department of Mathematics, School of Mathematical Sciences, Peking University.\\
		E-mail address: \texttt{chenmoving@stu.pku.edu.cn}
	\end{flushleft}

\end{document}